\documentclass[11pt, reqno]{amsart}

\usepackage{amsfonts,amssymb,amsmath,amsthm, ulem}
\usepackage[headinclude,DIV13]{typearea}
\usepackage[utf8]{inputenc}
\usepackage{hyperref}
\usepackage{geometry}
\usepackage{graphicx, psfrag,enumerate,enumitem}

\usepackage{upgreek}

\usepackage{soul}
\usepackage{cancel}
\usepackage{color}
 \usepackage{setspace}

\newcommand{\tX}{\widetilde{X}}
\newcommand{\tY}{\widetilde{Y}}
\newcommand{\ttau}{\widetilde{\tau}}
\newcommand{\rtheta}{l}

\newcommand{\al}[1]{{\color{black} #1}} 
\newcommand{\cha}[1]{{\color{black} #1}}

\newtheorem{thm}{Theorem}[section]
\newtheorem{prop}[thm]{Proposition}

\newtheorem{cor}[thm]{Corollary}
\newtheorem{lemma}[thm]{Lemma}

\newtheorem{rem}[thm]{Remark}
\numberwithin{equation}{section}

\newtheorem{innercustomass}{Assumption}
\newenvironment{customass}[1]
  {\renewcommand\theinnercustomass{#1}\innercustomass}
  {\endinnercustomass}
  
\theoremstyle{definition} 
\theoremstyle{definition}

\renewcommand{\P}{\mathbb{P}}
\newcommand{\R}{\mathbb{R}}

\newcommand{\E}{\mathbb{E}}

\newcommand{\N}{\mathbb{N}}

\newcommand{\eps}{\varepsilon}

\newcommand{\brho}{\mathfrak b}

\newcommand{\ud}{\mathrm{d}}

\numberwithin{equation}{section}

\begin{document}

\title[Parasite infection]
{Spread of parasites affecting death and division rates in a cell population}

\author{Aline Marguet}
\address{Aline Marguet, Univ. Grenoble Alpes, INRIA, 38000 Grenoble, France}
\email{aline.marguet@inria.fr}

\author{Charline Smadi}
\address{Charline Smadi, Univ. Grenoble Alpes, INRAE, LESSEM, 38000 Grenoble, France
 and Univ. Grenoble Alpes, CNRS, Institut Fourier, 38000 Grenoble, France}
\email{charline.smadi@inrae.fr}

\date{}

\maketitle

\begin{abstract}
We introduce a general class of branching Markov processes for the modelling of a parasite infection in a cell population.
Each cell contains a quantity of parasites which evolves as a diffusion with positive jumps. The drift, diffusive function and positive jump rate of this 
quantity of parasites depend on its current value. The division rate of the cells also depends on the quantity of parasites they contain. At division, 
a cell gives birth to two daughter cells and shares its parasites between them.
Cells may also die, at a rate which may depend on the quantity of parasites they contain.
We study the long-time behaviour of the parasite infection.
 \end{abstract}



\bigskip




\section*{Introduction}
We introduce a general class of continuous-time branching Markov processes for the study of
a parasite infection in a cell population. This framework is general enough to be applied for the modelling of other
structured populations, with individual traits evolving on the set of positive real numbers.
Another application we can think of, similar in spirit, is the modelling of the protein aggregates in a cell population. These latter, usually eliminated by the cells, can undergo sudden increases due to cellular stress (positive jumps), and are known to be distributed unequally between daughter cells (see \cite{rujano2006polarised} for instance).

The dynamic of the quantity of parasites in a cell is given by a Stochastic Differential Equation (SDE) with drift, diffusion and positive jumps.
Then, at a random time whose law may depend on the quantity of parasites in the cell, 
this latter dies or divides. At division, it shares its parasites between its two daughter cells.
We are interested in the long-time behaviour of the parasite infection in the cell population.
More precisely, we will focus on the quantity of parasites in the cells, including the possibility of explosion or extinction of the quantity of parasites. 
We will see that those quantities are very sensitive to the way cell division and death rates depend on the quantity of parasites in the cell, 
and to the law of the sharing of the parasites between the two daughter cells at division.

In discrete time, from the pioneer model of Kimmel \cite{Kimmel}, many studies have been conducted on branching within branching processes to study the host-parasite dynamics: on the associated quasistationary distributions \cite{bansaye2008}, considering random environment and immigration events \cite{bansaye2009}, on multitype branching processes \cite{alsmeyer2013, alsmeyer2016}. In continuous time, host-parasites dynamics have also  been studied using two-level branching processes, in which the dynamics of the parasites is modelled by a birth-death process with interactions \cite{meleard2013evolutive, osorio2020level} or a Feller process \cite{BT11,BPS} and the cell population dynamics by a structured branching process.

Some experiments, conducted in the TAMARA laboratory, have shown that cells distribute unequally their parasites between 
their two daughter cells \cite{stewart2005aging}.
This could be a mechanism aiming at concentrating the parasites in some cell lines in order to `save' the remaining lines.
It is thus important to understand the effect of this unequal sharing on the long-time behaviour of the infection in the cell population.
This question has been addressed by Bansaye and Tran in \cite{BT11}. They introduced and studied branching Feller diffusions with a cell division rate depending on the 
quantity of parasites in the cell and a sharing of parasites at division between the two daughter cells according to a random
variable with any symmetric distribution on $[0,1]$. They provided some extinction criteria for the 
infection in a cell line, in the case where 
the cell division rate is constant or a monotone function of the quantity of parasites in the cell, as well as recovery criteria
at the population level, 
in the constant division rate case. In \cite{BPS}, Bansaye {\it{et al.}} extended this study by providing the long
time asymptotic of the recovery rate in the latter case.
Our work further extends these results in several directions. First, we allow the drift and the diffusion coefficient of the quantity of parasites in a cell to vary with the quantity of 
parasites.
Second, we add the possibility to have positive jumps in the parasites dynamics, with a rate which may depend on the current
quantity of parasites in the cell.
Third, we allow the cell division rate to depend not monotonically on the quantity of parasites.
This situation is more difficult to study than the previous ones, as the genealogical tree of the cell population depends on 
the whole history of the quantity of parasites in the different cell lines.
Finally, we add the possibility for the cells to die at a rate which may depend on the quantity of parasites they contain, which complicates the underlying population process for the cells.

For the study of structured branching populations, a classical method to obtain information on the distribution of a trait in the population is to introduce a 
penalized Markov process, called auxiliary process, corresponding to the dynamics of the trait of a typical individual in the population, {\it i.e.} an individual picked uniformly at random. 
The link between the auxiliary process and the population process is given by Many-to-One formulae. We refer to \cite{georgii2003supercritical,baake2007mutation,hardy2009spine,bansaye2011limit,cloez2017limit,marguet2016uniform, marguet2017law} for general 
results on these topics in continuous time.

Part of our proof strategy consists thus in investigating the long-time behaviour of this auxiliary process  
and deduce properties on the 
asymptotic behaviour of the process at the population level, extending previous results derived for a smaller class of 
structured Markov branching processes (see \cite{BT11,bansaye2011limit,cloez2017limit} for instance). In the case of a constant growth rate for the cell population, the auxiliary process belongs to the class of continuous-state non-linear branching processes that has been studied in \cite{li2017general, companion}.
This case is explored in a companion paper \cite{marguet2023parasite}. In the general case, the auxiliary process is time-inhomogeneous, which makes the study of its asymptotic behaviour much more involved, even out of reach. An alternative approach consists in introducing a Markov process with another penalization, corresponding to the dynamics of the trait of an individual in the population, which is chosen according to some weight, and not picked uniformly at random \cite{cloez2017limit}. When this weight is a positive eigenvector of a given operator, the alternative auxiliary process that we obtain is time-homogeneous and its asymptotic study is easier. Additional work is needed to obtain properties at the population level.

The paper is structured as follows. In Section \ref{section_model}, we define the population process and give assumptions ensuring its existence and uniqueness 
as the strong solution to a SDE. In Section \ref{sec_constant_diff}, we provide \al{sufficient} conditions for the cell population to get rid of parasites or for the quantity of parasites 
to explode in all the cells in the case of the population survival. Then, in Section \ref{sec:proportion}, we explore asymptotics of the proportion of cells with a given level of infection. In particular in Section \ref{sec:LDCD}, we investigate the case of a linear division rate.
We prove that this strategy of division allows the cell population to contain the infection, and we give additional information on the asymptotic distribution of the quantity of parasites in the cells under some assumptions.
Sections \ref{sec:MTO} and \ref{sec:proofs} are dedicated to the proofs.
\\

In the sequel $\N:=\{0,1,2,...\}$ will denote the set of nonnegative integers, $\R_+:=[0,\infty)$ the real line,  $\overline{\R}_+:=\R_+\cup \{ + \infty \}$,
and $\R_+^*:=(0,\infty)$.
We will denote by $\mathcal{C}_{b}^2(\R_+)$ (resp. $\mathcal{C}_{0}^2(\R_+)$) the set of bounded (resp. vanishing at $0$ and infinity) twice continuously differentiable functions on $\R_+$. Finally, for any stochastic 
process $X$ on $\overline{\mathbb{R}}_+$ or $Z$ on the set of point measures on $\overline{\mathbb{R}}_+$, we will denote by 
$\mathbb{E}_x\left[f(X_t)\right]=\E\left[f(X_t)\big|X_0 = x\right]$ and $\mathbb{E}_{\delta_x}\left[f(Z_t)\right]=\E\left[f(Z_t)\big|Z_0 = \delta_x\right]$.

\section{Definition of the population process} \label{section_model}

\subsection{Parasites dynamics in a cell}

Each cell contains parasites whose quantity evolves as a diffusion with positive jumps. More precisely, we consider the SDE
\begin{align} \nonumber \label{X_sans_sauts} \mathfrak{X}_t =x + \int_0^t g(\mathfrak{X}_s)ds
+\int_0^t\sqrt{2\sigma^2 (\mathfrak{X}_s)}dB_s &+
\int_0^t\int_0^{p(\mathfrak{X}_{s^-})}\int_{\mathbb{R}_+}z\widetilde{Q}(ds,dx,dz)\\&+
\int_0^t\int_0^{\mathfrak{X}_{s^-}}\int_{\mathbb{R}_+}zR(ds,dx,dz),
\end{align}
where $x$ is nonnegative, $g$, $\sigma$ and $p$ are real functions on $\overline{\mathbb{R}}_+$, $B$ is a standard Brownian motion,  
$\widetilde{Q}$ is a compensated Poisson point measure (PPM) with intensity 
$ds\otimes dx\otimes \pi(dz)$, $\pi$ is a nonnegative measure on $\mathbb{R}_+$,
$R$ is a PPM with intensity 
$ds\otimes dx\otimes \rho(dz)$, and $\rho$ is a measure on $\mathbb{R}_+$ with density:
\begin{align*} 
\rho(dz)=\frac{c_\brho\brho(\brho+1)}{\Gamma(1-\brho)}\frac{1}{z^{2+\brho}}\ud z,
\end{align*}
where $\brho \in (-1,0)$ and \al{$c_\brho\leq 0$} (see \cite[Section 1.2.6]{kyprianou2006introductory} for details on stable distributions and processes).
Finally, $B$, $Q$ and $R$ are independent.

We will provide below conditions under which the SDE \eqref{X_sans_sauts} has a unique nonnegative strong solution.
In this case, it is a Markov process with infinitesimal generator $\mathcal{G}$, satisfying for all $f\in C_{0}^2(\mathbb{R}_+)$,
\begin{align} \label{def_gene}
\mathcal{G}f(x) = g(x)f'(x)+\sigma^2(x)f''(x)&+p(x)\int_{\mathbb{R}_+}\left(f(x+z)-f(x)-zf'(x)\right)\pi(dz)\\&+x\int_{\mathbb{R}_+}\left(f(x+z)-f(x)\right)\rho(dz)\nonumber ,
\end{align}
and $0$ and $+ \infty$ are two absorbing states.
Following \cite{marguet2016uniform}, we denote by $(\Phi(x,s,t),s\leq t)$ the corresponding stochastic flow {\it i.e.} 
the unique strong solution to \eqref{X_sans_sauts} satisfying $\mathfrak{X}_s = x$ and the dynamics of the trait between division events is well-defined.

\subsection{Cell division}

A cell with a quantity of parasites  
$x$ divides at rate $r(x)$ and is replaced by two 
individuals with quantity of parasites at birth given by $\Theta x$ and $(1-\Theta)x$. Here $\Theta$ is a nonnegative random variable on $(0,1)$ 
with associated distribution $\kappa(d\theta)$ \al{symmetric with respect to $1/2$ (so that $\Theta$ and $1-\Theta$ are identically distributed)} satisfying $\int_0^1 |\ln \theta|\kappa(d\theta)<\infty$.

\subsection{Cell death}

Cells have a death rate $q(x)$ which depends on the quantity of parasites $x$ they carry.
The function $q$ may be nondecreasing, because the presence of parasites may kill the cell, or nonincreasing, if parasites slow down the cellular machinery 
(production of proteins, division, {\it etc.}).

\subsection{Existence and uniqueness}

We use the classical Ulam-Harris-Neveu notation to identify each individual. Let us denote by
$\mathcal{U}:=\bigcup_{n\in\mathbb{N}}\left\{0,1\right\}^{n}$
the set of possible labels, $\mathcal{M}_P(\overline{\mathbb{R}}_+)$ the set of point measures on $\overline{\mathbb{R}}_+$, and 
$\mathbb{D}(\mathbb{R}_+,\mathcal{M}_P(\overline{\mathbb{R}}_+))$, the set of c\`adl\`ag measure-valued processes. 
For any $Z\in \mathbb{D}(\mathbb{R}_+,\mathcal{M}_P(\overline{\mathbb{R}}_+))$, $t \geq 0$, we write 
\begin{equation} \label{Ztdirac}
Z_t = \sum_{u\in V_t}\delta_{X_t^u},
\end{equation}
where $V_t\subset\mathcal{U}$ denotes the set of individuals alive at time $t$ and $X_t^u$ the trait at time $t$ of the individual $u$, \al{following \eqref{X_sans_sauts} between divisions}. \al{By convention, if $V_t=\emptyset$, $Z_t$ is defined as the null measure. For $u\in V_t$, and $s<t$, $X_s^u$ denotes the quantity of parasites in the (only) ancestor of $u$ alive at time $s$.} We denote by $N_{t} : = \mathrm{Card}(V_t)$ the number of cells in the population at time $t$.
Let $E = \mathcal{U}\times(0,1)\times \overline{\mathbb{R}}_+$  and 
$M(ds,du,d\theta,dz)$ be a PPM on $\mathbb{R}_+\times E$ with intensity 
$ds\otimes n(du)\otimes \kappa(d\theta)\otimes dz$, where $n(du)$ denotes the counting measure on $\mathcal{U}$. 
Let $\left(\Phi^u(x,s,t),u\in\mathcal{U},x\in\overline{\mathbb{\R}}_+, s\leq t\right)$ be a family of independent stochastic flows satisfying \eqref{X_sans_sauts} 
describing the individual dynamics.
We assume that $M$ and $\left(\Phi^u,u\in\mathcal{U}\right)$ are independent. We denote by $\mathcal{F}_t$ the filtration generated by the restriction of the PPM $M$ to $[0,t]\times E$ and 
the family of stochastic processes $(\Phi^u(x,s, t), u\in\mathcal{U},x\in\overline{\mathbb{R}}_+,s\leq t)$ up to time $t$. \\

We now consider assumptions to ensure the strong existence and uniqueness of the process.
We obtain a large class of branching Markov processes for the modelling of parasite infection in a cell population. Points $i)$ to $iii)$ of Assumption \ref{ass_A}.
(Existence and Uniqueness) ensure that 
the dynamics in a cell line is well-defined \cite{palau2018branching}
(as the unique nonnegative strong solution to the SDE \eqref{X_sans_sauts} up to explosion, and with infinite value after explosion); 
points $iv)$ and $v)$ ensure the non-explosion of the cell population size in finite time, as in \cite{marguet2016uniform}.
\begin{customass}{\bf{EU}} \label{ass_A}
We assume that
\begin{enumerate}[label=\roman*)]
\item The functions $r$ and $p$ are locally Lipschitz on $\R_+$, $p$ is non-decreasing and $p(0)= 0$. The function $g$ is continuous on $\R_+$, $g(0)=0$ and for any $n \in \N$ there exists a finite constant $B_n$ such that for any $0 \leq x \leq y \leq n$
\begin{align*} |g(y)-g(x)|
\leq B_n \phi(y-x),
\ 
\text{ where }\ 
\phi(x) = \left\lbrace\begin{array}{ll}
x \left(1-\ln x\right) & \textrm{if } x\leq 1,\\
1 & \textrm{if } x>1.\\
\end{array}\right.
\end{align*}
\item The function $\sigma$ is H\"older continuous with index $1/2$ on compact sets and $\sigma(0)=0$. 
\item The measure $\pi$ satisfies
$ \int_0^\infty \left(z \wedge z^2\right)\pi(dz)<\infty. $
\item There exist $r_1,r_2\geq 0$ and $\gamma\geq 0$ such that for all $x\geq 0$, $
r(x)-q(x)\leq r_1x^{\gamma}+r_2.$
\item There exist $c_1,c_2\geq 0$ such that, for all $x\in\mathbb{R}_+$,
\begin{align*}
\lim_{n\rightarrow+\infty}\mathcal{G}h_{n,\gamma}(x)\leq c_1 x^{\gamma}+c_2,
\end{align*}
where $\mathcal{G}$ is defined in \eqref{def_gene}, $\gamma$ has been defined in iv) and $h_{n,\gamma}\in C_b^2(\mathbb{R}_+)$ is a 
sequence of functions such that $\lim_{n\rightarrow+\infty} h_{n,\gamma}(x)=x^{\gamma}$ for all $x\in\mathbb{R}_+$.
\end{enumerate}
\end{customass}
Then the structured population process may be defined as the strong solution to a SDE. 
\begin{prop} \label{pro_exi_uni}
Under Assumption \ref{ass_A} there exists a strongly unique $\mathcal{F}_t$-adapted 
c\`adl\`ag process $(Z_t,t\geq 0)$ taking values in $\mathcal{M}_P(\overline{\mathbb{R}}_+)$ such that for all $f\in C_0^2(\overline{\mathbb{R}}_+)$ and $x_0,t\geq 0$, 
\begin{align*}\nonumber
\langle Z_{t},f\rangle = &f\left(x_{0}\right)+\int_{0}^{t}\int_{\mathbb{R}_+}\mathcal{G}f(x)Z_{s}\left(dx\right)ds+M_{t}^f(x_0)&\\\nonumber
 +\int_{0}^{t}\int_{E}&\mathbf{1}_{\left\{u\in V_{s^{-}}\right\}}\left(\mathbf{1}_{\left\{\ z\leq r(X_{s^{-}}^u)\right\}}\left(f\left(\theta X_{s^-}^u \right)+ f\left((1-\theta) X_{s^-}^u \right)-
f\left(X_{s^{-}}^{u}\right)\right)\right.\\
& \left.-\mathbf{1}_{\left\{0<z-r(X_{s^{-}}^u)\leq q(X_{s^{-}}^u)\right\}}
f\left(X_{s^{-}}^{u}\right)\right) M\left(ds,du,d\theta,dz\right),
\end{align*}
where $\mathcal{G}$ is defined in \eqref{def_gene} and for all $x\geq 0$, $M_{t}^f(x)$ is a $\mathcal{F}_t$-martingale.
\end{prop}
The proof is a combination of \cite[Proposition 1]{palau2018branching} and \cite[Theorem 2.1]{marguet2016uniform} (see Appendix \ref{proof_ex_uni} for details). 
Note that we replaced (1) and (3) of Assumption A in \cite{marguet2016uniform} by {\it iv)} in Assumption \ref{ass_A}. A careful look at 
the proof of \cite[Theorem 2.1]{marguet2016uniform} (in particular (2.5) in \cite[Lemma 2.5]{marguet2016uniform}) shows that in our case, the growth of the 
population is governed by the function $x\mapsto r(x)-q(x)$ so that our condition is sufficient. Note also that in \cite{marguet2016uniform} the exponent 
$\gamma$ of {\it iv)} in Assumption \ref{ass_A} is required to be greater than $1$ but this condition is not necessary for conservative fragmentation 
processes as considered here. 
In what follows, we will assume that all the processes under consideration satisfy Assumption \ref{ass_A}, but we will not indicate it.\\

\section{Containment or explosion of the infection} \label{sec_constant_diff}

In this section, we consider general cell division and death rates, and we look for sufficient conditions for the quantity of parasites to 
become large (resp. small) in every alive cell. We exhibit a function characterizing the parasites growth rate in a typical cell and 
compare it to the growth rate of the cell population. 

For $a \in \R_+ \setminus \{1\}$, we introduce when it is well-defined the function $G_a$, for $x>0$ via
\begin{align}\label{eq:Ga}
G_a(x) : = (a-1)\left[\frac{g(x)}{x}-a\frac{\sigma^2(x)}{x^2} + x^{-\brho}C_a-p(x)I_a(x)-2r(x)\frac{\E[\Theta^{1-a}]-1}{a-1}\right],
\end{align}
where we recall that $\Theta$ is a random variable with distribution $\kappa$ and use the notations
\begin{align}\label{eq:Ia}
C_{a}=\frac{1}{2+\mathfrak{b}}\int_0^\infty\frac{z\rho(dz)}{(1+z)^{a}},\quad 
I_a(x) = \frac{\left(a + \mathbf{1}_{\{a=0\}}\right)}{x^{2}}\int_0^\infty \left(\int_0^1\frac{z^2(1-v)dv}{(1 + zvx^{-1})^{1+a}}\right)\pi(dz).
\end{align}
From \cite[Lemma 7.1]{companion} (see Lemma \ref{prop:martingale-co} in our case), if $Y$ denotes the auxiliary  process describing the 
behaviour of a `typical individual',
then the process $$(Y_t^{1-a}e^{\int_0^t G_a(Y_s)ds}, t \geq 0)$$ is a local martingale, which entails that, $Y_t^{1-a}$ roughly behaves as 
$e^{-\int_0^tG_a(Y_s)ds}$ and thus, $G_a$ contains informations on the dynamics of the quantity of parasites in a typical cell \cha{(see Section \ref{sec:MTO} for explanations and computations on auxiliary processes and Many-to-One formulae)}. 
Moreover, the growth rate of a cell population with a constant quantity $x$ of parasites is $r(x)-q(x)$. The next assumptions combine conditions on 
those two key quantities, leading to results on the asymptotic behaviour of the infection in the entire population. 
\begin{customass}{\bf{EXPL}} \label{ass_expl_tps_fini}
 There exist $a>1$ such that $\E[\Theta^{1-a}]<\infty$ and $0\leq \gamma <\gamma'$ such that 
 $$r(x)-q(x) \leq  \gamma  <\gamma' \leq G_{a}(x), \quad \forall x \geq 0. $$
\end{customass}
\begin{customass}{\bf{EXT}} \label{ass_abs_tps_fini}
There  exist $a<1$ and $0\leq \gamma <\gamma'$ such that 
$$r(x)-q(x) \leq  \gamma  <\gamma' \leq G_{a}(x), \quad \forall x \geq 0. $$
\end{customass}
 Note that for $a>1$, the condition $G_a(x)\geq \gamma'$ in Assumption \ref{ass_expl_tps_fini} is equivalent to 
\begin{align*}
 \frac{g(x)}{x}-a\frac{\sigma^2(x)}{x^2} + x^{-\brho}C_a-p(x)I_a(x)-2r(x)\frac{\E[\Theta^{1-a}-1]}{a-1}\geq \al{\frac{\gamma'}{a-1}},
\end{align*}
for all $x\geq 0$, which can be interpreted as the growth of the parasites being stronger than the growth of the cell population and the noise. On the contrary, for $a<1$, the condition $G_a(x)\geq \gamma'$ of Assumption \ref{ass_abs_tps_fini} is equivalent to 
\begin{align*}
 \frac{g(x)}{x}-a\frac{\sigma^2(x)}{x^2} + x^{-\brho}C_a-p(x)I_a(x)-2r(x)\frac{\E[\Theta^{1-a}-1]}{a-1}\leq \al{-\frac{\gamma'}{1-a}}.
\end{align*}
In this case, the cell population growth and the noise outweigh the growth of the parasites.
\begin{prop} \label{prop_constant_diff}
Let $K>0$. Then for every $x>0$,
\begin{enumerate}[label=(\roman*)]
  \item Under Assumption \ref{ass_expl_tps_fini},
  $ \lim_{t \to \infty} \P_{\delta_x}\left( \exists u \in V_t, X_t^u \leq K \right) =0. $
  \item Under Assumption \ref{ass_abs_tps_fini},
  $ \lim_{t \to \infty} \P_{\delta_x}\left( \exists u \in V_t, X_t^u \geq K \right) =0. $
\end{enumerate}
\end{prop}
Thus, \al{if the cell population survives with a positive probability, then, conditionally on survival, the quantity of parasites goes to infinity in all cells in case {\it (i)}, and
in case {\it (ii)}, the quantity of parasites goes to zero in all cells with a probability close to one. The study of the survival of the cell population for general dynamics is complex. However, if there exists $c>0$ such that $r(x)-q(x)>c$ for all $x>0$, we can prove that the probability of survival of the cell population is positive.}

\cha{The assumptions of Proposition \ref{prop_constant_diff} may seem strong \al{as global inequalities (for all $x$)} are required. But the strength of this result lies \al{in its generality: it allows a large degree of freedom in both }the way cell division and death rates depend on the quantity of parasites they contain, as well as in the rate of positive jumps in the dynamics of parasites inside the cells. In particular, no monotonicity or concavity/convexity \al{assumption} is required.}

\cha{To illustrate this result, we consider the following example for $z \in \R_+$, $\theta \in (0,1)$:
$$\pi(dz) = \alpha_{\pi}z^{-1-\beta_{\pi}}dz, \quad \kappa(d\theta)=\delta_{1/2}(d\theta), \quad 
\sigma^2(z)=\alpha_{\sigma}z^2 +\beta_{\sigma}z, \quad g(z)= \alpha_g z, $$
with $\alpha_g,\alpha_{\sigma},\beta_{\sigma}, \alpha_{\pi}\in \mathbb{R}_+$ and $\beta_{\pi}\in (1,2)$.
Notice that using the change of variable $dz= x dy$ we get that $I_a(x)$ defined in \eqref{eq:Ia} may be rewritten
$$ I_a(x) = \alpha_\pi\frac{\left(a + \mathbf{1}_{\{a=0\}}\right)}{x^{\beta_\pi}}\int_0^\infty \left(\int_0^1\frac{y^{1-\beta_\pi}(1-v)dv}{(1 + yv)^{1+a}}\right)dy =: \al{\frac{\mathfrak{I}^{(\pi)}_a}{x^{\beta_\pi}}}. $$
\al{For this example}, Assumption \ref{ass_expl_tps_fini} writes: \al{there exists $a>1$ and $0<\gamma<\gamma'$ such that}
$$ q(x) \geq  r(x)-\gamma \quad \text{and} \quad r(x) \leq \frac{a-1}{2^a - 2} \left[ \alpha_g +x^{-\brho}C_a - \frac{\gamma'}{ a-1}- \mathfrak{I}^{(\pi)}_a\frac{p(x)}{x^{\beta_\pi}}- a (\alpha_\sigma + \frac{\beta_\sigma}{x}) \right].  $$
As $r(0)\geq 0$, we necessarily have $\beta_\sigma=0$. Hence, for this example,
$$ \text{\ref{ass_expl_tps_fini}} \Leftrightarrow \exists a>1, 0<\gamma<\gamma', \forall x \geq 0, \left\{\begin{array}{l} 
\beta_\sigma=0,\\
q(x) \geq  r(x)-\gamma,\\
r(x) \leq \frac{a-1}{2^a - 2} \left[ \alpha_g +\frac{C_a}{x^{\brho}} - \frac{\gamma'}{ a-1}- \mathfrak{I}^{(\pi)}_a\frac{p(x)}{x^{\beta_\pi}}- a \alpha_\sigma \right].
\end{array}\right. $$
In particular, this implies that $\alpha_g \geq \gamma'/(a-1)+ a \alpha_\sigma$, and thus the growth of parasites has to be strong enough to compensate the population growth and the fluctuations. The freedom in the choice of the function $r(\cdot)$ then depends on the equilibrium between the intensity of stable jumps given by the term $x^{-\brho}C_a$ and the fluctuations due to the (non stable) positive jumps, given by the term $ \mathfrak{I}^{(\pi)}_a\frac{p(x)}{x^{\beta_\pi}}$. Similar computations for the other case give:
$$ \text{\ref{ass_abs_tps_fini}} \Leftrightarrow \exists a<1, 0<\gamma<\gamma', \forall x \geq 0, \left\{\begin{array}{l} 
q(x) \geq  r(x)-\gamma,\\
r(x) \geq \frac{1-a}{2-2^a} \left[\frac{\gamma'}{1-a}+ \alpha_g +\frac{C_a}{x^{\brho}} - \mathfrak{I}^{(\pi)}_a\frac{p(x)}{x^{\beta_\pi}}- a \al{\left(\alpha_\sigma +\frac{\beta_\sigma}{x}\right)} \right].
\end{array}\right. $$
which can be analysed in a similar fashion.}\\

Proposition \ref{prop_constant_diff} applies to a large class of dynamics for the growth of parasites within cells, but imposes a limited \cha{population} growth (bounded $r-q$). In the next two sections, we focus on more general cases for \cha{population} growth, but we require more assumptions on parasite dynamics. \al{To ease the comparison with the results of the next sections, we give a straightforward corollary on the convergence of the proportion of infected cells. 
\begin{cor}\label{cor:convergence-proportions}
Let $K>0$. Then, for every $x>0$,
\begin{enumerate}[label=(\roman*)]
  \item Under Assumption \ref{ass_expl_tps_fini},
  $$ \lim_{t \to \infty} \E_{\delta_x}\left[ \mathbf{1}_{N_t\geq 1}\frac{\sum_{u \in V_t}\mathbf{1}_{\{ X_t^u \leq K \}}}{N_t}\right] =0. $$
  \item Under Assumption \ref{ass_abs_tps_fini},
  \begin{equation} \label{res_run_1} \lim_{t \to \infty} \E_{\delta_x}\left[\mathbf{1}_{N_t\geq 1} \frac{\sum_{u \in V_t}\mathbf{1}_{\{ X_t^u \geq K \}}}{N_t}\right] =0. \end{equation}
\end{enumerate}
\end{cor}
Point $(i)$ follows from Proposition \ref{prop_constant_diff} and the inequalities:
$$ \E_{\delta_x}\left[ \mathbf{1}_{N_t\geq 1}\frac{\sum_{u \in V_t}\mathbf{1}_{\{ X_t^u \leq K \}}}{N_t}\right] \leq \E_{\delta_x}\left[ \mathbf{1}_{N_t\geq 1}\mathbf{1}_{\exists u \in V_t, X_t^u \leq K}\right]\leq \P_{\delta_x}\left( \exists u \in V_t, X_t^u \leq K \right). $$
The proof of point $(ii)$ is similar.}

\section{Asymptotics of the proportion of cells with a given level of infection}\label{sec:proportion}

\subsection{Running example}

\cha{Before presenting the precise assumptions and results of this section, we introduce a class of possible \al{dynamics} for the parasites and the \al{cell population}. This running example will allow a simpler interpretation of the results as well as comparisons of assumptions and results presented in the different subsections and in the previous section.}

\cha{Let $\alpha_g,\alpha_{\sigma},\beta_{\sigma}\in \mathbb{R}_+$. We will consider a Feller diffusion in a Brownian environment:
\begin{align} \label{Feller_diff_jumps} \mathfrak{X}_t =&\mathfrak{X}_0 + \int_0^t \alpha_g\mathfrak{X}_sds
+\int_0^t\sqrt{2\left(\alpha_{\sigma}\mathfrak{X}_s^2 +\beta_{\sigma}\mathfrak{X}_s\right)}dB_s,
\end{align}
where $\mathfrak{X}_0 \in \R_+$, $B$ has been defined in \eqref{X_sans_sauts}.
On top of the parasites dynamics, we consider the framework of a linear division and death rate for the cell population and symmetric partitioning: for $\alpha,\beta,\alpha_q,\beta_q\in\mathbb{R}_+$, and for $x \geq 0$, let
\begin{align}\label{eq:def_example_cells}
r(x) = \alpha x + \beta,\quad q(x) = \alpha_q x+ \beta_q,\quad \kappa(d\theta)=\delta_{1/2}(d\theta).
\end{align}
For this example, Assumption \ref{ass_A} holds and the structured branching process is well-defined.
}

\subsection{The role of noise}\label{sec:roleofnoise}

We consider a simple case where the parasites follow a Geometric Brownian motion. This is a particular case of CSBP in a L\'evy environment. The case of the Brownian environment, specifically studied in \cite{boeinghoff2011branching,palau2017continuous}, can be seen as an agitation or disturbance due to the cell environment. The cells might, for example, tune the intensity of this perturbation (quantified by the parameter $\sigma$) by modifying the temperature or the fluidity of the cell medium. The assumptions of this first case are summarized below.
\begin{customass}{\bf{LGBE}}{(Linear growth, Brownian environment)}\label{ass:betar-q}
We assume that
\begin{itemize}
\item There are no jumps in the parasites dynamics ($p=\mathfrak{c}_\mathfrak{b}= 0$).
\item There exist $g,\sigma\in\mathbb{R}_+$ such that $g(x) = gx$ and $\sigma^2(x) = \sigma^2x^2$. 
\item There exist $\gamma\in [0,+\infty)$ and $\al{\mathfrak{c}\in \R}$ such that $\gamma r(x) -q(x) = \mathfrak{c}$ for all $x\geq 0$. 
\end{itemize}
\end{customass}
Under Assumption \ref{ass:betar-q}, we provide conditions under which the mean number of cells with a small (resp. a large) quantity of parasites is equivalent to the mean number of cells alive in the population. \al{Let us define
$$
\alpha_0:=\inf\left\{\alpha<0\text{ s.t. }\E[\Theta^{\alpha}]<\infty\right\},
$$ 
with the convention that $\alpha_0=-\infty$ if $\E[\Theta^{\alpha}]<\infty$ for all $\alpha<0$. Then, for each $\gamma\geq 0$, let $\alpha_\gamma> \alpha_0$ to be such that 
$
2\E\left[\Theta^{\al{\alpha_\gamma}}\right] - 1 = \al{\gamma},
$
which is well-defined because $\psi:\alpha\in (\alpha_0,1)\mapsto  2\E\left[\Theta^\alpha\right] - 1$ is decreasing, and $\psi((\alpha_0,1]) = [0,+\infty)$.}
\begin{prop} \label{prop:betar-q} Let $\gamma\in [0,+\infty)$ be such that Assumption \ref{ass:betar-q} holds and \al{recall that $\alpha_{\gamma}>\alpha_0$} is such that $2\E\left[\Theta^{\alpha_\gamma}\right] - 1 = \gamma$. Then,
\begin{enumerate}[label=\roman*)]
\item\label{point1-propbetar-q} if $\gamma\in [0,1)$ (then $\alpha_\gamma\in (0,1]$) and $\alpha_\gamma<1-g/\sigma^2$, for all $K>0$ and $t \geq 0$,
\begin{align*}
\frac{\E_{\delta_x}\left[\sum_{u\in V_t} \mathbf{1}_{\lbrace X_t^u>K\rbrace}\right]}{\E_{\delta_x}\left[N_t\right]} \cha{ \leq  \left(\frac{x}{K}\right)^{\al{\alpha_\gamma}}e^{\alpha_\gamma(g + (\alpha_\gamma -1)\sigma^2)t} }.
\end{align*}
\item\label{point2-propbetar-q} if $\gamma\in [0,1)$ (then $\alpha_\gamma\in (0,1]$) and there exist $x_0(\alpha_\gamma)\geq 0$ and $\eta>0$ such that 
\begin{equation} \label{star1} g\leq \sigma^2(1-2\alpha_\gamma) + 2\E\left[\Theta^{\alpha_\gamma}\ln(1/\Theta)\right]r(x)-\eta,\quad \forall x\geq x_0(\alpha_\gamma),\end{equation} then for all $\eps,x>0$ there exists $K_{\eps,x}$ such that $K\geq K_{\eps,x}$ implies
\begin{equation} \label{res_run_2}
\limsup_{t \to \infty}\frac{\E_{\delta_x}\left[\sum_{u\in V_t} \mathbf{1}_{\lbrace X_t^u\geq K\rbrace}\right]}{\E_{\delta_x}\left[N_t\right]}\leq \eps.
\end{equation}
\item\label{point3-propbetar-q} if $\gamma>1$ (then $\alpha_\gamma<0$) and there exist $x_0(\alpha_\gamma)\geq 0$ and $\eta>0$ such that
\begin{equation} \label{star2} 
g \geq \sigma^2(1+2|\alpha_\gamma|) + 2\E\left[\Theta^{\alpha_\gamma}\ln(1/\Theta)\right]r(x)+\eta,\quad \forall x\leq x_0(\alpha_\gamma) ,
\end{equation}
 then for all $\eps,x>0$ there exists $K_{\eps,x}$ such that $K\leq K_{\eps,x}$ implies
\begin{align*}
\limsup_{t \to \infty}\frac{\E_{\delta_x}\left[\sum_{u\in V_t} \mathbf{1}_{\lbrace X_t^u\leq K\rbrace}\right]}{\E_{\delta_x}\left[N_t\right]}\leq \eps.
\end{align*}
\end{enumerate}
\end{prop}
This result shows the importance of noise in the process dynamics for the long-term behaviour of the infection. It implies that if the cells induce more fluctuations in the dynamics of the parasites while maintaining the same rate of growth for the parasites, the cell population may contain the infection.
The first point of Proposition \ref{prop:betar-q} corresponds to the case of a growing cell population ($r(x)-q(x)\geq 0$, if $\mathfrak{c}\geq 0$), with a moderate growth of the parasites compared to the noise of the environment $(g<\sigma^2)$. The condition $\alpha_\gamma<1-{g}/{\sigma^2}$ links the growth of the population and the growth of the parasites: if ${g}/{\sigma^2}$ is small, there is no additional restriction on the growth of the cell population (except that $\gamma\in[0,1)$), but if ${g}/{\sigma^2}$ is close to $1$, then $\al{\gamma}$ has to be close to $1$ for the condition to be fulfilled, which corresponds to a growth of the cell population independent of the quantity of parasites ($(r-q)(x)$ close to $ \mathfrak{c}$).
\cha{Notice again that if there exists $c > 0$ such that $r(x) - q(x) > c$ for all $x > 0$, we can prove that the probability of survival of the
cell population is positive. This is in particular the case in points \textit{(i)-(ii)} as soon as $\mathfrak{c}> 0$. }
\begin{rem}
The case $\gamma=1$ ($\alpha_\gamma=0$) corresponds to a constant $r-q$, and is studied in details and in more generality in \cite{marguet2023parasite}. In our case, if $r$ and $q$ are constant, for the second and third points, it amounts to look at the sign of
$$ g- \sigma^2- 2r\E\left[\ln(1/\Theta)\right] ,$$
and completes \cite[Corollary 16]{BPS} for the case of Geometric Brownian motion. 
\end{rem}
\cha{{\bf Running example.} 
Let us illustrate the conditions of Proposition \ref{prop:betar-q} using the running example \eqref{Feller_diff_jumps}-\eqref{eq:def_example_cells}.
To simplify the presentation, we will distinguish the cases $\alpha>0$ and $\alpha=0$, which correspond to very different situations. Indeed, the cell division rate depends on the intracellular quantity of parasites only when $\alpha>0$.
\begin{itemize}
\item[$\bullet$] $\alpha>0$; {\bf{LGBE}} $\Leftrightarrow$ $\beta_\sigma=0$ and $\gamma \alpha = \alpha_q$.
\begin{itemize}
\item[$\bullet$] if $\gamma\in [0,1)$ (corresponding to $\alpha_q< \alpha$)
\begin{itemize}
\item  if $\alpha_\sigma\ln(1 +\gamma)>\alpha_g\ln 2$, Proposition \ref{prop:betar-q}\ref{point1-propbetar-q} applies: the proportion of infected cells tends to $0$ for any level of infection.
\item if $\alpha_\sigma\ln(1 +\gamma)\leq \alpha_g\ln 2$, the parasites spread quickly, but condition \eqref{star1} always holds: the cell population contains the infection (Proposition \ref{prop:betar-q}\ref{point2-propbetar-q}). 
\end{itemize} 
\item[$\bullet$] if $\gamma>1$ (corresponding to $\alpha_q> \alpha$), Proposition \ref{prop:betar-q}\ref{point3-propbetar-q} applies if
$$
\alpha_g> \alpha_\sigma\left(2\frac{\ln(1 + \gamma)}{\ln 2}-1\right) + \beta(1+ \gamma) \ln 2;
$$
hence the parasite growth outweighs other mechanisms, and cells typically contain a large quantity of parasites in the long run.
\end{itemize}
\item[$\bullet$] $\alpha =0$; {\bf{LGBE}} $\Leftrightarrow$ $\beta_\sigma=\alpha_q=0$. 
\al{\begin{itemize}
\item if $\alpha_\sigma >\alpha_g$, Proposition \ref{prop:betar-q}\ref{point1-propbetar-q} applies,
\item if $\alpha_\sigma\leq \alpha_g< \alpha_\sigma  +   2\ln 2\beta$, Proposition \ref{prop:betar-q}\ref{point2-propbetar-q} applies, 
\item if $\alpha_g> \alpha_\sigma + 2\beta \ln 2$, Proposition \ref{prop:betar-q}\ref{point3-propbetar-q} applies.
\end{itemize}}
\end{itemize}}
\cha{Notice that} \al{if $\alpha=0$, and $\alpha_\sigma=0$, the last condition reads $\alpha_g> 2\beta \ln 2$, which is the condition for Assumption \ref{ass_expl_tps_fini} to hold, and Proposition \ref{prop_constant_diff}{\it i)} applies.
}
\subsection{Competition between division and parasites growth}\label{sec:competition}
We consider another case where the mean growth of parasites is linked with the cell death rate (see Assumption \ref{ass:gx-q} below). It may be the case when parasites kill the cell to be extruded in the cell medium. 
\begin{customass}{\bf{PGCD}}{(Parasite Growth, Cell Death)}\label{ass:gx-q}
We assume that
\begin{itemize}
\item There are no stable jumps in the parasites dynamic ($c_\mathfrak{b}=0$).
\item There exists \al{$\mathfrak{c}\in\R$} such that $g(x)/x -q(x) = \mathfrak{c}$ for all $x\geq 0$. 
\end{itemize}
\end{customass}
To state the next results, we introduce generalized conditions from \cite{companion} under which the quantity of parasites 
reach the state $\infty$ ({\bf{(SN$\infty$)}} for strong noise at $\infty$). Recall the definition of $C_a$ and $I_a$ in \eqref{eq:Ia}. 
\begin{enumerate}[label=\bf{(SN$\infty$)}]
\item\label{eq:SNinfinity-1}  There exist $a<1$ and $f:\R_+\to \R_+$ such that
\begin{equation*}
\frac{g(u)}{u}- a\frac{\sigma^2(u)}{u^2} +u^{-\brho}C_a-r(u)\frac{1-\E[\Theta^{1-a}]}{1-a} -p(u)I_a(u) = -f(u)+o(\ln(u)), \quad u \to \infty.
\end{equation*}
\end{enumerate}
Under Assumption \ref{ass:gx-q}, $C_a=0$, but we will use later the condition including $C_a$.
\begin{prop} \label{prop:gx-q} Let $\mathfrak{c}\in\R$ be such that Assumption \ref{ass:gx-q} holds. Then,
\begin{enumerate}[label=\roman*)]
\item\label{point1-propgx-q} if 
\begin{equation} \label{add_ass} \limsup_{x \to \infty}\left\{\frac{\sigma^2(x)}{x^4}+ \frac{r(x)}{x^2}+\frac{p(x)}{x^3} \right\}<\infty,
 \end{equation} 
and if there exists $x_0> 0$, $\eta_0>0$ such that
\begin{equation} \label{star4}
\frac{g(x)}{x} - \frac{\sigma^2(x)}{x^2} - \frac{p(x)}{2x^2}\int_{\R_+}\frac{z^2}{1+z/x}\pi(dz) - r(x)\left(\E\left[1/\Theta\right]-1/2\right) \geq \eta_0,\quad \forall x\leq x_0,
\end{equation}
\begin{equation}\label{eq:condinfty}
\frac{g(x)}{x} + \frac{\sigma^2(x)}{x^2} + \frac{p(x)}{x^2}\int_{\R_+}z^2\pi(dz) + 2r(x)\left(\E\left[\Theta^2\right]-1/2\right) \leq -\eta_0,\quad \forall x\geq 1/x_0,
\end{equation}
then for all $\eps,x>0$ there exists $K_{\eps,x}$ such that $K\leq K_{\eps,x}$ implies
\begin{align*}
\limsup_{t \to \infty}\frac{\E_{\delta_x}\left[\sum_{u\in V_t} \mathbf{1}_{\lbrace X_t^u\leq K\rbrace}\right]}{\E_{\delta_x}\left[N_t\right]}\leq \eps.
\end{align*}
\item\label{point2-propgx-q} if \ref{eq:SNinfinity-1} holds and if there exist $x_0\geq 0$, $\eta>0$ such that
\begin{equation} \label{star3}
\frac{g(x)}{x} + \frac{\sigma^2(x)}{x^2} + p(x)I_0(x) + 2\E\left[\Theta\ln\Theta\right]r(x) \leq -\eta,\quad \forall x\geq x_0,
\end{equation}
 then for all $\eps,x>0$ there exists $K_{\eps,x}$ such that $K\geq K_{\eps,x}$ implies
\begin{align*}
\limsup_{t \to \infty}\frac{\E_{\delta_x}\left[\sum_{u\in V_t} \mathbf{1}_{\lbrace X_t^u\geq K\rbrace}\right]}{\E_{\delta_x}\left[N_t\right]}\leq \eps.
\end{align*}
\end{enumerate}
\end{prop}
\begin{rem}
 The class of processes considered here includes, in particular, the case $q(\cdot) \equiv q$, $g(x)=g x$ studied in previous works. Furthermore, assumptions on the cell division rate $r(\cdot)$, and on the noise $\sigma(\cdot)$ in the parasite dynamics are only on their local behaviour (around $0$ and $\infty$). This constitutes a strong generalisation of previous results.
\end{rem}
\begin{rem}
The results of Proposition \ref{prop:gx-q} are weaker than the ones of Proposition \ref{prop_constant_diff} as they concern the normalized mean numbers of cells with given amounts of parasites, and not the proportions of such cells in the population. 
But the assumptions of Proposition \ref{prop:gx-q} are much more flexible than the assumptions of Proposition \ref{prop_constant_diff}: they are in some way their local versions.  
\end{rem}
A consequence of Proposition \ref{prop:gx-q} is that if the cells are able to increase their rate of division above an explicit threshold, this is sufficient to contain the infection. In the same direction, if the parasites manage to slow down the rate of cell division sufficiently, this gives them enough time to multiply and be numerous in a large proportion of the cells.\\

\al{{\bf Running example.} We consider again the case of parasites following \eqref{Feller_diff_jumps} with cell population dynamic given by \eqref{eq:def_example_cells}. Assumption \ref{ass:gx-q} holds if $\alpha_q = 0$. In this case (see \cite[Section 2.2.3]{marguet2016uniform}) 
$$
\E_{\delta_x}[N_t]=e^{(\beta-\beta_q)t} + \frac{\alpha x}{\alpha_g-\beta}\left(e^{(\alpha_g-\beta_q)t} - e^{(\beta-\beta_q)t}\right).
$$
\begin{itemize}
\item if $\alpha>0$, 
\begin{itemize}
\item \ref{eq:SNinfinity-1} and \eqref{star3} hold so that Proposition \ref{prop:gx-q}\ref{point2-propgx-q} applies. In particular, if $\beta_\sigma=0$, we retrieve the result of Proposition \ref{prop:betar-q}\ref{point2-propbetar-q} in the case $\gamma=0$;
\item \eqref{add_ass} and \eqref{eq:condinfty} hold for every set of parameters and \eqref{star4} holds if $\beta_\sigma=0$ and $\alpha_g >\alpha_\sigma +3\beta/2$. In this case, Proposition \ref{prop:gx-q}\ref{point1-propgx-q} applies.
\end{itemize} 
Therefore, if \cha{$\beta_\sigma=0$ and} $\alpha_g >\alpha_\sigma +3\beta/2$, there are no cells with very small or very large amount of parasites in the longtime asymptotic.

Proposition \ref{prop:gx-q} states that if $\alpha>0$ and $\alpha_q=0$, we might have a basal level of infection (Proposition \ref{prop:gx-q}\ref{point1-propgx-q}) if the growth of the parasites is sufficiently strong ($\alpha_g >\alpha_\sigma +3\beta/2$) whereas using Proposition \ref{prop:betar-q}\ref{point2-propbetar-q}, for $\alpha>0$ and $\alpha_q=0$, we could only give condition leading to a moderate infection.
\item if $\alpha=0$, \eqref{star4} and \eqref{eq:condinfty} cannot hold simultaneously but as \ref{eq:SNinfinity-1} always holds and \eqref{star3} holds if $\alpha_g<\beta\ln 2-\alpha_\sigma$, in this case Proposition \ref{prop:gx-q}\ref{point2-propgx-q} applies and the proportion of very infected cells goes to $0$.
\end{itemize} }

\al{By applying Propositions \ref{prop_constant_diff}, \ref{prop:betar-q}, \ref{prop:gx-q} to the running example, we exhibited cases where one result is a generalization of a result above. But this is not always the case.
\cha{We can find cases which satisfy the conditions of only one of the propositions. To illustrate this point, we will focus on the case $\alpha>0$ and on the property \eqref{res_run_2}.}
To prove that \eqref{res_run_1} implies \eqref{res_run_2} in case ex-\ref{prop_constant_diff}ii) below, we refer to the adaptation of Lemma \ref{lem_esp_Nt2} (see p.\pageref{par:adaptationNt}) and to the inequality
$$ \frac{\E_{\delta_x}\left[\sum_{u\in V_t} \mathbf{1}_{\lbrace X_t^u\geq K\rbrace}\right]}{\E_{\delta_x}\left[N_t\right]}
\leq \sqrt{\E_{\delta_x}\left[\frac{\sum_{u\in V_t} \mathbf{1}_{\lbrace X_t^u\geq K\rbrace}}{N_t}\right]\E_{\delta_x}\left[\frac{N^2_t}{\E^2_{\delta_x}\left[N_t\right]}\right]}.$$ }

\al{
Let us consider the following conditions: for $\gamma'>\gamma>0$ and $a<1$,
\begin{itemize}
\item[ex-\ref{prop_constant_diff}ii):\hspace{.23cm}] $ \alpha_q\geq \alpha>0, \ \beta_\sigma>0, \ \beta_q +\gamma > \beta \geq \max\left(\left(\frac{1-a}{2-2^a} \left[\frac{\gamma'}{1-a}+ \alpha_g  -  a \alpha_\sigma \right]\right), \alpha_g, \beta_q\right)$.
\item[ex-\ref{prop:betar-q}i-ii):] $\alpha > \alpha_q>0,\ \beta_\sigma=0, \ \alpha_g<\beta_q,$
\item[ex-\ref{prop:gx-q}ii):\hspace{.23cm}] $\alpha>0$, $\alpha_q=0$.
\end{itemize}
We can check that each set of parameters satisfies the conditions of the proposition indicated at the beginning of the line, and does not satisfy the conditions of any other result.
Therefore, the different results are not redundant but complementary.}\\

\subsection{Linear division rate, constant death rate}\label{sec:LDCD}

In this section, we study in details the case of a linear division rate. Let us first state the general assumptions of the section.

\begin{customass}{\bf{LDCG}}{(Linear Division, Constant Growth)} \label{ass_E}
We assume that
\begin{itemize}
\item There are no stable jumps ($c_\mathfrak{b}=0$).
\item There exist $\alpha,\beta>0$, $\ g,q \geq 0$ such that $\max(g,\beta)>q$ and
$$g(x) = gx,\ q(x) \equiv q,\ r(x) = \alpha x+\beta .$$
\item $p$ is differentiable and for all $x\geq 0$, $xp'(x)\geq p(x)$, and $\al{\int_{\mathbb{R}_+}} (z^2\wedge z^3)\pi(dz)<\infty$.
\end{itemize}
\end{customass}
We first make three remarks on this assumption: 
\begin{itemize}
\item[-]If the division rate is linear in the quantity of parasites, the mean number of cells at time $t$ depends on the expectation of the total quantity of parasites in the population, which is infinite if we consider stable jumps.
\item[-]In Lemma \ref{lem_esp_Nt2}, we will see that the quantity $\max(g,\beta)-q$ is the Malthusian growth rate of the population. Therefore, we only consider the case of a growing population. 
\item[-] The last point is needed for the existence of an auxiliary process (see Prop. \ref{prop_sol_SDE_auxi}).
\end{itemize}
A linear increase of the division rate in order to get rid of the parasites is an efficient strategy for the cell population, as stated in the next proposition. Recall that $\mathcal{G}$ has been defined in \eqref{def_gene}. To get stronger convergence results, we also consider additional conditions to control the noise in the dynamic of the parasites for large values.
\begin{customass}{\bf{LDCG+}} \label{ass_E1}
We assume that Assumption \ref{ass_E} holds and that
\begin{equation}\label{eq:limsup_sigma}
\limsup_{x \to \infty} \frac{\sigma^2(x)}{x^2}<\infty\quad\text{and }\quad \limsup_{x \to \infty} \frac{p(x)}{x^2}<\infty.
\end{equation}
\end{customass}
\begin{prop}\label{prop:containment}
For all $x\geq 0$:
\begin{enumerate}
\item[i)] Under Assumption \ref{ass_E}, for all $a >0$,
$$ \frac{\E_{\delta_x}[\#\lbrace u\in V_t: X_t^u \in [a,a+da]\rbrace]}{\E_{\delta_x}[N_t]} = \frac{\P_x(\mathcal{Y}_t \in [a,a+da])}{a\E_x[\mathcal{Y}_t^{-1}]}, $$
where $\mathcal{Y}$ has infinitesimal generator
\begin{align} \nonumber\label{infgen_mathcalY} \mathcal{A}f(x)= & \mathcal{G}f(x) + p(x)\int \left(f(x+z)-f(x)-zf'(x)\right)\frac{z}{x}\pi(dz) \\
&+\left(2\frac{\sigma^2(x)}{x} +\frac{p(x)}{x}\int z^2\pi(dz)\right)f'(x)  + 2r(x)\left[\int_0^1\theta \left(f(\theta x)-f(x)\right) \kappa(d\theta)\right].
\end{align}
\item[ii)] Under Assumption \ref{ass_E1}, for all $\varepsilon>0$,
\begin{align*} 
\lim_{K\rightarrow \infty}\lim_{t\rightarrow \infty} \mathbb{P}_{\delta_x}
\left( \mathbf{1}_{\{N_t \geq 1\}} \frac{\#\lbrace u\in V_t: X_t^u >K\rbrace}{N_t}>\eps \right) = 0.
\end{align*}
\end{enumerate}
\end{prop}
Point $i)$ of Proposition \ref{prop:containment} gives information on the distribution of the quantity of parasites in the cells in large time.
Point $ii)$ extends the results of \cite{BT11} to a class of division rates increasing with 
the quantity of parasites. It is similar in spirit to \cite[Conjecture 5.2]{BT11} in the case of birth rates increasing with 
the quantity of parasites, but we relax the assumption of bounded division rates. 
Moreover, we consider positive jumps and various diffusive functions for the dynamics of the parasites, and add the possibility for the cells to die. \al{ Point {\it ii)} relies on Proposition \ref{prop:gx-q}\ref{point2-propgx-q}, but thanks to Assumption \ref{ass_E}, which is stronger than Assumption \ref{ass:gx-q}, we are able to give a more precise result on the convergence of the proportions.}\\

From this result, we see that the proportion of \al{highly} infected cells goes to $0$ as $t$ tends to infinity so that a linear division rate is sufficient to contain the infection.\\

\al{
{\bf Running example.} In the case of parasites following \eqref{Feller_diff_jumps} with cell population dynamic given by \eqref{eq:def_example_cells}, Assumption \ref{ass_E} is satisfied if \cha{$\alpha>0$}, $\alpha_q=0$ and $\max(\alpha_g, \beta)>\beta_q$. Then, according to Proposition \ref{prop:containment}ii), the proportion of very infected cells tends to $0$ in probability as time goes to infinity, which is stronger than Proposition \ref{prop:gx-q}\ref{point2-propgx-q}.\\
}

Under additional technical assumptions, we are able to establish a law of large numbers, linking asymptotically the behaviour of a typical individual at time $t$, given by the auxiliary process $Y^{(t)}$, with the behaviour of the whole population. 
\begin{customass}{\bf{LDCG++}} \label{ass_E2} We assume that Assumption \ref{ass_E1} holds and that
 \begin{equation} \label{cond_moment2_pi}
 \int_0^\infty (z \vee z^6)\pi(dz)<\infty.
\end{equation}
\end{customass}
Under this assumption, we obtain a convergence result for the branching process.
\begin{thm}\label{thm:conv_auxi}
Suppose that Assumptions \ref{ass_E2} holds.
Then, for all bounded measurable functions $F:\mathbb{D}([0,T],\mathcal{X})\rightarrow\mathbb{R}$, for all $x,y, T\geq 0$,
\begin{align*}
\E_{\delta_{x}}\left[\mathbf{1}_{\lbrace N_{t+T}\geq 1\rbrace}\left(\frac{\sum_{u\in V_{t+T}}F\left(X_{t+s}^{u},s\leq T\right)}{N_{t+T}}-
\mathbb{E}\left[F\left(Y_{t+s}^{(t+T)},s\leq T\right)\Big|Y_{0}^{(t+T)} = y \right]\right)^2\right]\underset{t\rightarrow \infty}{\longrightarrow} 0,
\end{align*}
where $Y$ is a time-inhomogeneous Markov process solution of a SDE given below in \eqref{eq:EDSauxi}.
\end{thm}
Theorem \ref{thm:conv_auxi} ensures that asymptotically, the trajectory of the traits of a sampling along its ancestral lineage corresponds to the trajectory of $Y$. Hence, the study of the asymptotic behaviour of the proportion of individuals satisfying some properties, such as the proportion of infected individuals, 
is reduced to the study of the time-inhomogenous process $Y$.

To state the following results, we also need to introduce conditions under which the quantity of parasites may {\bf{(LN0)}} or may not {\bf{(SN0)}} reach the state $0$. They are almost necessary and sufficient conditions (see  \cite[Remark 3.2 and Theorem 3.3]{companion}). \al{ 
Condition \ref{A1} (for small noise at $0$) writes
\begin{enumerate}[label=\bf{(SN0)}]
\item\label{A1} 
There exist $a>1$ and $f:\R_+\to \R_+$ such that $\E[\Theta^{1-a}]<\infty$ and
\begin{equation*}
\frac{g(u)}{u}- a\frac{\sigma^2(u)}{u^2} +u^{-\brho}C_a- p(u)I_a(u) = f(u)+o(\ln(u)), \quad u \to 0,
\end{equation*}
\end{enumerate}
where $C_a$ and $I_a$ are defined in \eqref{eq:Ia}.}
Under Assumption \ref{ass_E2}, it reduces to $\sigma^2(u)/u^2 = o(\ln(u^{-1}))$ close to $0$. Condition \ref{A2} (for large noise at $0$) writes
\begin{enumerate}[label=\bf{(LN0)}]
\item\label{A2}  There exist $\eta>0$ and $u_0> 0$ such that for all $u\leq u_0$
\begin{equation*}
\frac{g(u)}{u}-\frac{\sigma^2(u)}{u^2} \leq -\ln(u^{-1}) \left(\ln\ln(u^{-1})\right)^{1+\eta}.
\end{equation*}
\end{enumerate}
Under Assumption \ref{ass_E2} it reduces to $\sigma^2(u)/u^2 \geq \ln(u^{-1})\left(\ln\ln(u^{-1})\right)^{1+\eta}$ close to $0$.

The following proposition states that, in the case of a division rate increasing linearly with the quantity of parasites, at least two long-time behaviours are possible for the infection at the cell population level: extinction or stabilization of the infection.
\begin{prop}\label{prop:temps_long_auxi}
Assume that Assumption \ref{ass_E2} holds.
\begin{itemize}
\item [i)] If \ref{A2} holds, then
\begin{equation*}
\E_{\delta_0}\left[\mathbf{1}_{\{N_t \geq 1\}}\left(\frac{\#\lbrace u\in V_t: X_t^u >0\rbrace} {N_t}\right)^2\right] \to 0, \quad (t\rightarrow \infty).
\end{equation*}
\item [ii)] If \ref{A1} holds, then for any $x_0,y_0>0$, $0<\mathfrak{a}<\mathfrak{b}$, 
\begin{align*}
\E_{\delta_{x_0}}\left[\mathbf{1}_{\lbrace N_{t}\geq 1\rbrace}\left(\frac{\sum_{u\in V_{t}}\mathbf{1}_{\{\mathfrak{a}<X_{t}^{u}<\mathfrak{b}\}}}{N_{t}}-
\P\left(\mathfrak{a}\leq Y^{(t)}_t \leq \mathfrak{b}\Big| Y_0^{(t)}=y_0\right)\right)^2\right]\xrightarrow[t\rightarrow +\infty]{} 0,
\end{align*}
and there exist two processes $(\mathfrak{Y}^{(1)},\mathfrak{Y}^{(2)})$ satisfying
$ \mathfrak{Y}^{(1)}_t\leq Y^{(t)}_t\leq \mathfrak{Y}^{(2)}_t \ \text{a.s.}\ \forall t \geq 0 $
and admitting a proper limit (without atoms in $0$ and $\infty$) at infinity.
 \end{itemize}
\end{prop}
Notice that point {\it i)} covers the classical diffusive function ($\sigma^2(x)=\sigma^2 x$, $\sigma>0$).
The cell population may thus recover if the dynamics of the parasites in a cell is such that the probability of absorption of the infection process 
is positive (condition \ref{A2}). \\
Under condition \ref{A1}, the parasites in the cell line do not get extinct, and point {\it ii)} gives information on the distribution of the quantity of parasites in the cells for large times.  
The auxiliary process $Y$ of Theorem \ref{thm:conv_auxi}, describing the behavior of a typical quantity of parasites, is time-inhomogeneous. The result is thus obtained by a coupling with $\mathfrak{Y}^{(1)}$ and $\mathfrak{Y}^{(2)}$ such that our auxiliary process is sandwiched between the two processes $\mathfrak{Y}^{(1)}$ and $\mathfrak{Y}^{(2)}$, for which we are able to prove convergence and some properties of the limit. \\

\al{
{\bf Running example}. In the case of parasites following \eqref{Feller_diff_jumps} with cell population dynamic given by \eqref{eq:def_example_cells}, if $\alpha>0$, $\alpha_q=0$ and $\alpha_g\vee \beta>\beta_q$, Assumptions \ref{ass_E} and \ref{ass_E2} are satisfied. Moreover,
\begin{itemize}
\item if $\beta_\sigma>0$ then \ref{A2} holds so that Proposition \ref{prop:temps_long_auxi}i) applies and the proportion of infected cells goes to $0$ in $\mathbb{L}^2$ as times goes to infinity;
\item if $\beta_\sigma=0$, then \ref{A1} holds so that Proposition \ref{prop:temps_long_auxi}ii) applies and the proportion of cells with a quantity of parasites in a given interval converges to the \cha{corresponding} probability \cha{for} the auxiliary process $Y^{(t)}$. Therefore, the infection remains but is contained. 
\end{itemize} 
}
The rest of the paper is dedicated to the proofs of the results presented in previous sections. As mentioned before, the proofs 
rely on the construction of an auxiliary process, which gives information on the dynamics of the quantity of parasites in a `typical' cell. 
But to have information on the long-time behaviour of the infection at the population level, we need to derive additional results on the number of cells alive, 
which is not an easy task due to both the death rate and the dependence of the cell division rate in the quantity of parasites.
\section{Many-to-One formula}\label{sec:MTO}

\subsection{Construction of the auxiliary process} \label{section_constr_aux}

Recall from \eqref{Ztdirac} that the population state at time $t$, $Z_t$, can be represented by a sum of Dirac masses. 
 We denote by $(M_t,t\geq 0)$ the first-moment semi-group associated with the population 
process $Z$ given for all measurable functions $f$ and $x,t\geq 0$ by
$$
M_tf(x)=\mathbb{E}_{\delta_x}\left[\sum_{u\in V_t} f(X_t^u)\right].
$$ 
The trait of a typical individual in the population at time $t$ is characterized by the so-called 
auxiliary process $(Y^{(t)}_s, s\leq t)$ (see \cite[Theorem 3.1]{marguet2016uniform} for detailed computations and proofs). Its associated time-inhomogeneous 
semi-group is given for $r \leq s \leq t$, $x \geq 0$ by
\begin{equation}\label{eq:normalization-for-auxi}
P_{r,s}^{(t)}f(x)=\frac{M_{s-r}(fM_{t-s}\mathbf{1})(x)}{M_{t-r}\mathbf{1}(x)},
\end{equation}
where $\mathbf{1}$ is the constant function on $\mathbb{R}_+$ equal to $1$.
More precisely, if we denote by $m(x,s,t)=M_{t-s}\mathbf{1}(x)$ the mean number of cells in the 
population at time $t$ starting 
from one individual with trait $x$ at time $s$ with $s\leq t$,
then, for all measurable bounded functions $F:\mathbb{D}([0,t],\mathbb{R}_+)\rightarrow\mathbb{R}$, we have:
\begin{equation}\label{eq:mto}
\E_{\delta_{x}}\left[\sum_{u\in V_{t}}F\left(X_{s}^{u},s\leq t\right)\right]=m(x,0,t)\E_{x}\left[F\left(Y_{s}^{(t)},s\leq t\right)\right].
\end{equation}
The Markov process $\left(Y_{s}^{(t)}, s\leq t\right)$ is time-inhomogeneous and its law is characterized by its associated 
infinitesimal generator 
$(\mathcal{A}_{s}^{(t)}, s\leq t)$ given for $f\in\mathcal{D}(\mathcal{A})$ and $x\geq 0$ by: 
\begin{align}\label{eq:gene_auxi}
\mathcal{A}_{s}^{(t)}f(x)= & \widehat{\mathcal{G}}_{s}^{(t)}f(x) + 2r(x)\int_{\R_+}\left(f\left(\theta x\right)-f\left(x\right)\right)\frac{m(\theta x,s,t)}{m(x,s,t)}\kappa(d\theta),
\end{align}
\label{eq:auxi}
\begin{align*}
\text{where} \quad \mathcal{D}(\mathcal{A})=\left\lbrace f\in\mathcal{C}_b^2(\mathbb{R}_+)\text{ s.t. } m(\cdot,s,t)f\in\mathcal{C}_b^2(\mathbb{R}_+),\ \forall t\geq 0,\ \forall s\leq t \right\rbrace \quad \text{and} 
\end{align*}
\begin{align*}
\widehat{\mathcal{G}}_{s}^{(t)}f(x)= & \left(g(x)+2\sigma^2(x)\frac{\partial_xm(x,s,t)}{m(x,s,t)}+ p(x)\int_{\mathbb{R}_+}z\left(\frac{m(x+z,s,t)-m(x,s,t)}{m(x,s,t)}\right)\pi(dz)\right)f'(x)\\
& +\sigma^2(x)f''(x) + p(x)\int_{\mathbb{R}_+}(f(x+z)-f(x)-zf'(x))\frac{m(x+z,s,t)}{m(x,s,t)}\pi(dz)\\
& +x\int_{\mathbb{R}_+}(f(x+z)-f(x))\frac{m(x+z,s,t)}{m(x,s,t)}\rho(dz).
\end{align*}
Those formulae come from \cite[Theorem 3.1]{marguet2016uniform}, with $B(x) = r(x)+q(x)$ and $m(x,dy) = 2r(x)(r(x)+q(x))^{-1}\int_0^1\delta_{\theta x}(dy)\kappa(d\theta)$.
Note that explicit expressions for the mean population size $m(x,s,t)$ are usually out of range, except for particular cases (see Section \ref{sec:LDCD}).

\subsection{Role of the death rate in the auxiliary process}
In this section, we compare the auxiliary process associated to a population with or without death. More precisely, we demonstrate why the death rate does not appear in the generator of the auxiliary process.

Let $(\tilde{Z}_t,t\geq 0)$ be the previously defined population process to which we add a trait $D_t^u$ to each individual $u$ in the population: 
$D_t^u=0$, the individual is still alive, $D_t^u=1$, the individual is dead. To compare the population dynamics with or without death, 
we consider that the trait of the dead individuals still evolves and that they can still divide but their 
descendants will be born with the status $D_t=1$. More precisely, 
$$
\tilde{Z}_t := \sum_{u\in V_t}\delta_{(X_t^u, D_t^u)} = \sum_{u\in V_t^0}\delta_{(X_t^u,0)} + \sum_{u\in V_t^1}\delta_{(X_t^u, 1)},
$$
where $V_t^0$ (respectively $V_t^1$) denotes the alive (respectively dead) individuals in the population at time $t$. 
We denote by $N_t^0$ (respectively $N_t^1$) its cardinal and introduce $\tX_t^u  =(X_t^u, D_t^u)$ for all $u\in \mathcal{U}$ and $t\geq 0$.  Next, we consider the following dynamics:
\begin{itemize}
\item[-] a death event for $u$ leads to set $D_t^u=1$. Therefore, it does not affect dead cells. 
\item[-] a division event does not change the status $D_t^u$ of an individual and its descendants inherit the status of their ancestor. 
\item[-] we extend the generator $\mathcal{G}$ to the functions $f:\overline{\mathbb{R}}_+\times \lbrace 0,1\rbrace \rightarrow \mathbb{R}_+$ such that $f(\cdot, 0), f(\cdot,1)\in\mathcal{C}_b^2(\overline{\mathbb{R}}_+)$. 
\end{itemize}
Then, $(\tilde{Z}_t,t\geq 0)$ is defined as the unique strong solution in $\mathcal{M}_P(\mathbb{R}_+\times\lbrace 0, 1\rbrace)$ to
\begin{multline*}
\langle \tilde{Z}_{t},f\rangle = f\left(x_{0}, 0\right)
+\int_{0}^{t}\int_{\mathbb{R}_+}\mathcal{G}f(\widetilde{x})\tilde{Z}_{s}\left(d\widetilde{x}\right)ds+M_{t}^f(x_0,0)\\
 +\int_{0}^{t}\int_{E}\mathbf{1}_{\left\{u\in V_{s^{-}}\right\}}
\left(  \mathbf{1}_{\left\{z\leq r(X_{s^{-}}^u)\right\}}\left(f\left(\theta X_{s^-}^u, D_s^u \right)+ f\left((1-\theta) X_{s^-}^u, D_s^u \right)-
f\left(X_{s^{-}}^{u}, D_s^u\right)\right)\right.\\
\left.+\mathbf{1}_{\left\{0<z-r(X_{s^{-}}^u)\leq q(X_{s^{-}}^u)\right\}}\left(f\left( X_{s^-}^u, 1 \right)- f\left(X_{s^{-}}^{u}, D_{s^{-}}^u \right)\right)\right)
 M\left(ds,du,d\theta,dz\right),
\end{multline*}
for all $f:\mathbb{R}_+\times \lbrace 0,1\rbrace \rightarrow \mathbb{R}_+$ such that $f(\cdot, 0), f(\cdot,1)\in\mathcal{C}_b^2(\mathbb{R}_+)$, 
where $M_\cdot^f$ is an $\tilde{\mathcal{F}}_t$-martingale ($\tilde{\mathcal{F}}_t$ denotes the canonical extension of $\mathcal{F}_t$).
Let $N_t = N_t^0 + N_t^1$. Introduce 
$$m_0(x,s,t):=\mathbb{E}\left[N^0_t \big| \tilde{Z}_s =\delta_{(x,0)}\right]$$
and consider the auxiliary process $\tY^{(t)}_s = (Y_s^{(t)}, D_s)$ for all $0\leq s\leq t$ and its associated generator 
$\widetilde{\mathcal{A}}_s^{(t)}$ given for all $\psi:\overline{\mathbb{R}}_+\times \lbrace 0,1\rbrace \rightarrow \mathbb{R}_+$ 
such that $\psi(\cdot,0),\psi(\cdot,1)\in\mathcal{D}(\mathcal{A})$, and for all $(x,d)\in\overline{\mathbb{R}}_+\times \lbrace 0,1\rbrace$, by
\begin{align}\nonumber\label{eq:gene_2var}
\widetilde{\mathcal{A}}_s^{(t)}\psi(x,d) = &\widehat{\mathcal{G}}_s^{(t)}\psi(\cdot, d)(x)+ 2r(x)\int_{\R_+}\left(\psi(\theta x,d)-\psi(x,d)\right)\frac{m((y,d),s,t)}{m((x,d),s,t)}\kappa(d\theta)\\
&+ q(x)\left(\psi(x,1)-\psi(x,d)\right).
\end{align}
Using the Many-to-One formula \eqref{eq:mto}, we get
\begin{align*}
m_0(x,s,t) = \mathbb{E}\left[\sum_{u\in V_t}\mathbf{1}_{\lbrace D_t^u = 0\rbrace}\Big| \tilde{Z}_s = \delta_{(x,0)}\right] = m((x,0),s,t)\mathbb{P}\left( D_t = 0\big|
\tY_s^{(t)}= (x,0)\right). 
\end{align*}
As we can see on the expression of the generator of the auxiliary process in \eqref{eq:gene_2var}, $D_t$ switches from $0$ to $1$ at rate $q(x)$ and 
$1$ is absorbing for $D_t$. Therefore, 
\begin{align*}
m_0(x,s,t) = \mathbb{E}\left[\exp\left(-\int_s^t q(Y_u^{(t)})du\right)\Big|Y_s^{(t)} = x\right]m((x,0),s,t),
\end{align*}
and in the case $q(x)\equiv q\geq 0$ for all $x\geq 0$, we get
$
m_0(x,s,t) = e^{-q(t-s)}m((x,0),s,t).
$
In particular, for all $x,y\geq 0$
\begin{align}\label{mxtbis}
\frac{m_0(y,s,t)}{m_0(x,s,t)}=\frac{m((y,0),s,t)}{m((x,0),s,t)}.
\end{align}
The expressions appearing in the generator of the auxiliary process given page \pageref{eq:auxi} are identical with or without death, but the difference is hidden in the ratios of $m(y,s,t)/m(x,s,t)$.
From the previous computations, we obtain that
in the case of a constant death rate, the auxiliary 
process is the same as the auxiliary process of a population process without death and 
$$
\mathbb{E}\left[\sum_{u\in V_t^0} f(X_s^u)\Big| Z_r = \delta_x\right] = m_0(x,r,t)\mathbb{E}\left[f\left(Y_s^{u}\right)\big| Y_r^{(t)}=x\right]=e^{-q(t-r)}\mathbb{E}\left[\sum_{u\in V_t} f(X_s^u)\Big| Z_r = \delta_x\right].
$$

\subsection{The case $r(x)= \alpha x + \beta $, $q(x)\equiv q$}
\label{sec_constr_lineaire}

Assume that $c_\brho=0$ (no stable positive jumps). 
Then under Assumption \ref{ass_E}, a direct computation (see \cite[Section 2.2.3]{marguet2016uniform} for details) shows that if $g\neq \beta$, the mean number of individuals can be written
\begin{equation} \label{mxst}
m(x,s,t)  =\frac{\alpha x}{g-\beta}e^{(g-q)(t-s)}+\left(1-\frac{\alpha x}{g-\beta}\right)e^{(\beta-q)(t-s)}.\end{equation}
We introduce the following functions for $y>0$, $s,z\geq 0$, and $\theta \in [0,1]$: 
\begin{align}\label{eq:f1}
f_1(y,s):=
gy + \left(2\sigma^2(y)+p(y)\int_{\R_+}z^2\pi(dz)\right)\frac{\alpha \left(e^{(g-\beta)s}-1\right)}{g-\beta+\alpha y\left(e^{(g-\beta)s}-1\right)},
\end{align}
\begin{align}\label{eq:f2}
f_2(y,s,\theta):=
2 (\alpha y+\beta)\frac{g-\beta+\alpha \theta y \left(e^{(g-\beta)s}-1\right)}{g-\beta+\alpha y\left(e^{(g-\beta)s}-1\right)} \quad \text{and} \quad
\end{align}
\begin{align}\label{eq:f3}
f_3(y,s,z):= 
p(y)\left(1+\frac{\alpha z\left(e^{(g-\beta)s}-1\right)}{(g-\beta)+\alpha y\left(e^{(g-\beta)s}-1\right)}\right).
\end{align}
We obtain that $\mathcal{A}^{(t)}$ is the infinitesimal generator of the solution to the following SDE, when existence and uniqueness in law
of the solution hold. For $0\leq s \leq t$,
\begin{align}\nonumber\label{eq:EDSauxi}
 Y_s^{(t)} =&  Y_0^{(t)}+\int_0^sf_1(Y_u^{(t)},t-u)du + \int_0^s\int_0^\infty \int_0^{f_3(Y_{u^-}^{(t)},t-u,z)}z\widetilde{Q}(du,dz,dx) \\
& + \int_0^s \sqrt{2\sigma^2\left( Y_u^{(t)}\right)}dB_u+ \int_0^s\int_0^1\int_{0}^{f_2(Y_{u^-}^{(t)},t-u,\theta)}
(\theta-1)Y_{u^-}^{(t)}N(du,d\theta,dz),
\end{align}
where $\widetilde{Q}$, $B$ are the same as in \eqref{X_sans_sauts} and $N$ is a PPM on $\R_+\times \R_+\times [0,1]$ with intensity $ds\otimes dx\otimes \kappa(d\theta)$.\\

The auxiliary process $Y^{(t)}$ can be realised as the unique strong solution to the SDE \eqref{eq:EDSauxi}
under some moment conditions on the measure associated with the positive jumps. We also need an additional assumption on $p$ that ensures that the rate of positive jumps $f_3$ of the process $Y$ is increasing with the quantity of parasites.
\begin{prop}\label{prop_sol_SDE_auxi}
 Suppose that Assumption \ref{ass_E} holds.
 Then, Equation \eqref{eq:EDSauxi} has a pathwise unique nonnegative strong solution.
\end{prop}
The proof of this proposition is given in Appendix \ref{app:prop_sol_SDE_auxi}. The case $g=\beta$ will not be considered in this work, as it entails additional computations and 
does not bring new insights.

\section{Proofs}\label{sec:proofs}
\subsection{Proofs of Section \ref{sec_constant_diff}}

Let us introduce the SDE
\begin{align}\nonumber \label{SDE_Y_diff_ct}
Y_t= x& + \int_0^tg(Y_s)ds+ \int_0^t \sqrt{2 \sigma^2(Y_s) }dB_s+ \int_0^t \int_0^{p(Y_{s^-})}
\int_{\mathbb{R}_+}z\widetilde{Q}(ds,dx,dz)\\
&+\int_0^t\int_0^{Y_{s^-}}\int_{\mathbb{R}_+}zR(ds,dx,dz)+\int_0^t\int_0^{2r(Y_{s^-})} \int_0^1  (\theta-1)Y_{s^-}N(ds,dx,d\theta),
\end{align}
where $N$ is a PPM on $\R_+\times \R_+\times [0,1]$ with intensity $ds\otimes dx\otimes \kappa(d\theta)$. Note that $Y$ is well-defined as the unique strong solution to \eqref{SDE_Y_diff_ct} under Assumption \ref{ass_A} (see  \cite[Proposition 2.2]{companion} and Appendix \ref{proof_ex_uni}).
Then, Proposition \ref{prop_constant_diff} is a consequence of the following two lemmas.
\begin{lemma}\label{maj_mto}
 Assume that there exists a real number $\gamma$ such that $r(x)-q(x)\leq \gamma$ for any $x \in \R_+$, and let $f$ be a nonnegative measurable function on $\R_+$.
Then for $x,t \geq 0$,
$$ \E_{\delta_{x}}\left[\sum_{u\in V_{t}}f\left(X_{t}^{u}\right)\right]\leq e^{\gamma t}\E_{x}\left[f\left(Y_{t}\right)\right],
 $$
where $Y$ is the unique strong solution to the SDE \eqref{SDE_Y_diff_ct}.
\end{lemma}
\begin{lemma}\label{lem_gene}
Let $\zeta>0$.
\begin{itemize}
\item[i)]  If Assumption \ref{ass_expl_tps_fini} holds for some $a>1$ such that $\E[\Theta^{1-a}]<\infty$ then  
  $$ \lim_{t \to \infty} \E_{\delta_x}\left[\sum_{u \in V_t} \left(X_t^u\vee \zeta\right)^{1-a} \right] =0,\quad \forall x\geq 0. $$
\item[ii)] If Assumption \ref{ass_abs_tps_fini} holds for some $a<1$
then 
  $$ \lim_{t \to \infty} \E_{\delta_x}\left[\sum_{u \in V_t} \left(X_t^u\wedge \zeta\right)^{1-a} \right] =0,\quad \forall x\geq 0. $$
 \end{itemize}
\end{lemma}
\begin{proof}[Proof of Lemma \ref{maj_mto}]
Let us introduce the generator
\begin{align}\label{eq:gene_auxi1}
\mathcal{A}f(x) = \mathcal{G}f(x) + 2r(x)\int_0^1 \left(f(\theta x)-f(x)\right)\kappa(d\theta).
\end{align}
We normalize the population process similarly as in \eqref{eq:normalization-for-auxi}. Let $\gamma \in \mathbb{R}$ be such that for all $x\geq 0$, $r(x)-q(x)\leq \gamma$. Let $f\in\mathcal{C}_b^2$ and for $x,t\geq 0$ let
$$
\al{\delta_x}\gamma_tf = e^{-\gamma t}\E\left[\sum_{u\in V_t} f(X_t^u)|Z_0 = \delta_x\right]
$$
be the renormalized first moment semigroup of $Z$. 
Then we have
\begin{align*}
\frac{\partial}{\partial t}\al{\delta_x}\gamma_tf  = &\int_{\R_+}\left(\mathcal{G}f(\mathfrak{x})-\gamma f(\mathfrak{x})\right)\gamma_t(d\mathfrak{x})\\
& +\int_{\R_+}\left(r(\mathfrak{x})\int_0^1(f(\theta \mathfrak{x})+f((1-\theta)\mathfrak{x})-f(\mathfrak{x}))\kappa(d\theta)-
q(\mathfrak{x})f(\mathfrak{x})\right)\al{\delta_x}\gamma_t(d\mathfrak{x})\\
= & \int_{\R_+}\left(\mathcal{G}f(\mathfrak{x})-\gamma f(\mathfrak{x})\right)\al{\delta_x}\gamma_t(d\mathfrak{x})\\
& +\int_{\R_+}\left(2r(\mathfrak{x})\int_0^1(f(\theta \mathfrak{x})-f(\mathfrak{x}))\kappa(d\theta)+(r(\mathfrak{x})-q(\mathfrak{x}))f(\mathfrak{x})
\right)\al{\delta_x}\gamma_t(d\mathfrak{x}).
\end{align*}
Using that $r(x)-q(x)\leq \gamma$ for all $x\geq 0$, and recalling the definition of $\mathcal{A}$ in \eqref{eq:gene_auxi1} we obtain
\begin{align*}
\frac{\partial}{\partial t}\al{\delta_x}\gamma_tf\leq & \al{\delta_x}\gamma_t(\mathcal{A}f),
\end{align*}
Finally, by unicity of the solution to the Kolmogorov's backward equation, 
\begin{align}\label{eq:ineq_MTO}
\al{\delta_x}\gamma_tf =  \E\left[\sum_{u\in V_t} f(X_t^u)|Z_0 = \delta_x\right]e^{-\gamma t}\leq \E_{x}\left[f\left(Y_t\right)\right],
\end{align}
where $Y$ is the unique strong solution to the SDE
\eqref{SDE_Y_diff_ct}.
\end{proof}
\begin{proof}[Proof of Lemma \ref{lem_gene}]
Let $Y$ be the unique strong solution to \eqref{SDE_Y_diff_ct}. First, we prove using a coupling argument that under Assumption \ref{ass_expl_tps_fini}, for all $y>0$, $\P_y\left(\tau^-(0)<\infty\right)=0$, 
where 
\begin{equation} \label{tau2} \tau^-(0)=\inf\left\{t\geq 0, Y_t=0\right\}.\end{equation}
 Let $K> 0$. We consider the process $\tilde{Y}$ defined as the unique strong solution to
\begin{align*}
\tilde{Y}_t= x& + \int_0^tg(\tilde{Y}_s)ds+ \int_0^t \sqrt{2 \sigma^2(\tilde{Y}_s) }dB_s+ \int_0^t \int_0^{p(\tilde{Y}_{s^-})}
\int_{\mathbb{R}_+}z\widetilde{Q}(ds,dx,dz)\\
&+\int_0^t\int_0^{\tilde{Y}_{s^-}}\int_{\mathbb{R}_+}zR(ds,dx,dz)+\int_0^t\int_0^{2\overline{r}_K} \int_0^1  (\theta-1)\tilde{Y}_{s^-}N(ds,dx,d\theta),
\end{align*}
where $B,\widetilde{Q}$ and $N$ are the same as in \eqref{SDE_Y_diff_ct} and $\overline{r}_K=\sup_{x\in[0,K]}r(x)$. Then, as $p$ is a non-decreasing function, 
\begin{align}\label{eq:coupling}
\tilde{Y}_t\leq Y_t\quad \text{ for all }\quad t\leq \tau^+(K):=\inf\left\{t\geq 0, Y_t\geq K\right\}.
\end{align}
Let $\tilde{\tau}^-(0)=\inf\{t\geq 0, \tilde{Y}_t=0\}$. Adapting \cite[Theorem 3.3]{companion} (see Appendix \ref{app:th33-with-stable-jumps}), if \ref{A1} holds, $\P_y\left({\tilde{\tau}}^-(0)<\infty\right)=0$ for all $y>0$.
Let $a>1$ be as in Assumption \ref{ass_expl_tps_fini}. Using that $G_a(x)\geq \gamma'$, we have
\begin{align*}
\frac{g(x)}{x}-a\frac{\sigma^2(x)}{x^2}+x^{-\brho}C_a -p(x)I_a(x)\geq (a-1)^{-1}\gamma'+2r(x)\frac{\E[\Theta^{1-a}-1]}{a-1},
\end{align*}
so that \ref{A1} holds.
We thus obtain that for all $y>0$,
$
\P_y\left(\tilde{\tau}^-(0)<\infty\right)=0.
$
Then, from \eqref{eq:coupling} we get $\P_y\left(\tau^-(0)<\tau^+(K)\right)=0$ and letting 
$K$ tend to infinity yields 
\begin{equation}\label{non_ext_gene}
\P_y\left(\tau^-(0)<\infty\right)=0.
\end{equation}
Now,  let $\zeta>0$ and $\eps>0$.  Let $T(\eps):=\tau^-(\eps) \wedge \tau^+(1/\eps)$ where 
\begin{equation} \label{tau1} 
\tau^-(\eps)=\inf\left\{t\geq 0, Y_t<\eps\right\}.
\end{equation}
Then we have from \eqref{eq:ineq_MTO}, for every $t\geq 0$, $x>0$,
 \begin{align*}
  \E_{\delta_x} \left[ \sum_{u \in V_t} \left(X_t^u \vee \zeta \right)^{1-a} \right] \leq e^{\gamma t}  \E_x \left[\mathbf{1}_{\{t \leq T(\eps)\}}Y_t^{1-a} \right]
   + e^{\gamma t}  \E_x \left[\mathbf{1}_{\{t > T(\eps)\}}\left(Y_t \vee \zeta \right)^{1-a} \right].
 \end{align*}
For the first term, using that $G_a(x)\geq \gamma'$ for all $x\geq 0$ and the martingale property proved in Lemma \ref{prop:martingale-co}, we have
 \begin{align*}
  e^{\gamma t}  \E_x \left[\mathbf{1}_{\{t \leq T(\eps)\}}Y_t^{1-a} \right] 
  & = e^{(\gamma-\gamma') t}  \E_x \left[Y_{t \wedge T(\eps)}^{1-a} e^{\gamma'(t \wedge T(\eps))}
  \mathbf{1}_{\{t \leq T(\eps)\}}\right]\\
  & \leq e^{(\gamma-\gamma') t}  \E_x \left[ Y_{t \wedge T(\eps)}^{1-a} e^{\int_0^{t \wedge T(\eps)}
  G_{a}(Y_s)ds} \right] = x e^{(\gamma-\gamma') t}.
 \end{align*}
And for the second term,
 \begin{align*}
  \E_x \left[\mathbf{1}_{\{t > T(\eps)\}}\left(Y_t \vee \zeta \right)^{1-a} \right] 
  & \leq  \E_x \left[\mathbf{1}_{\{t>\tau^-(\eps)\}}\left(Y_t \vee \zeta \right)^{1-a} \right] 
  +  \E_x \left[\mathbf{1}_{\{ t>\tau^+(1/\eps)  \}}\left(Y_t \vee \zeta \right)^{1-a} \right] .
 \end{align*}
For any fixed $t\geq 0$, as $a>1$,
$$ \mathbf{1}_{\{t>\tau^-(\eps) \}}\left(Y_t \vee \zeta \right)^{1-a}\leq \zeta ^{1-a} $$
where $\zeta^{1-a}$ is finite and does not depend on $\eps$, and we know thanks to \eqref{non_ext_gene} that 
$$  \lim_{\eps \to 0}\mathbf{1}_{\{t>\tau^-(\eps) \}} = 0 \quad \text{almost surely.} $$
Hence by the dominated convergence theorem, we obtain that
$$\lim_{\eps \to 0}\E_x \left[\mathbf{1}_{\{t>\tau^-(\eps)\}}\left(Y_t \vee \zeta \right)^{1-a} \right]  = 0 \quad \text{almost surely.} $$
Let us finally consider the last term. 
First, notice that for every $\eps>0$
\begin{equation}\label{maj_CVD}
\mathbf{1}_{\{ t> \tau^+(1/\eps)  \}}\left(Y_t \vee \zeta \right)^{1-a}\leq \zeta^{1-a}
\end{equation}
where $\zeta^{1-a}$ is finite and does not depend on $\eps$.
Now, let us consider the sequence of stopping times $(\tau^+(1/\eps), \eps>0)$. This sequence increases when $\eps$ decreases, and there exists 
$\tau^+(\infty)$, which may be infinite, defined by
\begin{equation} \label{deftauinfty} \lim_{\eps \to 0} \tau^+(1/\eps)=: \tau^+(\infty). \end{equation}
There are two cases:
\begin{itemize}
 \item If $\tau^+(\infty)\leq t$ then $Y_t=\infty$ and 
 $  \mathbf{1}_{\{ t>\tau^+(1/\eps)   \}}\left(Y_t \vee \zeta \right)^{1-a} = 0$ forall $\eps>0$.
 \item If $t>\tau^+(\infty)$ then there exists $\eps_0>0$ such that for any $\eps \leq \eps_0$, $\tau^+(1/\eps) \geq  t$, and  
  $$  \mathbf{1}_{\{ t>\tau^+(1/\eps) \}}\left(Y_t \vee \zeta \right)^{1-a} = 0,\quad \forall \eps\leq\eps_0.$$
\end{itemize}
We deduce that for any fixed $t\geq 0$
$$ \lim_{\eps \to 0} \mathbf{1}_{\{ t>\tau^+(1/\eps)  \}}\left(Y_t \vee \zeta \right)^{1-a} = 0 \quad \text{almost surely.} $$
From \eqref{maj_CVD} we may apply the dominated convergence theorem and obtain
$$ \lim_{\eps \to 0} \E_x \left[\mathbf{1}_{\{ t>\tau^+(1/\eps)  \}}\left(Y_t \vee \zeta \right)^{1-a} \right] = 0. $$

To sum up, we proved that for any $t \geq 0$, $x>0$, and $\eps>0$, 
 \begin{align*}
  \E_{\delta_x} \left[ \sum_{u \in V_t} \left(X_t^u \vee \zeta \right)^{1-a} \right] 
   & \leq x e^{(\gamma-\gamma') t} + e^{\gamma t}  \E_x \left[\mathbf{1}_{\{t > T(\eps)\}}\left(Y_t \vee \zeta \right)^{1-a} \right] 
   \xrightarrow[\eps \to 0]{} x e^{(\gamma-\gamma') t} .
 \end{align*}
Letting $t$ tend to infinity ends the proof of point {\it i)}, as $\gamma<\gamma'$.
 
Let us now turn to the proof of point $ii)$. It is similar in spirit to the proof of point $i)$. 
Let $a<1$ and $\gamma'>0$ be such that Assumption \ref{ass_abs_tps_fini} holds. 
Adapting \cite[Theorem 4.1i)]{companion} (see Proposition \ref{prop_B3}), under Condition \ref{eq:SNinfinity-1}, we get for $Y$ (defined as before as the unique strong solution to \eqref{SDE_Y_diff_ct}), and for all $y>0$,
\begin{align}\label{eq:non_expl_caseii}
 \P_y\left(\tau^+(\infty)<\infty\right)=0,
\end{align}  
where $\tau^+(\infty)$ has been defined in \eqref{deftauinfty}.
And as $G_a(x)\geq \gamma'$ for all $x\geq 0$, \ref{eq:SNinfinity-1} holds.

Let $\eps>0$.
Similar computations as for point $i)$ lead to
 \begin{align*}
  \E_{\delta_x}\left[ \sum_{u \in V_t} \left(X_t^u \wedge \zeta \right)^{1-a} \right] 
   & \leq x e^{(\gamma-\gamma') t} + e^{\gamma t}  
   \left( \E_x \left[\mathbf{1}_{\{\tau^-(\eps) <t\}}\left(Y_t \wedge \zeta \right)^{1-a} \right] 
  +  \P_x ( \tau^+(1/\eps) < t )\zeta^{1-a} \right).
 \end{align*}
From \eqref{eq:non_expl_caseii}, the last term converges to $0$ when $\eps$ goes to $0$.
Moreover, distinguishing between the cases 
$ \{\tau^-(0)\leq t\}$ and $\{\tau^-(0)> t\} $
and applying the dominated convergence theorem, the second term also converges to $0$ when $\eps$ goes to $0$.
This ends the proof of {\it ii)}.
\end{proof}
\begin{proof}[Proof of Proposition \ref{prop_constant_diff}]
Let $K>0$, $t\geq 0$. Under Assumption \ref{ass_expl_tps_fini}, as $a>1$:
\begin{align*}
 \P_{\delta_x}\left( \exists u \in V_t, X_t^u \leq K \right)& = \P_{\delta_x}\left( \exists u \in V_t, \left(X_t^u \vee K\right)^{1-a} = K^{1-a} \right)\\
 & \leq \P_{\delta_x}\left( \sum_{u \in V_t} \left(X_t^u \vee K\right)^{1-a} \geq K^{1-a} \right) \leq \E_{\delta_x}\left[ \sum_{u \in V_t} \left(X_t^u \vee K\right)^{1-a}\right]K^{a-1}.
\end{align*}
The other case is similar. Proposition~\ref{prop_constant_diff} follows from Lemma~\ref{lem_gene}.
 \end{proof}

\subsection{Proofs of Section \ref{sec:roleofnoise} and \ref{sec:competition} }

\begin{proof}[Proof of Proposition \ref{prop:betar-q}]
Recall that $\alpha_\gamma$ is well-defined because $\psi:\alpha\mapsto  2\E\left[\Theta^\alpha\right] - 1$ is decreasing on $\R$, and $\psi((-\infty,1]) = [0,+\infty)$. Let $\mathcal{B}$ be the infinitesimal generator associated with the first moment semigroup of the branching process given for all $f\in\mathcal{C}^2_b(\mathbb{R}_+)$ by
\begin{align} \label{defmathcalB}
\mathcal{B}f(x) = g(x)f'(x) + \sigma^2(x)f''(x) +p(x)\int_{\mathbb{R}_+}\left(f(x+z)-f(x)-zf'(x)\right)\pi(dz)\\
+ r(x)\int_0^1\left(f(\theta x)+ f((1-\theta)x)-f(x)\right)\kappa(d\theta) - q(x)f(x). \nonumber
\end{align}
 Let $V_\alpha(x) = x^\alpha$. Under Assumption \ref{ass:betar-q}, we have
\begin{align*}
\mathcal{B}V_\alpha(x) = \left(\alpha g + \alpha(\alpha - 1)\sigma^2+ r(x) \int_0^1\left(2\theta^\alpha-1\right)\kappa(d\theta)	-q(x)\right)V_\alpha(x).
\end{align*} Then, as  $2\E\left[\Theta^{\alpha_\gamma}\right] - 1 = \gamma$, we obtain
\begin{align*}
\mathcal{B}V_{\alpha_\gamma}(x) = \left(\alpha_\gamma g + \alpha_\gamma(\alpha_\gamma - 1)\sigma^2 + \mathfrak{c}\right)V_{\alpha_\gamma}(x):=\lambda V_{\alpha_\gamma}(x).
\end{align*}
Then, according to \cite[Lemma 3.3]{cloez2017limit}, we have for all $x\geq 0$, $f\in\mathcal{C}^2_b(\mathbb{R}_+)$, 
\begin{align}\label{eq:mtoV}
\E_{\delta_x}\left[\sum_{u\in V_t} f(X_t^u)\right]= x^{\alpha_\gamma}e^{\lambda t}\E_x\left[ f(\mathcal{Y}_t)V_{\alpha_\gamma}^{-1}(\mathcal{Y}_t)\right],
\end{align}
where $\mathcal{Y}$ is a Markov process with infinitesimal generator given by
\begin{align*}
\mathcal{A}_{V_{\alpha_\gamma}}f(x) = \left(g + 2\sigma^2  \alpha_\gamma\right)xf'(x) + \sigma^2 x^2f''(x) + 2r(x)\left[\int_0^1\theta^{\alpha_\gamma}  \left(f(\theta x)-f(x)\right) \kappa(d\theta)\right],
\end{align*}
 and $\mathcal{F}^\mathcal{Y}$ its natural filtration \al{(see Appendix \ref{app:generator_auxi} for details of the computation of $\mathcal{A}_{V_{\alpha_\gamma}}$ in the case $\alpha_\gamma=1$)}.
\al{In particular, we have
\begin{align*}
d\mathcal{Y}_t^{-\alpha_\gamma} &= \left(-\alpha_\gamma \left(g + 2\sigma^2  \alpha_\gamma\right) + \sigma^2\alpha_\gamma(\alpha_\gamma + 1) + 2r\left(\mathcal{Y}_t\right)\int_0^1\left(1-\theta^{\alpha_\gamma} \right)\kappa(d\theta)\right)\mathcal{Y}_t^{-\alpha_\gamma}dt + dM_t^{(1)}\\
& = \left( r(\mathcal{Y}_t)- q(\mathcal{Y}_t) -\lambda \right)\mathcal{Y}_t^{-\alpha_\gamma}dt + dM^{(1)}_t,
\end{align*}
where $(M^{(1)}_t,t\geq 0)$ is a $\mathcal{F}^\mathcal{Y}_t$-martingale, and for the last equality, we used $2r(x)\E\left[\Theta^{\alpha_\gamma}\right] = \al{\gamma} r(x) + r(x) = \mathfrak{c} + r(x) + q(x)$ and $\lambda = \alpha_\gamma g + \alpha_\gamma(\alpha_\gamma - 1)\sigma^2 + \mathfrak{c}.$}
Note also that taking $f\equiv 1$ in \eqref{eq:mtoV}, we obtain
\begin{align}\label{eq:Nt}
\E_{\delta_x}\left[N_t\right]= x^{\alpha_\gamma}e^{\lambda t}\E_x\left[\mathcal{Y}_t^{-\alpha_\gamma}\right].
\end{align}
Next, let $K\geq 0$ and $f(x)=\mathbf{1}_{\lbrace x>K\rbrace}$. Combining \eqref{eq:mtoV} and \eqref{eq:Nt}, we have
\begin{align*}
\frac{\E_{\delta_x}\left[\sum_{u\in V_t} \mathbf{1}_{\lbrace X_t^u>K\rbrace}\right]}{\E_{\delta_x}\left[N_t\right]}= \frac{\E_x\left[ \mathbf{1}_{\lbrace \mathcal{Y}_t>K\rbrace}\mathcal{Y}_t^{-\alpha_\gamma}\right]}{\E_x\left[\mathcal{Y}_t^{-\alpha_\gamma}\right]}.
\end{align*}
First, we prove the point \ref{point1-propbetar-q} of Proposition \ref{prop:betar-q}. Let $\gamma\in [0,1)$. Then, $\alpha_\gamma\in (0,1]$ and  
\begin{align}\label{eq:casebeta01}
\frac{\E_{\delta_x}\left[\sum_{u\in V_t} \mathbf{1}_{\lbrace X_t^u>K\rbrace}\right]}{\E_{\delta_x}\left[N_t\right]}\leq \frac{K^{-\alpha_\gamma}\P_x(\mathcal{Y}_t>K)}{\E_x\left[\mathcal{Y}_t^{-\alpha_\gamma}\right]}.
\end{align}
Moreover, in this case
\begin{align*}
\frac{d}{dt}\E_x\left[\mathcal{Y}_t^{-\alpha_\gamma}\right] = \E_x\left[\left( (1-\gamma)r(\mathcal{Y}_t) + \mathfrak{c} -\lambda \right)\mathcal{Y}_t^{-\alpha_\gamma}\right]\geq\left( \mathfrak{c} -\lambda \right) \E_x\left[\mathcal{Y}_t^{-\alpha_\gamma}\right],
\end{align*}
and by Grönwall's lemma we obtain
\begin{align*}
\E_x\left[\mathcal{Y}_t^{-\alpha_\gamma}\right]\geq x^{-\alpha_\gamma}e^{(\mathfrak{c} -\lambda)t}.
\end{align*}
Combining this inequality with \eqref{eq:casebeta01}, we have
\begin{align*}
\frac{\E_{\delta_x}\left[\sum_{u\in V_t} \mathbf{1}_{\lbrace X_t^u>K\rbrace}\right]}{\E_{\delta_x}\left[N_t\right]}\leq \left(\frac{x}{K}\right)^{\alpha_\gamma}e^{\alpha_\gamma(g + (\alpha_\gamma -1)\sigma^2)t}
\end{align*}
where we used that $\lambda-\mathfrak{c} = \alpha_\gamma(g + (\alpha_\gamma -1)\sigma^2)$. This ends the proof of point \ref{point1-propbetar-q}. 

For the point \ref{point3-propbetar-q}, let $\gamma>1$. Then, $\alpha_\gamma<0$. Using Itô's formula, we have for all $t\geq 0$,
\begin{align*}
\ln(\mathcal{Y}_t) = \ln(\mathcal{Y}_0) + (g + \sigma^2(2\alpha_\gamma -1))t + 2\E\left[\Theta^{\alpha_\gamma}\ln(\Theta)\right]\int_0^tr(\mathcal{Y}_s)ds + M^{(2)}_t, 
\end{align*}
where $(M^{(2)}_t,t\geq 0)$ is a $\mathcal{F}^\mathcal{Y}_t$-martingale. By assumption, there exist $x_0(\alpha_\gamma)>0$ (later noted $x_0$) and $\eta>0$ such that \eqref{star2} holds.
Then, for $0<y_0<x_0$, 
\begin{align}\label{eq:logY}
\ln\left(\mathcal{Y}_{t\wedge \uptau^+(x_0)}\right) - \ln(y_0) \geq \eta\left(t\wedge \uptau^+(x_0)\right) + M^{(2)}_{t\wedge \uptau^+(x_0)},
\end{align}
where $\uptau^+(x_0)$ has the same definition as in \eqref{eq:coupling} except that it is for the process $\mathcal{Y}$.
Note that $\mathcal{Y}_{\uptau^+(x_0)} = x_0$. Therefore, for all $t\geq 0$,  $\ln\left(\mathcal{Y}_{t\wedge \uptau^+(x_0)}\right)\leq \ln\left(x_0\right)$ almost surely. Then, taking the expectation in \eqref{eq:logY}, using this inequality and letting $t$ tend to infinity, we obtain 
\begin{align*}
\E_{y_0}\left[\uptau^+(x_0)\right]\leq \eta^{-1}\ln\left(\frac{x_0}{y_0}\right)<\infty.
\end{align*}
Applying \cite[Theorem 7.1.4]{BN}, $\mathcal{Y}_t$ converges in law to a variable $\mathcal{Y}_\infty$ on $\R_+$, satisfying
$$ \P(\mathcal{Y}_\infty\leq A) = \frac{1}{\E_{x_0}[\uptau^+(x_0)]}\E \left[ \int_0^{\uptau^+(x_0)} \mathbf{1}_{\{\mathcal{Y}_s \leq A\}}ds \right] .$$ 
Moreover, as \ref{A1} is satisfied thanks to \eqref{star2}, applying \cite[Theorem 3.3i)]{companion}, we obtain $\P(\mathcal{Y}_\infty = 0)=0$. Recalling that as $\alpha_{\gamma}<0$, Fatou's lemma implies that for any $x>0$,
\begin{align} \label{min_Yinfty} \liminf_{t \to \infty} \E_x\left[\mathcal{Y}_t^{-\alpha_\gamma}\right] \geq \liminf_{t \to \infty} \E_x[(\mathcal{Y}_t\wedge 1)^{-\alpha_\gamma}] =\E_x[(\mathcal{Y}_\infty\wedge 1)^{-\alpha_\gamma}] >0 .
\end{align}
Finally, as before, combining \eqref{eq:mtoV} and \eqref{eq:Nt}, we obtain for any $K\geq 0$, 
\begin{align*}
\frac{\E_{\delta_x}\left[\sum_{u\in V_t} \mathbf{1}_{\lbrace X_t^u\leq K\rbrace}\right]}{\E_{\delta_x}\left[N_t\right]}= \frac{\E_x\left[ \mathbf{1}_{\lbrace \mathcal{Y}_t\leq K\rbrace}\mathcal{Y}_t^{-\alpha_\gamma}\right]}{\E_x\left[\mathcal{Y}_t^{-\alpha_\gamma}\right]}&\leq  \frac{K^{-\alpha_\gamma}}{\E_x\left[\mathcal{Y}_t^{-\alpha_\gamma}\right]}.
\end{align*}
Using \eqref{min_Yinfty} ends the proof of \ref{point3-propbetar-q}.

The proof of point \ref{point2-propbetar-q} is very similar, and we only give the main ideas. Assume that there exist $x_0(\alpha_\gamma),\eta>0$ such that \eqref{star1} holds.
Then, one can prove as before that for $y_0>x_0>0$,
$$  \E_{y_0}[\uptau^-(x_0)]\leq (\E[\ln 1/\Theta]+ \ln(y_0/x_0))/\eta, $$
where $\uptau^-(x_0)$ has the same definition as in \eqref{tau1} except that it is for $\mathcal{Y}$.
Applying again \cite[Theorem 7.1.4]{BN}, $\mathcal{Y}_t$ converges in law to a variable $\mathcal{Y}_\infty$ on $\R_+$. As \ref{eq:SNinfinity-1} is satisfied thanks to \eqref{star1}, \cite[Theorem 4.1i)]{companion} implies $\P(\mathcal{Y}_\infty = \infty)=0$. Finally, as $\alpha_\gamma \geq 0$,
$$ \E_x[\mathcal{Y}_\infty^{-\alpha_\gamma}] \geq\E_x\left[(\mathcal{Y}_\infty\vee 1)^{-\alpha_\gamma}\right] >0. $$
Combining this with \eqref{eq:mtoV} and \eqref{eq:Nt} ends the proof as before. 
\end{proof}
\begin{proof}[Proof of Proposition \ref{prop:gx-q}]
The proof is very similar to the proof of Proposition \ref{prop:betar-q}, we thus only give the main ideas. 
Recall the definition of $\mathcal{B}$ in \eqref{defmathcalB}. Then for $V_1(x)=x$, we obtain 
$$ \mathcal{B}V_1(x) = g(x) - q(x)V_1(x)= \left( \frac{g(x)}{x}-q(x) \right)V_1(x)=\mathfrak{c} V_1(x), $$
and \eqref{eq:mtoV} thus holds where $\mathcal{Y}$ is the Markov process with infinitesimal generator given by
\begin{align*}
\mathcal{A}_{V_{1}}f(x) = & \left(g(x) + 2\frac{\sigma^2(x)}{x} +\frac{p(x)}{x}\al{\int_{\mathbb{R}_+}} z^2\pi(dz)\right)f'(x) + \sigma^2(x)f''(x)\\
&+ p(x)\al{\int_{\mathbb{R_+}}} \left(f(x+z)-f(x)-zf'(x)\right)\frac{(x+z)}{x}\pi(dz)   + 2r(x)\left[\int_0^1\theta \left(f(\theta x)-f(x)\right) \kappa(d\theta)\right].
\end{align*}
\al{For more details on the computations of this generator, we refer to Appendix \ref{app:generator_auxi}. }

Let us begin with the proof of point \ref{point2-propgx-q}. 
We have
\begin{align*}
\frac{\E_{\delta_x}\left[\sum_{u\in V_t} \mathbf{1}_{\lbrace X_t^u\geq K\rbrace}\right]}{\E_{\delta_x}\left[N_t\right]}= \frac{\E_x\left[ \mathbf{1}_{\lbrace \mathcal{Y}_t\geq K\rbrace}\mathcal{Y}_t^{-1}\right]}{\E_x\left[\mathcal{Y}_t^{-1}\right]}&\leq  \frac{K^{-1}}{\E_x\left[\mathcal{Y}_t^{-1}\right]}.
\end{align*}
Let us prove that $\liminf_{t\rightarrow \infty}\E_x\left[\mathcal{Y}_t^{-1}\right]>0$. Using Itô's formula and Taylor's formula for the term accounting for positive jumps, we have for all $t\geq 0$,
\begin{align*}
\ln(\mathcal{Y}_t)= & \ln(\mathcal{Y}_0) + \int_0^t \left(\frac{g(\mathcal{Y}_s)}{\mathcal{Y}_s} + \frac{\sigma^2(\mathcal{Y}_s)}{\mathcal{Y}_s^2} + p(\mathcal{Y}_s)I_0\left(\mathcal{Y}_s\right)\right)ds+ 2\E\left[\Theta\ln(\Theta)\right]\int_0^tr(\mathcal{Y}_s)ds + M^{(3)}_t, 
\end{align*}
where $(M^{(3)}_t,t\geq 0)$ is a $\mathcal{F}^\mathcal{Y}_t$-martingale. 
 By assumption, there exist $x_0,\eta>0$ such that \eqref{star3} holds.
Then, for $0<x_0<y_0$, 
\begin{align}\label{eq:logY}
\ln\left(\mathcal{Y}_{t\wedge \uptau^-(x_0)}\right) - \ln(y_0) \leq -\eta\left(t\wedge \uptau^-(x_0)\right) + M^{(3)}_{t\wedge \uptau^-(x_0)}.
\end{align}
Hence, taking the expectation in \eqref{star3}, we obtain 
$
\E_{y_0}[\uptau^-(x_0)]\leq (\E[\ln 1/\Theta]+ \ln(y_0/x_0))/\eta.
$

Applying \cite[Theorem 7.1.4]{BN}, $\mathcal{Y}_t$ converges in law to a variable $\mathcal{Y}_\infty$ on $\R_+$, satisfying
$$ \P(\mathcal{Y}_\infty\leq A) = \frac{1}{\E_{x_0}[\uptau^-(x_0)]}\E \left[ \int_0^{\uptau^-(x_0)} \mathbf{1}_{\{\mathcal{Y}_s \leq A\}}ds \right] .$$ 
As \ref{eq:SNinfinity-1} is satisfied, \cite[Theorem 4.1i)]{companion} ensures $\P(\mathcal{Y}_\infty = \infty)=0$.  Fatou's lemma implies that for any $x>0$,
\begin{align*} \liminf_{t \to \infty} \E_x\left[\mathcal{Y}_t^{-1}\right] \geq \liminf_{t \to \infty} \E_x[(\mathcal{Y}_t\vee 1)^{-1}] =\E_x[(\mathcal{Y}_\infty\vee 1)^{-1}] >0,
\end{align*}
which ends the proof of point \ref{point2-propgx-q}.

For point \ref{point1-propgx-q}, 
applying the generator $\mathcal{A}_{V_{1}}$ to the function $f(x)=x^{-2}$ (as $\P(\mathcal{Y}_t=0)=0$ for all $t\geq 0$ using \eqref{eq:mtoV} with $f(x)=x\mathbf{1}_{\{x=0\}}$), we obtain:
$$ d\left(\mathcal{Y}_t^{-2}\right)=- 2\mathcal{Y}_t^{-2} \left( \frac{g(\mathcal{Y}_t)}{\mathcal{Y}_t}- \frac{\sigma^2(\mathcal{Y}_t)}{\mathcal{Y}_t^2}- \frac{p(\mathcal{Y}_t)}{2\mathcal{Y}_t} \int_{\R_+}\frac{z^2\pi(dz)}{\mathcal{Y}_t+z}- r(\mathcal{Y}_t)\E\left[\frac{1}{\Theta}-\frac{1}{2}\right] \right)dt+  dM^{(4)}_{t}, $$
where $(M^{(4)}_t,t\geq 0)$ is a $\mathcal{F}^\mathcal{Y}_t$-martingale. 
From \eqref{add_ass}, there exists $C_1<\infty$ such that
$$ \sup_{x \geq x_0} \left\{2\frac{\sigma^2(x)}{x^4} + \frac{p(x)}{x^3} \int_{\R_+}\frac{z^2\pi(dz)}{x+z}+ \frac{r(x)}{x^2}\E\left[\frac{2}{\Theta}-1\right]  \right\} = C_1. $$
Adding condition \eqref{star4}, we thus obtain that
$$ \frac{d}{dt} \E \left[\mathcal{Y}_t^{-2} \right]\leq \left( -2\eta_0  \E \left[\mathcal{Y}_t^{-2} \right] +\frac{2}{x_0^2}\\eta_0 + C_1 \right), $$
which yields the existence of a finite constant $C_2$ such that
$ \sup_{t \geq 0} \E \left[\mathcal{Y}_t^{-2}\right]=C_2. $
Applying the generator $\mathcal{A}_{V_1}$ to the function identity and taking the expectation, we get
\begin{align*}
\frac{d}{dt}\E\left[\mathcal{Y}_t\right]= \E\left[\mathcal{Y}_t\left( \frac{g(\mathcal{Y}_t)}{\mathcal{Y}_t}+2 \frac{\sigma^2(\mathcal{Y}_t)}{\mathcal{Y}_t^2}+ \frac{p(\mathcal{Y}_t)}{\mathcal{Y}_t^2} \int_{\R_+}z^2\pi(dz)+ 2r(\mathcal{Y}_t)\left(\E\left[\Theta^2\right]-\frac{1}{2}\right) \right)\right].
\end{align*}
Combining \eqref{eq:condinfty} with \ref{ass_A}, there exists $C_3>0$ such that
$
\frac{d}{dt}\E \left[\mathcal{Y}_t \right]\leq  -\eta_0\E\left[\mathcal{Y}_t \right] + C_3,
$
 which yields the existence of a finite constant $C_4$ such that
$ \sup_{t \geq 0} \E \left[\mathcal{Y}_t \right]=C_4. $ Using Jensen's inequality, we obtain $ \E \left[\mathcal{Y}_t^{-1} \right]^{-1}\leq C_4. $ Next, from the definition of $\mathcal{Y}$ we know that
\begin{align*}
\frac{\E_{\delta_x}\left[\sum_{u\in V_t} \mathbf{1}_{\lbrace X_t^u\leq K\rbrace}\right]}{\E_{\delta_x}\left[N_t\right]}= \frac{\E_x\left[ \mathbf{1}_{\lbrace \mathcal{Y}_t\leq K\rbrace}\mathcal{Y}_t^{-1}\right]}{\E_x\left[\mathcal{Y}_t^{-1}\right]}.
\end{align*}
Using the previous computations and Cauchy-Schwarz inequality, we obtain
\begin{align*} \frac{\E_x\left[ \mathbf{1}_{\lbrace \mathcal{Y}_t\leq K\rbrace}\mathcal{Y}_t^{-1}\right]}{\E_x\left[\mathcal{Y}_t^{-1}\right]} &\leq C_4 \sqrt{\P(\mathcal{Y}_t^{-2}\geq K^{-2})\E \left[\frac{1}{\mathcal{Y}_t^{2}} \right]}
\leq C_2C_4K, 
\end{align*}
which ends the proof of point \ref{point1-propgx-q}.
\end{proof}

\subsection{Proof of Section \ref{sec:LDCD}}

\subsubsection{Proof of Proposition \ref{prop:containment}}

First, we need to control the value of the second moment of the population size relatively to the square of its mean.
\begin{lemma} \label{lem_esp_Nt2}
Suppose that Assumption \ref{ass_E} holds.
Then
for all $x> 0$, 
\begin{itemize}
\item[{\it (i)}] if $\beta> \max(g,q)$,
\begin{align*}
\mathbb{E}_{\delta_x}\left[N_t^2\right] \underset{t \to \infty}{\sim} C_1^2(x)e^{2(\beta-q)t},\quad 
\mathbb{E}_{\delta_x}^2\left[N_t\right] \underset{t \to \infty}{\sim} \left(1 + \frac{\alpha x}{\beta-g}\right)^2e^{2(\beta-q)t},
\end{align*}
where 
$$C_1^2(x)=1 + \frac{\alpha x}{2\beta-g-q} - \frac{\alpha^2 x^2}{(\beta-g)^2}
+\left(1+\frac{\alpha x}{\beta-g}\right)\left(\frac{2\alpha x}{\beta-g}+\frac{\beta+q}{\beta-q}\right)+\frac{\alpha x}{g-\beta}\frac{(\beta+q)}{2\beta-g-q}.$$
\item[{\it (ii)}] if $g>\max(\beta,q)$ and $\alpha>0$,
\begin{align*}
\mathbb{E}_{\delta_x}\left[N_t^2\right] \underset{t \to \infty}{\sim}\mathbb{E}_{\delta_x}^2\left[N_t\right] \underset{t \to \infty}{\sim} 
\left(\frac{\alpha x}{\beta-g}\right)^2e^{2(g-q)t},
\end{align*}
\end{itemize}
\end{lemma}
\begin{proof}
\al{From Proposition \ref{pro_exi_uni}, we have for $f\equiv 1$,
\begin{align}\label{eq:Nt}
N_t = N_0 + \int_{0}^{t}\int_{E}&\mathbf{1}_{\left\{u\in V_{s{-}}\right\}}\left(\mathbf{1}_{\left\{z\leq r(X_{s{-}}^u)\right\}}-\mathbf{1}_{\left\{0<z-r(X_{s{-}}^u)\leq q)\right\}}
 \right)M\left(ds,du,d\theta,dz\right)
\end{align}}
\cha{where we recall that $M(ds,du,d\theta,dz)$ is a PPM on $\mathbb{R}_+\times E$ with intensity 
$ds\otimes n(du)\otimes \kappa(d\theta)\otimes dz$, where $n(du)$ denotes the counting measure on $\mathcal{U}$.} It\^o's formula yields for $t\geq 0$,
\al{
\begin{align}\label{eq:Nt2}
N_t^2 = N_0^2 + \int_{0}^{t}\int_{E}&\mathbf{1}_{\left\{u\in V_{s{-}}\right\}}\left(\mathbf{1}_{\left\{z\leq r(X_{s{-}}^u)\right\}}(1+2N_t) + \mathbf{1}_{\left\{0<z-r(X_{s{-}}^u)\leq q)\right\}}(1-2N_t)
 \right)M\left(ds,du,d\theta,dz\right).
\end{align}}
Then, for all $t\geq 0$ and $x> 0$,
\begin{align*}
\frac{d}{dt}\mathbb{E}_{\delta_x}\left[N_t^2\right]& =  \mathbb{E}_{\delta_x}\left[\left(\alpha x e^{(g-q)t}+\beta N_t\right)\left(1+2N_t\right)\right]
+  \mathbb{E}_{\delta_x}\left[q N_t\left(1-2N_t\right)\right]\\
 & =\alpha xe^{(g-q)t}+2\alpha x e^{(g-q)t}\mathbb{E}_{\delta_x}[N_t]+(\beta+q)\mathbb{E}_{\delta_x}[N_t]+2(\beta-q) \mathbb{E}_{\delta_x}\left[N_t^2\right]. 
\end{align*}
Using \eqref{mxst} and variation of constants we obtain
\begin{align*}
\E_{\delta_x}\left[N_t^2\right] = & e^{2(\beta-q)t} +  \alpha x\frac{\left(e^{(g-q)t}-e^{2(\beta-q)t}\right)}{g-2\beta+q} + \frac{2\alpha^2 x^2}{g-\beta}\frac{\left(e^{2(g-q)t}-e^{2(\beta-q)t}\right)}{2(g-\beta)}\\
&+2\alpha x\left(1-\frac{\alpha x}{g-\beta}\right)\frac{\left(e^{(g+\beta-2q)t}-e^{2(\beta-q)t}\right)}{g-\beta}+(\beta+q)\frac{\alpha x}{g-\beta}\frac{\left(e^{(g-q)t}-e^{2(\beta-q)t}\right)}{g-2\beta+q}\\
& - \left(1-\frac{\alpha x}{g-\beta}\right)(\beta+q)\frac{\left(e^{(\beta-q)t}-e^{2(\beta-q)t}\right)}{\beta-q}.
\end{align*}
Therefore, 
$$
\begin{array}{ll}
\text{if }g>\max(\beta,q),& \E_{\delta_x}\left[N_t^2\right]\underset{t\rightarrow +\infty}{\sim}\frac{\alpha^2x^2}{(g-\beta)^2}e^{2(g-q)t}\\
\text{if }\beta>\max(g,q), & \E_{\delta_x}\left[N_t^2\right]\underset{t\rightarrow +\infty}{\sim}C_1^2(x)e^{2(\beta-q)t}.
\end{array}
$$
Moreover, 
\begin{align*}
\E_{\delta_x}\left[N_t\right]^2 = \frac{\alpha^2 x^2}{(g-\beta)^2}e^{2(g-q)t}+ \left(1-\frac{\alpha x}{g-\beta}\right)^2e^{2(\beta-q)t}-\frac{\alpha x }{\beta-g}\left(1-\frac{\alpha x}{g-\beta}\right)e^{g+\beta-2q},
\end{align*}
so that
$$
\begin{array}{ll}
\text{if }g>\max(\beta,q),& \E_{\delta_x}\left[N_t\right]^2\underset{t\rightarrow +\infty}{\sim}\frac{\alpha^2x^2}{(g-\beta)^2}e^{2(g-q)t}\\
\text{if }\beta>\max(g,q), & \E_{\delta_x}\left[N_t\right]^2\underset{t\rightarrow +\infty}{\sim}\left( 1+\frac{\alpha x }{\beta-g}\right)^2e^{2(\beta-q)t}.
\end{array}
$$
\end{proof}
\al{{\bf Adaptation of Lemma \ref{lem_esp_Nt2} for the running example: Case $\alpha_q>\alpha$, $\alpha_g<\beta$, $\beta>\beta_q$.}\label{par:adaptationNt}
First, from Proposition \ref{pro_exi_uni},
\begin{align*}
\frac{d}{dt}\E_{\delta_x}\left[\sum_{u\in V_t}X_t^u\right] =\alpha_g \E_{\delta_x}\left[\sum_{u\in V_t}X_t^u\right] - \E_{\delta_x}\left[\sum_{u\in V_t}q(X_t^u)X_t^u\right]\leq (\alpha_g-\beta_q)\E_{\delta_x}\left[\sum_{u\in V_t}X_t^u\right].
\end{align*}
By \eqref{eq:Nt} (with $q(X_{s-}^u)$ instead of $q$), we have 
\begin{align*}
\frac{d}{dt}\E_{\delta_x}[N_t] &= (\alpha-\alpha_q)\E_{\delta_x}\left[\sum_{u\in V_t}X_t^u\right] + (\beta-\beta_q)\E_{\delta_x}[N_t].
\end{align*}
Combining the two previous equations and using that $\alpha\leq \alpha_g$, we obtain
\begin{align*}
(\alpha - \alpha_q)xe^{(\alpha_g-\beta_q)t}+ (\beta-\beta_q)\E_{\delta_x}[N_t]\leq \frac{d}{dt}\E_{\delta_x}[N_t]\leq (\beta-\beta_q)\E_{\delta_x}[N_t],
\end{align*} 
hence
\begin{align*}
e^{(\beta-\beta_q)t}\left(1  + \frac{\alpha-\alpha_q}{\alpha_g-\beta}\left(e^{(\alpha_g-\beta)t}-1\right) \right)\leq \E_{\delta_x}[N_t]\leq e^{(\beta-\beta_q)t}.
\end{align*}
Next, by \eqref{eq:Nt2} (with $q(X_{s-}^u)$ instead of $q$), we have
\begin{align*}
\frac{d}{dt}\E_{\delta_x}[N_t^2]  =& (\alpha + \alpha_q)\E_{\delta_x}\left[\sum_{u\in V_t}X_t^u\right] + (\beta+\beta_q)\E_{\delta_x}[N_t] + 2(\alpha-\alpha_q)\E_{\delta_x}\left[N_t\sum_{u\in V_t}X_t^u\right]\\
& + 2 (\beta-\beta_q)\E_{\delta_x}[N_t^2]\\
  \leq & (\alpha+\alpha_q)xe^{(\alpha_g-\beta_q)t} + (\beta+\beta_q)e^{(\beta-\beta_q)t} + 2 (\beta-\beta_q)\E_{\delta_x}[N_t^2],
\end{align*}
which implies
\begin{align*}
\E_{\delta_x}[N_t^2]\leq e^{2(\beta-\beta_q)t} + \frac{\alpha+\alpha_q}{\alpha_g+\beta_q-2\beta}x\left(e^{(\alpha_g-\beta_q)t}-e^{2(\beta-\beta_q)t}\right)+\frac{\beta+\beta_q}{\beta_q-\beta}x\left(e^{(\beta-\beta_q)t}-e^{2(\beta-\beta_q)t}\right).
\end{align*}
The previous computations show the existence of $C<\infty$ such that
\begin{align*}
\E_{\delta_x}[N_t]^2\geq e^{2(\beta-\beta_q)t}\left(1  + \frac{\alpha-\alpha_q}{\alpha_g-\beta}\left(e^{(\alpha_g-\beta)t}-1\right) \right)^2,
 \E_{\delta_x}[N_t^2]\leq Ce^{2(\beta-\beta_q)t} + o(e^{2(\beta-\beta_q)t}).
\end{align*}
Then, as $\alpha_g<\beta$, $\E_{\delta_x}[N_t^2]/\E_{\delta_x}[N_t]^2$ is bounded for $t$ large enough.
}
\begin{proof}[Proof of Proposition \ref{prop:containment}]
Let $a>0$, $V_1(x)=x$ for all $x\geq 0$ and recall the definition of $\mathcal{B}$ in \eqref{defmathcalB}. As in the proof of Proposition \ref{prop:gx-q}, using that $\mathcal{B}V_1(x)=(g-q)V_1(x)$, if we consider the process $\mathcal{Y}$ with infinitesimal generator \eqref{infgen_mathcalY}, we get from \cite[Lemma 3.3]{cloez2017limit} the following Many-to-One formulae, when considering the functions $f(x)=x\mathbf{1}_{x \in [a,a+da]}$ and $f(x) \equiv 1$,
\begin{align*}
a\E_{\delta_x}\left[\sum_{u\in V_t} \mathbf{1}_{X_t^u \in [a,a+da]} \right]= xe^{(g-q) t}\P_x\left(\mathcal{Y} \in [a,a+da] \right)
\quad
\text{and}
\quad
\E_{\delta_x}\left[\sum_{u\in V_t}1\right]= xe^{(g-q) t}\E_x\left[ \mathcal{Y}^{-1}_t\right].
\end{align*}
It concludes the proof of point $i)$.

We now prove point $ii)$. First \ref{eq:SNinfinity-1} is satisfied because of the linear division rate. Next \eqref{eq:limsup_sigma} implies \eqref{star3}. Then, Proposition \ref{prop:gx-q}$ii)$ and Markov's inequality yield for any $\eps>0$,
\begin{align*}
\lim_{K\rightarrow \infty}\lim_{t\rightarrow \infty} \mathbb{P}_{\delta_x}
\left( \mathbf{1}_{\{N_t \geq 1\}} \frac{\#\lbrace u\in V_t: X_t^u >K\rbrace}{\E_{\delta_x}[N_t]}>\eps \right) = 0.
\end{align*}
To prove point {\it ii)},
it is thus enough to prove that
\begin{align} \label{to_prove}
\lim_{K\rightarrow \infty}\lim_{t\rightarrow \infty}\text{Diff}(K,t,\eps)  = 0,
\end{align}
where 
$$ \text{Diff}(K,t,\eps):= \mathbb{P}_{\delta_x}
\left( \mathbf{1}_{\{N_t \geq 1\}} \left|\frac{\#\lbrace u\in V_t: X_t^u >K\rbrace}{\E_{\delta_x}[N_t]}-\frac{\#\lbrace u\in V_t: X_t^u >K\rbrace}{N_t}\right|>\eps \right). $$
We use different strategies depending on whether  $g>\max(\beta,q)$ and $\alpha>0$ or  $\beta>\max(g,q)$.\\

\paragraph{\bf{Case $g>\max(\beta,q)$ and $\alpha>0$}} Applying Markov's and Cauchy-Schwarz inequalities,
\begin{align*}
\text{Diff}(K,t,\eps) \leq  \frac{1}{\eps \E_{\delta_x}[N_t] } \mathbb{E}_{\delta_x}\left[\mathbf{1}_{\{N_t \geq 1\}} \frac{\#\lbrace u\in V_t: X_t^u >K\rbrace}{N_t}|N_t-\E[N_t]|\right]
\leq  \frac{\sqrt{\text{Var}_{\delta_x}(N_t)}}{\eps \E_{\delta_x}[N_t] } 
\end{align*}
which goes to $0$ as $t$ goes to $\infty$ according to Lemma \ref{lem_esp_Nt2} and \eqref{to_prove} holds.\\

\paragraph{\bf{Case $\beta>\max(g,q)$}} Notice that in this case, from the proof of Proposition \ref{prop:gx-q} Many-to-One formula with the function $f(x)=x$ writes
$ \E_{\delta_x}[P_t] = x e^{(g-q)t} $, where $P_t:=\sum_{u\in V_t} X_t^u$, and from Lemma \ref{lem_esp_Nt2},
$$ \mathbb{E}_{\delta_x}\left[N_t\right] \underset{t \to \infty}{\sim} \left(1 + \frac{\alpha x}{\beta-g}\right)e^{(\beta-q)t}. $$
Hence 
$\E_{\delta_x}[P_t] \ll  \mathbb{E}_{\delta_x}\left[N_t\right]$ for large $t$.
Let us introduce, for $\eps,t,K>0$, the event 
$$ A_{\eps,t,K}:= \left\{\mathbf{1}_{\{N_{t}\geq 1\}}\frac{\#\lbrace u\in V_t: X_t^u >K\rbrace}{N_t}>\eps \right\}.  $$
We will make a reductio ad absurdum and assume that point $ii)$ of Proposition \ref{prop:containment} does not hold. As a consequence, there exists $\eps_0,a_0>0$ such that 
\begin{align} \label{ass_absurde}
\limsup_{K\rightarrow \infty}\limsup_{t\rightarrow \infty} \mathbb{P}_{\delta_x}
\left( A_{\eps_0,t,K} \right) \geq a_0.
\end{align}
But for any $K,t>0$, we have
\begin{align}
\E_{\delta_x}[P_t]  \geq K  \E_{\delta_x}\left[\mathbf{1}_{\{N_t \geq 1\}}\frac{\sum_{u\in V_t} \mathbf{1}_{\{X_t^u>K\}}}{N_t}N_t\right] 
 \geq  K\eps_0 \E_{\delta_x}[N_t\mathbf{1}_{A_{\eps_0,t,K}}].\label{min_Pt}
\end{align}
The process $(N_t, t \geq 0)$ has the same law as a supercritical branching process, with individual birth  and death rates $\beta$ and $q$, and a time inhomogeneous immigration (with rate $\alpha P_t$).
Then,
\begin{enumerate}
\item either there exists $T_0<\infty$ such that there is no more immigration after time $T_0$. In this case, either $N_t$ goes to $0$ at infinity, or $N_te^{-(\beta- q)t}$ converges almost surely when $t$ goes to infinity to a positive random variable $W_0$ (\cite[p.112]{athreya1972branching}).
\item or there is an infinite number of migrants, and as every birth-death process has probability $1-q/\beta>0$ to survive (\cite[p.109]{athreya1972branching}), there is a time $T_1<\infty$  such that the immigrant arriving at time $T_1$ has an infinite line of descent, and thus $N_{t-T_1}e^{-(\beta- q)(t-T_1)}$ is larger than a process which converges almost surely when $t$ goes to infinity to a positive random variable $W_0$. 
\end{enumerate}
From this analysis, we deduce that 
$
\liminf_{t\to \infty}\E\left[N_te^{-(\beta-q)t}\mathbf{1}_{\{N_t\geq 1\}}\right]> 0.
$
According to \eqref{ass_absurde}, there are sequences $(K_n,n\in \N)$, $(t_n, n \in \N)$ going to $\infty$ and $n_0 \in \N$ such that $n\geq n_0$ implies
$ \mathbb{P}_{\delta_x}
\left( A_{\eps_0,t_n,K_n} \right) \geq a_0/2.$
As for all $n\geq 0$, $A_{\eps_0,t_n,K_n}\subseteq \{N_t\geq 1\}$, there exists $C>0$ such that for all $n\geq n_0$
\begin{align*}
\E_{\delta_x}[N_{t_n}e^{-(\beta- q)t_n}\mathbf{1}_{A_{\eps_0,t_n,K_n}}]=\E_{\delta_x}[N_{t_n}e^{-(\beta- q)t_n}\mathbf{1}_{\{N_{t_n}\geq 1\}}]\P(A_{\eps_0,t_n,K_n})\geq Ca_0/2.
\end{align*}

We deduce, using \eqref{min_Pt}, that for any $n \geq n_0$, 
$$xe^{(g-q)t_n}= \E_{\delta_x}[P_{t_n}]\geq K_n  \eps_0 \E_{\delta_x}[N_{t_n}\mathbf{1}_{A_{\eps_0,t_n,K_n}}]\geq C\frac{a_0}{2} K_n  \eps_0 e^{(\beta- q)t_n},$$
which is absurd because $\beta>g$. It concludes the proof.
\end{proof}

\subsubsection{Preliminary results on the auxiliary process}

Recall that the auxiliary process $(Y_s^{(t)},s\leq t)$ is well-defined as the unique strong solution to \eqref{eq:EDSauxi} under \ref{ass_E}. In what follows, we set $Y_s^{(t)} =Y_t^{(t)}$ for all $s\geq t$ and $f_3(x,u,z)= 0$ if $u \leq 0$, for all $x,z\geq 0$. 

The next proposition is an analogue of the auxiliary process of \cite[Theorem 3.3]{companion} 
Let
$$
\tau_t^-(0):= \inf\lbrace 0<s\leq t : Y_s^{(t)} = 0\rbrace.
$$
\begin{prop}\label{prop_extin_auxi}
Suppose that Assumptions \ref{ass_E} holds, and that \al{$\int_0^\infty z \pi(dz)<\infty$}.
\begin{itemize}
\item[i)] 
If Condition \ref{A1} holds, then $\mathbb{P}_x\left(\tau_t^-(0)<\infty\right)=0$ for all $x>0$.
\item[ii)]
If Condition \ref{A2} holds, then for any $x>0$ and $s>0$, 
$\mathbb{P}_x\left(\tau_t^-(0)<s\right) >0$.
\end{itemize}
\end{prop}
\begin{proof}[Proof of Proposition \ref{prop_extin_auxi}]
 This proof is very similar to the proof of \cite[Theorem 3.3]{companion}. The only modifications are due to the time-inhomogeneity, and to the fact that the time interval is restricted to $[0,t]$.
We proceed by coupling to overcome these two difficulties.
 
{\it i)} Introduce $\tY$ the unique strong solution to
\begin{align}\nonumber\label{eq:tildeY}
\tY_t  = \tY_0 + g\int_0^t \tY_s ds+ \int_0^t \sqrt{2 \sigma^2(\tY_s) }dB_s 
+ \int_0^t \int_{\mathbb{R}_+}\int_0^{f_3(\tY_{s^-},t-s,z)}z\widetilde{Q}(ds,dz,dx)
\\ + \int_0^t \int_0^{2(\alpha \tY_{s^-}+\beta)} \int_0^1  (\theta-1)\tY_{s^-}N(ds,dz,d\theta),
\end{align}
where $B, \widetilde{Q}, N$ are the same as in \eqref{eq:EDSauxi}, $f_3$ is defined in \eqref{eq:f3}. 
Notice that 
$$f_1(y,s)\geq gy, \quad f_2(y,s,\theta)\leq 2(\alpha y+\beta),\quad \forall y,z,t\geq 0,\ s\in [0,t],\ \theta\in [0,1].$$
and $f_3(x,t-s,z)$ is non-decreasing in $x$ because $xp'(x)\geq x$ for all $x\geq 0$.
In particular this implies that if $\tY$ is a solution with $\tY_0= Y_0^{(t)}$, then $\tY_s\leq Y_s^{(t)}$ for any $s$ smaller than $t$.
But $\tY$ satisfies the assumptions of a modified version of point {\it i)} of \cite[Theorem 3.3]{companion} where the rate of positive jumps depends on time (see Appendix \ref{app:th33withtimeinjumps}). Therefore, applying this result, we proved that $\tY$ does not reach $0$ in finite time, so that $Y^{(t)}$ 
does not reach $0$ before time $t$.\\

{\it ii)}
First notice that 
for any $x> 0$ and $s \leq t$, the function 
$f_1$ defined in \eqref{eq:f1} satisfies
$$ f_1(x,s) \leq gx + \left(2\sigma^2(x) + p(x)\int_{\R_+}z^2\pi(dz)\right) \frac{A_t}{1+x A_t}=:\bar{g}_t(x), $$
where $A_t = \alpha(e^{(g-\beta)t}-1)/(g-\beta)$.
Let
 $(\bar{Y}_s,s\geq 0)$ be the unique strong solution to
\begin{align}\nonumber\label{eq:barY}
\bar{Y}_s = & Y_0^{(t)}+ \int_0^s\bar{g}_t\left(\bar{Y}_u\right)du + \int_0^s\sqrt{2\sigma^2\left(\bar{Y}_u\right)}dB_u + 
\int_0^s\int_0^1\int_{0}^{\bar{r}(\theta)}(\theta-1)\bar{Y}_{u^-} N(du,d\theta,dz)\\
& + \int_0^s \int_0^\infty\int_0^{f_3(\bar{Y}_{s^-},t-u,z)}z\widetilde{Q}(du,dz,dx),
\end{align}
where for all $x,s\geq 0$ and $\theta\in [0,1]$,
$
\bar{r}(\theta) := 2\theta\beta\leq f_2(x,s,\theta),
$
with $f_2$ defined in \eqref{eq:f2}, $f_3$ in \eqref{eq:f3}, 
$B$, $N$ and $\widetilde{Q}$ are the same as in \eqref{eq:EDSauxi}.
Then, for all $0\leq s \leq t$, $\bar{Y}_s \geq Y_s^{(t)}.$ Then, for
$
\bar{\tau}^-(0):= \inf\lbrace s\geq 0 : \bar{Y}_s = 0\rbrace,
$
if we prove that
\begin{equation} \label{bartau0} 
\P(\bar{\tau}^-(0)<v)>0,\quad  \forall 0<v\leq t, 
\end{equation}
it will imply that
$ \P(\tau_t^-(0)<v) \geq \P(\bar{\tau}^-(0)<v)>0$
and end the proof.
To prove \eqref{bartau0}, we apply \cite[Theorem 3.3{\it iii)}]{companion} to the process $\bar{Y}$. Notice that here, unlike in \cite[Theorem 3.3]{companion}, the division rate $\bar{r}$ depends on $\theta$. 
However, the dependence in $\theta$ in the division rate can be removed by considering a new PPM $N'$ with a modified fragmentation kernel so that all the results derived above still hold. We refer the reader to Appendix \ref{app:generalization SDE} for more details, \al{and to Appendix \ref{app:th33withtimeinjumps} for the generalization of \cite[Theorem 3.3{\it iii)}]{companion} to processes with rate of positive jumps depending on time.}
\end{proof}
In the case where the absorption of the auxiliary process occurs with positive probability, 
we prove the convergence of the auxiliary process trajectory on a time window of any size. 
\begin{prop}\label{prop:conv_auxi}
Let $T\geq 0$. Suppose that Assumption \ref{ass_E1} and \ref{A2} hold, and that \al{$\int_0^\infty z \pi(dz)$}.
Then, there exist $C,\overline{c}>0$ 
and a probability measure $\Pi$ on the Borel $\sigma$-field of $\mathbb{D}\left([0,T],\mathcal{X}\right)$ endowed with the Skorokhod distance 
such that for all bounded measurable functions $F:\mathbb{D}\left([0,T],\mathcal{X}\right)\rightarrow \mathbb{R}$,
\begin{align*}
\left|\mathbb{E}\left[F\left(Y_{t+s}^{(t+T)},s\leq T\right)\Big|Y_{0}^{(t+T)} = x \right]- \Pi(F)\right|\leq C e^{-\overline{c}t}\left\Vert F\right\Vert_{\infty}(1 + x),\ \forall x\geq 0.
\end{align*}
\end{prop}

We prove the convergence of the auxiliary process by verifying a Foster-Lyapunov inequality and a minoration condition, both stated in Lemma \ref{lemma:Lyapunov_min} below. Those standard conditions were exhibited in \cite{marguet2017law} as an extension of \cite{hairer2011yet} to time-inhomogeneous processes. The Foster-Lyapunov inequality (Condition {\it i)} in Lemma \ref{lemma:Lyapunov_min}) ensures that
$$
\E_x\left[V\left(Y_s^{(t)}\right)\right]\leq e^{-as}V(x)+\frac{d}{a}\left(1-e^{-as}\right),
$$
where $a$ and $d$ are positive constants,
so that the process is brought back to the sublevel sets of $V$. The minoration condition ({\it i.e.} Condition {\it ii)} in Lemma \ref{lemma:Lyapunov_min}) ensures some type of irreducibility of the process on those sublevel sets.
Let $V(x)=x$  for $x\in\mathbb{R}_+$. \al{Recall that $\mathcal{A}_s^{(t)}$ is defined in \eqref{eq:gene_auxi}, $f_1$ in \eqref{eq:f1}, and $f_2$ in \eqref{eq:f2}.}
\begin{lemma}\label{lemma:Lyapunov_min}Under the assumptions of Proposition \ref{prop:conv_auxi}, we have the following:
\begin{enumerate}[label=\roman*)]
\item There exist $a,d>0$ such that for all $0\leq s\leq t$ and $x\in\mathbb{R}_+^{*}$,
\begin{align*}
\mathcal{A}_s^{(t)}V(x)\leq -aV(x)+d.
\end{align*}
\item There exists $R>2da^{-1}$ such that for all $r<s\leq  t$, there exist $\alpha_{s-r}>0$ and a probability measure $\nu$ on $\mathbb{R}_+$ 
such that for all Borel sets $A$ of $\mathbb{R}_+$,
\begin{align*}
\inf_{x\leq R}\mathbb{P}\left(Y_{s}^{(t)}\in A\big| Y_r^{(t)}=x\right)\geq \alpha_{s-r} \nu(A).
\end{align*}
\end{enumerate}
\end{lemma}
\begin{proof}{\it{i)}} 
We have
\begin{align*}
\mathcal{A}_s^{(t)}V(x) & = f_1(x,t-s)-\int_0^1f_2(x,t-s,\theta)x(1-\theta)\kappa(d\theta)\\
& \leq V(x)\left(g+ 2\frac{\sigma(x)^2}{x^2}+\frac{p(x)}{x^2}\int_{\R_+}z^2\pi(dz) -2\alpha x \E\left[\Theta(1-\Theta)\right]\right).
\end{align*}
According to \eqref{eq:limsup_sigma}, there exist $A>0$ and $a>0$ such that for all $x>A$
\begin{align*}
2\frac{\sigma(x)^2}{x^2}+\frac{p(x)}{x^2}\int_{\R_+}z^2\pi(dz) -2\alpha x \E\left[\Theta(1-\Theta)\right]<-(a+g).
\end{align*}
Then, 
\begin{align*}
\mathcal{A}_s^{(t)}V(x) \leq -ax+ \mathbf{1}_{\lbrace x\leq A\rbrace}\left(2\frac{\sigma(x)^2}{x}+\frac{p(x)}{x}\int_{\R_+}z^2\pi(dz) -2\alpha x^2 \E\left[\Theta(1-\Theta)\right]\right),
\end{align*}
and according to Assumption \ref{ass_A}, there exists $d>0$ such that for every $x \geq 0$,
\begin{align*}
\mathcal{A}_s^{(t)}V(x) & \leq -ax+d.
\end{align*}
{\it ii)} Let $R>2da^{-1},$ where $a,d$ are given in {\it i)}. We will prove the minoration condition with $\nu = \delta_0$, where $\delta_0$ is the Dirac measure at $0$. Consider again $\bar{Y}$, defined as 
the unique strong solution to the SDE \eqref{eq:barY}. We recall that $\bar{Y_s}\geq Y_s^{(t)}$, for all $s\leq t$. Therefore for all $r<s\leq t$ and all Borel sets $A$ of $\mathbb{R}_+$,
\begin{align*}
\mathbb{P}\left(Y_s^{(t)}\in A\big|Y_r^{(t)}=x\right) \geq \mathbb{P}\left(Y_s^{(t)}=0\big|Y_r^{(t)} = x\right)\delta_0(A) \geq \mathbb{P}\left(\bar{Y}_s=0\big|\bar{Y}_r = x\right)\delta_0(A).
\end{align*} 
Next, notice that if $\bar{Y}^1, \bar{Y}^2$ are two solutions to \eqref{eq:barY} with respective initial conditions at time $r$ satisfying 
$\bar{Y}^1_r\leq \bar{Y}^2_r$, then $\bar{Y}^1_s\leq \bar{Y}^2_s$ for all $r\leq s\leq t$. Hence, for all $x\leq R$,
\begin{align*}
 \mathbb{P}\left(\bar{Y}_s=0\big|\bar{Y}_r = x\right)=\mathbb{P}\left(\bar{Y}_{s-r}=0\big|\bar{Y}_0 = x\right)\geq 
 \mathbb{P}\left(\bar{Y}_{s-r}=0\big|\bar{Y}_0 = R\right).
\end{align*}
Finally, as \ref{A2} holds, using \cite[Theorem 3.3{\it iii)}]{companion} on $\bar{Y}$, there exists $\alpha_{s-r}>0$ such that 
$$
\mathbb{P}\left(\bar{Y}_{s-r}=0\big|\bar{Y}_0 = R\right)>\alpha_{s-r},
$$
which ends the proof.\end{proof}
\begin{proof}[Proof of Proposition \ref{prop:conv_auxi}] 
This result is a direct application of \cite[Proposition 3.3]{marguet2017law}. The main assumptions are satisfied thanks to Lemma \ref{lemma:Lyapunov_min}. Note that using the expression of $m$ given in \eqref{mxst}, \cite[Assumption 2.4]{marguet2017law} is satisfied because
$\int_0^1\frac{m(\theta x,s,t)}{m(x,s,t)}\kappa(d\theta)\leq 1/2.
$\end{proof}
This convergence result allows us to establish a law of large numbers, linking asymptotically the behaviour of a 
typical individual with the behaviour of the whole population.

\subsubsection{Proof of Theorem \ref{thm:conv_auxi}}

It is a direct application of \cite[Corollary 3.7]{marguet2017law}. Assumptions 2.1, 2.3 and 2.4 in \cite{marguet2017law} are 
satisfied thanks to Assumption \ref{ass_A}, using \eqref{mxtbis} and \eqref{mxst}, and the fact that $\beta>0$. 
We proved that Assumption 3.1 in \cite{marguet2017law} is verified in Lemma \ref{lemma:Lyapunov_min}. It remains to check that Assumptions 3.4 and 3.6 in \cite{marguet2017law} are satisfied. Note that in our case, the function $c(x)$ defined in \cite[Equation 3.3]{marguet2017law} 
is equal to $\max(g,\beta)-q$ and the first point of Assumption 3.4 in \cite{marguet2017law} is satisfied.

Next, we set some notations, introduced in \cite{marguet2017law}. For all $x,y\geq 0$ and $s\geq 0$, we define
\begin{equation*}
\varphi_{s}(x,y)=\sup_{t\geq s} \frac{m(x,0,s)m(y,s,t)}{m(x,0,t)},
\end{equation*}
(which does not depend on $q$)
and for all measurable functions $f:\mathbb{R}_+\rightarrow\mathbb{R}$ and $x\geq 0$,
\begin{align*}
Jf(x)=2\int_{0}^1f\left(\theta x\right)f\left((1-\theta)x\right)\kappa(d\theta).
\end{align*}
The next lemma amounts to check the second point of Assumption 3.4 in \cite{marguet2017law}. 
\begin{lemma}
Under Condition \ref{ass_E2}, then for all $x\geq 0$,
\begin{equation*}
\sup_{t\geq 0}\E_{x}\left[r\left(Y_{t}^{(t)}\right)J\left((1\vee V(\cdot))\varphi_{t}\left(x,\cdot\right)\right)\left(Y_{t}^{(t)}\right)\right]<\infty.
\end{equation*}
\end{lemma}
\begin{proof}
Note that if $x=0$, $Y_t^{(t)}=0$ almost surely for all $t\geq 0$. Therefore, we only need to consider $x>0$. 
First assume that $\alpha>0$. Notice that for all $t\geq 0$ $x>0$, and $y\geq 0$, 
$$
\varphi_t(x,y)\leq \left(1+\frac{\alpha x}{|g-\beta|}\right)\left(1+\frac{\alpha y}{|g-\beta|}\right)\left(\min\left(\frac{\alpha x}{|g-\beta|},1\right)\right)^{-1},
$$
where we simplified by $e^{\max(g,\beta)t}$ in the fraction in the definition of $\varphi_s$.
Next, for all $x>0$, 
\begin{align*}
& \E_{x}\left[r\left(Y_{t}^{(t)}\right)J\left((1\vee V(\cdot))\varphi_{t}\left(x,\cdot\right)\right)\left(Y_{t}^{(t)}\right)\right]\\
& \leq \left(\frac{|g-\beta|+\alpha x}{\min\left(\alpha x,|g-\beta|\right)}\right)^2\E_{x}\left[\left(\alpha Y_{t}^{(t)}+\beta\right)
2\int_{0}^1\left(1\vee\theta Y_{t}^{(t)}\right)\left(1\vee(1-\theta) Y_{t}^{(t)}\right)\left(1+\frac{\alpha Y_{t}^{(t)}}{|g-\beta|}\right)^2\kappa(d\theta)\right]. 
\end{align*}
For all $k\geq 0$, we define
\begin{align*}
f_k^{(t)}(x,s)=\mathbb{E}_x\left[\left(Y_s^{(t)}\right)^k\right]
\end{align*}
and we end the proof of the lemma by showing that,
 $\sup_{t\geq 0}\sup_{s\leq t}f_5^{(t)}(x,s)<\infty.$
According to It\^o's formula, we have for $k\geq 2$,
\begin{align*}
f_k^{(t)}(x,s) = &  k\int_0^s \mathbb{E}_x\left[\left(Y_u^{(t)}\right)^{k-1}f_1\left(Y_u^{(t)},t-u\right)\right]du + k(k-1)\int_0^s \mathbb{E}_x\left[\sigma^2\left(Y_u^{(t)}\right)\left(Y_u^{(t)}\right)^{k-2}\right]du\\
& +\int_0^s \int_{\mathbb{R}_+}\mathbb{E}_x\left[f_2\left(Y_u^{(t)},t-u,\theta\right)\left(Y_u^{(t)}\right)^k(\theta^k-1)\right]\kappa(d\theta)du\\
& +\int_0^s \int_{\mathbb{R}_+}\mathbb{E}_x\left[f_3\left(Y_u^{(t)},t-u,z\right)\left(\left(Y_u^{(t)}+z\right)^k-\left(Y_u^{(t)}\right)^k-kz\left(Y_u^{(t)}\right)^{k-1}\right)\right]\pi(dz)du.
\end{align*}
Differentiating with respect to $s$ and using that for all $x,s,z\geq 0$ and $\theta\in[0,1]$, $f_2(x,s,\theta)\geq 2\theta\alpha x$, 
$f_3(x,s,z)\leq (x+z)p(x)/x$, and $x^{k-1}f_1(x,s)\leq gx^k+2\sigma(x)^2x^{k-2}+p(x)x^{k-2}\int_{\R_+}z^2\pi(dz)$.

Let $ H(k,y) = (k+1)\sigma^2(y)y^{-2}+kp(y)y^{-2}\int_{\R_+}z^2\pi(dz)$. Applying Taylor's formula with integral remainder, we obtain 
\begin{align*}
\partial_s f_k^{(t)}(x,s)\leq &  gk \mathbb{E}_x\left[\left(Y_s^{(t)}\right)^k\right]
 + k\mathbb{E}_x\left[H(k,Y_s^{(t)})(Y_s^{(t)})^{k}\right] - \int_{\mathbb{R}_+}\mathbb{E}_x\left[2\alpha \left(Y_s^{(t)}\right)^{k+1}\theta(1-\theta^k)\right]\kappa(d\theta)\\
& +k(k-1) \int_{\mathbb{R}_+}\int_0^z(z-u)\mathbb{E}_x\left[\frac{p\left(Y_s^{(t)}\right)}{Y_s^{(t)}}\left(Y_s^{(t)}+u\right)^{k-2}\left(Y_s^{(t)}+z\right)\right]du \pi(dz).
\end{align*}
Moreover, for all $y> 0$,
\begin{align*}
&\int_{\mathbb{R}_+}\int_0^z(z-u)\frac{p\left(y\right)}{y}\left(y+u\right)^{k-2}\left(y+z\right)du \pi(dz) \\
& \leq  \int_{\mathbb{R}_+}z^2\frac{p(y)}{y}\left(y+z\right)^{k-1} \pi(dz)
=  \frac{p(y)}{y}\int_{\mathbb{R}_+}z^2\sum_{l = 0}^{k-1}\dbinom{k-1}{l}y^lz^{k-1-l} \pi(dz)\\
& = \frac{p(y)}{y}\int_{\mathbb{R}_+}z^2\sum_{l = 0}^{k-2}\dbinom{k-1}{l}y^lz^{k-1-l} \pi(dz)+\frac{p(y)}{y^2}\left(\int_{\R_+}z^2\pi(dz)\right)y^{k}.
\end{align*}
Combining the last two inequalities, we get 
\begin{align*}
\partial_sf_k^{(t)}(x,s)\leq & k\left( A_t^{(k)}+B_t^{(k)}-C_t^{(k)}+D_t^{(k)}\right),
\end{align*}
with
$$ A_t^{(k)} = gf_k^{(t)}(x,s),\ 
B_t^{(k)} = \mathbb{E}_x\left[H(k,Y_s^{(t)})(Y_s^{(t)})^{k}\right],\ C_t^{(k)} =\frac{2\alpha}{k} \mathbb{E}\left[\Theta(1-\Theta^k)\right]f_{k+1}^{(t)}(x,s),
$$
$$D_t^{(k)} =(k-1) \int_{\mathbb{R}_+}\mathbb{E}_x\left[\frac{p(Y_s^{(t)})}{\left(Y_s^{(t)}\right)^2}\sum_{l = 0}^{k-2}\dbinom{k-1}{l}(Y_s^{(t)})^{l+1}z^{k+1-l}\right] \pi(dz).
$$
To end the proof we consider the case $k=5$.
According to \eqref{eq:limsup_sigma} and using that $\sigma$ and $p$ are continuous (Assumption \ref{ass_A}), there exist $C_1,C_2, A>0$ such that for all $y\geq 0$,
\begin{align*}
H(5,y)y^{5} = H(5,y)y^5\mathbf{1}_{\lbrace y>A\rbrace}+H(5,y)y^5\mathbf{1}_{\lbrace y\leq A\rbrace}&\leq C_1y^5\mathbf{1}_{\lbrace y>A\rbrace}
+C_2\mathbf{1}_{\lbrace y\leq A\rbrace}\leq C_1y^5+C_2.
\end{align*}
Moreover, $\limsup_{0^+}p(x)/x<\infty$ as $p(0)=0$ and $p$ is locally Lipschitz, $\limsup_\infty p(x)/x^2<\infty$ thanks to \eqref{eq:limsup_sigma} 
and $\int_{\R_+}z^6\pi(dz)<\infty$, which yields
\begin{align*}
D_t ^{(5)}\leq C_3  (f_5^{(t)}(x,s) + 1),
\end{align*}
for some $C_3\geq 0$. Combining the last two inequalities, there exists $D_1>0$ such that
\begin{align*}
\partial_sf_5^{(t)}(x,s)\leq  D_1 (f_5^{(t)}(x,s)+1)-D_2f_{6}^{(t)}(x,s),
\end{align*}
where $D_2 = \frac{2\alpha}{5} \mathbb{E}\left[\Theta(1-\Theta^5)\right]$.
Applying Jensen inequality, we have $f_6^{(t)}(x,s)\geq f_5^{(t)}(x,s)^{6/5}$. Finally, we obtain
\begin{align*}
\partial_sf_5^{(t)}(x,s)\leq  F\left(f_5^{(t)}(x,s)\right),
\end{align*}
with $F(y) = D_1(y+1)-D_2y^{1+1/5}$. Any solution to the equation $y'=F(y)$ is bounded by $y(0)\vee x_0$, where $x_0 = \left(5D_1/6D_2\right)^5$
and so is $f_5^{(t)}(x,\cdot)$. It ends the proof for this case.
\end{proof}
Adding Lemma \ref{lem_esp_Nt2} (which corresponds to Assumption 3.6 in \cite{marguet2017law}.) we have all required assumptions to apply \cite[Corollary 3.4]{marguet2017law}. This ends the proof of Theorem \ref{thm:conv_auxi}.

\subsubsection{Proof of Proposition \ref{prop:temps_long_auxi}}
 
{\it i)} The first step consists in proving that for every $x\geq 0$,
\begin{equation} \label{abs_finitetime_Yt}
 \P_x(\exists t<\infty, Y^{(t)}_t=0)=1. 
\end{equation}
A direct application of \cite[Theorem 6.2]{companion} is not possible 
because of the time-inhomogeneity of $Y^{(t)}$. Therefore, we couple $Y^{(t)}$
with a process $\hat{Y}$ defined as the unique strong solution to
\begin{align} \nonumber\label{eq:hatY}
\hat{Y}_s = & Y_0^{(t)}+ \int_0^s\hat{g}(\hat{Y}_u)du + \int_0^s\sqrt{2\sigma^2(\hat{Y}_u)}dB_u 
+ \int_0^s\int_0^1(\theta-1)\int_{0}^{\hat{r}(\hat{Y}_u,\theta)}\hat{Y}_u N(du,d\theta,dx)\\
& + \int_0^s \int_0^\infty\int_0^{f_3(\hat{Y}_u,t-s,z)}z\widetilde{Q}(du,dz,dx),
\end{align}
where $B$, $N$ and $\tilde{Q}$ are the same as in \eqref{eq:EDSauxi} and for $x,s\geq 0$, $0 \leq \theta\leq 1$,
\begin{align*}
&f_1(x,s)\leq \hat{g}(x): =
gx+ \left(\frac{\alpha\mathbf{1}_{\lbrace\beta>g\rbrace}}{\beta-g+ \alpha x}+\frac{\mathbf{1}_{\lbrace g>\beta\rbrace}}{x}\right)\left( 2\sigma^2(x)+ p(x)\int_{\R_+}z^2\pi(dz)\right)\\ \nonumber
&f_2(x,s,\theta)\geq \hat{r}(x,\theta):= 2\theta(\alpha x+\beta).
\end{align*}
Then, for all $t\geq 0$ and $0\leq s \leq t$, 
$
Y_s^{(t)} \leq \hat{Y}_s. 
$
In particular, for all $t\geq 0$,
\begin{align}\label{couplage_t}
Y_t^{(t)} \leq \hat{Y}_t. 
\end{align}
According to Lemma \ref{chgt-eds}, there exists a PPM $N'$ on $\mathbb{R}_+\times [0,1]\times\mathbb{R}_+$ with intensity 
$du\otimes\hat{\kappa}(d\theta)\otimes dx$ where $\hat{\kappa}(d\theta)=2\theta\kappa(d\theta)$, such that $\hat{Y}$ is also a strong pathwise solution to
\begin{multline*}
\hat{Y}_s =  Y_0^{(t)}+ \int_0^s\hat{g}\left(\hat{Y}_u\right)du + \int_0^s\sqrt{2\sigma^2\left(\hat{Y}_u\right)}dB_u 
+ \int_0^s\int_{0}^{r(\hat{Y}_u)}\int_0^1(\theta-1)\hat{Y}_u N'(du,dx,d\theta)\\
 + \int_0^s \int_0^\infty\int_0^{f_3(\hat{Y}_u,t-s,z)}z\widetilde{Q}(du,dz,dx),
\end{multline*}
where we used that $\int_0^12\theta\kappa(d\theta)=1$ as $\kappa$ is symmetrical with respect to $1/2$.
The jump rate $f_3$ depends on jump size and time, but we have  
$$f_3(x,s,z)\leq p(x)\left(1+ \alpha z \frac{(e^{(g-\beta)s}-1)}{g-\beta}\right)$$
so that (6.3) in \cite[Theorem 6.2]{companion} holds for $\hat{Y}$ (see Appendix \ref{app:th62withtimeinjumps}) if \ref{eq:SNinfinity-1}, \ref{A2} and {\bf (LSG)} are satisfied, where
\begin{enumerate}[label=\bf{(LSG)}]
\item \label{eq:LSG} There exist $\eta_1,x_1>0$ such that for all $x>x_1$
\begin{align*}
\hat{H}(x):=\frac{\hat{g}(x)}{x}-\frac{\sigma^2(x)}{x^2}+\hat{r}(x)\E\left[\ln\Theta\right]<-\eta_1.
\end{align*}
\end{enumerate} 
First, \ref{A2} is satisfied by assumption. Let us check that \ref{eq:LSG} holds. We have
\begin{align*}
\hat{H}(x) &\leq g + \left(\frac{\mathbf{1}_{\{\beta>g\}}\alpha x}{\beta-g + \alpha x} + \mathbf{1}_{\{g \geq\beta\}}\right)\left(2\frac{\sigma^2(x)}{x^2} +
\frac{p(x)}{x^2}\int_{\R_+}z^2\pi(dz)
\right)+2(\alpha x+\beta)\E\left[\Theta\ln \Theta \right].
\end{align*}
According to Assumptions \ref{ass_A} and \eqref{eq:limsup_sigma}, there exist $x_1,C>0$ such that for all $x>x_1$ 
\begin{align*}
\hat{H}(x) \leq g + C\left(\frac{\mathbf{1}_{\{\beta>g\}}\alpha x}{\beta-g + \alpha x} + \mathbf{1}_{\{g \geq\beta\}}\right)+2(\alpha x+\beta)\E\left[\Theta\ln \Theta \right],
\end{align*}
and as $\E\left[\Theta\ln \Theta \right]<0$, condition \ref{eq:LSG} is satisfied. Finally, we check that \ref{eq:SNinfinity-1} is satisfied for $\hat{Y}$. And it is the case according to \eqref{eq:limsup_sigma} as
\begin{align*}
\frac{\hat{g}(x)}{x}
= g  + \left(\mathbf{1}_{\{\beta>g\}}\frac{\alpha x}{\beta-g + \alpha x} + \mathbf{1}_{\{g>\beta\}}\right)\left(2\frac{\sigma^2(x)}{x^2} +
\frac{p(x)}{x^2}\int_{\R_+}z^2\pi(dz)
\right).
\end{align*}
Hence \cite[Eq.~(6.3)]{companion} holds for $\hat{Y}$, and \eqref{couplage_t} gives \eqref{abs_finitetime_Yt}.
Applying Theorem~\ref{thm:conv_auxi} to the function $F(X_{t+s}^u, s \leq T) = \mathbf{1}_{\{X_{t+T}^u>0\}}$ concludes the proof of point {\it i)}.\\

Let us now prove point {\it ii)}.
The convergence result stems directly from Theorem \ref{thm:conv_auxi}. For the bounds on the process $(Y_t^{(t)}, t\geq 0)$, the idea is to use the same couplings as before, and apply the generalization of \cite[Theorem 6.2(i)]{companion} proved in Appendix \ref{app:th62withtimeinjumps}.

First, we consider again $\hat{Y}$ defined as the unique strong solution to \eqref{eq:hatY}. Recall that for all $t\geq 0$ and all $0\leq s \leq t$, $Y_s^{(t)}\leq \hat{Y}_s$. To apply \cite[Theorem 6.2(i)]{companion}, we have to check that \ref{A1}, \ref{eq:SNinfinity-1} and {\bf (LSG)} are satisfied. By assumption, \ref{A1} holds for $\hat{Y}$. We refer the reader to the proof of {\it i)} for the two other conditions. According to \cite[Theorem 6.2(i)]{companion}, $\hat{Y}$ converges in law to $\hat{Y}_\infty$ as $t$ goes to infinity. 

Next, we consider again $\tY$, defined as the unique strong solution to \eqref{eq:tildeY}. Recall that for all $t\geq 0$ and $0\leq s\leq t$, $\tY_s\leq Y_s^{(t)}$. Let us check that \ref{A1}, \ref{eq:SNinfinity-1} and {\bf (LSG)} of \cite{companion} holds for $\tY$. Again, \ref{A1} holds by assumption. Then, \ref{eq:SNinfinity-1} holds combining \eqref{eq:limsup_sigma} with the fact that the division rate of $\tY$ is linear in $x$. Finally, let
$$\widetilde{H}(x): = g-\frac{\sigma^2(x)}{x^2} + 2(\alpha x + \beta)\E\left[\ln\Theta\right].$$
From \eqref{eq:limsup_sigma}, there exist $\eta,x_1>0$ such that for $x>x_1$, $\widetilde{H}(x)\leq -\eta$. Hence {\bf (LSG)} of \cite{companion} holds for $\tY$ and according to \cite[Theorem 6.2(i)]{companion}, $\tY$ converges in law to $\tY_\infty$ as $t$ goes to $\infty$.

Finally, we have for all $t\geq 0$
$
\mathfrak{Y}^{(1)}_t:=\tY_t\leq Y_t^{(t)}\leq \mathfrak{Y}^{(2)}_t:=\hat{Y}_t,
$
which yields the result.
\appendix

\section{Proof of Proposition \ref{pro_exi_uni}} \label{proof_ex_uni}

To prove that the SDE \eqref{X_sans_sauts} admits a unique nonnegative strong solution with generator $\mathcal{G}$ defined in \eqref{def_gene}, we apply \cite[Proposition 1]{palau2018branching}. The proof is the same as the proof of \cite[Proposition 2.1]{companion}, except that we have to take into account the extra stable term. To prove that the result still holds, it is enough to check that for any $n \in \N$, there exists a finite constant $A_n$ such that for any $0 \leq x,y \leq n$,
$$ \int_{\R_+} \left| \mathbf{1}_{\{ x \leq u \}}z \wedge n - \mathbf{1}_{\{ y \leq u \}}z \wedge n \right|\frac{dz}{z^{2+\beta}}du \leq A_n |x-y|,$$
where we recall that $\beta \in (-1,0)$. This is a consequence of the following series of equalities:
\begin{align*}
 \int \left| \left(\mathbf{1}_{\{ x \leq u \}} - \mathbf{1}_{\{ y \leq u \}}\right)z \wedge n \right|\frac{dzdu}{z^{2+\beta}} &= \int \left| x  - y \right| z \wedge n\frac{dzdu}{z^{2+\beta}}= |x-y| \left( \int_0^n \frac{dz}{z^{1+\beta}} +n \int_n^\infty \frac{dz}{z^{2+\beta}} \right).
\end{align*}
To prove that it gives the existence and uniqueness of the process at the cell population level, we apply \cite[Theorem 2.1]{marguet2016uniform}.

\al{\section{Detailed computation of the infinitesimal generator $\mathcal{A}_{V_1}$}\label{app:generator_auxi}
In the proof of Proposition \ref{prop:gx-q}, we apply \cite[Lemma 3.3]{cloez2017limit} to our branching process, where the dynamics of the cells between the jumps are given by \eqref{def_gene}, with $c_\mathfrak{b}=0$ (no stable jumps). Recall that $V_1(x)=x$. For all $x\geq 0$, $f\in\mathcal{C}^2_b(\mathbb{R}_+)$, we have
\begin{align*}
\E_{\delta_x}\left[\sum_{u\in V_t} f(X_t^u)\right]= xe^{\lambda t}\E_x\left[ f(\mathcal{Y}_t)V_{1}^{-1}(\mathcal{Y}_t)\right],
\end{align*}
where $\mathcal{Y}_t$ is a Markov process with infinitesimal generator $\mathcal{A}_{V_1}=M_1 + J_1$, 
where for all $x\geq 0$, $f\in\mathcal{C}^2_b(\mathbb{R}_+)$, $J_1f(x) = 2r(x)\left[\int_0^1\theta \left(f(\theta x)-f(x)\right) \kappa(d\theta)\right]$, and
\begin{align*}
M_1f(x)  = &\ \left(\mathcal{G}(f\times V_1)(x)-f(x)\mathcal{G}V_1(x)\right)/{V_1(x)}\\
= & \ \left(g(x) + 2\frac{\sigma^2(x)}{x} +\frac{p(x)}{x}\int_{\mathbb{R}_+}z^2\pi(dz)\right)f'(x) + \sigma^2(x)f''(x)\\
&+ p(x)\int_{\mathbb{R}_+}\left(f(x+z)-f(x)-zf'(x)\right)\frac{(x+z)}{x}\pi(dz).
\end{align*}}
\section{Extension of results of \cite{companion}}
\subsection{Proof of \cite[Theorem 3.3]{companion} with stable jumps.}\label{app:th33-with-stable-jumps}
In this section we give the proof of the generalization of \cite[Theorem 3.3i)]{companion} to processes with stable jumps defined as the unique solution (see \cite[Proposition 1]{palau2018branching}) to \eqref{SDE_Y_diff_ct}.
Recall that $G_a$ is defined in \eqref{eq:Ga}, and $\tau^+, \tau^-$ in \eqref{tau2}, \eqref{eq:coupling} and \eqref{tau1}. 
\begin{lemma}\cite[Lemma 7.1]{companion}\label{prop:martingale-co}
Suppose that Assumption \ref{ass_A} holds. For all $b>c>0$, let $T = \tau^-(c)\wedge \tau^+(b)$. Then, for all $a\in(0,1)$ or $a>1$ such that $\E[\Theta^{1-a}]<\infty$, the process 
$$
Z^{(a)}_{t\wedge T}:=\left(Y_{t\wedge T}\right)^{1-a}\exp\left(\int_0^{t\wedge T} G_a\left(Y_s\right)ds\right)
$$ 
is a $\mathcal{F}_t$-martingale, where $Y$ is the unique solution of \eqref{SDE_Y_diff_ct}.
\end{lemma}
\begin{proof}[Proof of Lemma \ref{prop:martingale-co}]
First notice that using Taylor's formula, we have
\begin{align}\label{eq:equivca}
C_a=(2+\mathfrak{b})^{-1}\int_{\mathbb{R}_+}z(1+z)^{-a}\rho(dz)=(1-a)^{-1}\int_{\mathbb{R}_+}\left((1+z)^{(1-a)}-1\right)\rho(dz),
\end{align}
\begin{align}\label{eq:equivIa}
\text{and} \quad (1-a)I_a(x) = \int_{\mathbb{R}_+}\left((1+zx^{-1})^{(1-a)}-1-(1-a)zx^{-1}\right)\pi(dz).
\end{align}
We follow the proof of \cite[Lemma 7.1]{companion}. Let $a\in(0,1)\cup\{a>1,\ \E[\Theta^{1-a}]<\infty\}$. 
Combining It\^o's formula with jumps, \eqref{eq:equivca}, \eqref{eq:equivIa}, we have for all $t\geq 0$
\begin{align*}
Y_t^{1-a}   =  Y_0^{1-a}-\int_0^t Y_s^{1-a}G_a(X_s)ds+ M_t,
\end{align*}
where $\left(M_t, t\geq 0\right)$ is a local martingale.
Next, using integration by parts we obtain that $(Z^{(a)}_{t\wedge T}, t\geq 0 )$ is a local martingale. Similarly to \cite{li2017general}, combining Assumptions \ref{ass_A} and \cite[Theorem 51 p.38]{protter2005stochastic} ends the proof.
\end{proof}
\begin{prop}
Suppose that Assumption \ref{ass_A} holds and let $Y$ be the pathwise unique solution to \eqref{SDE_Y_diff_ct}. If \ref{A1} holds
then $\mathbb{P}_x\left(\tau^-(0)<\infty\right)=0$ for all $x>0$. 
\end{prop}
\begin{proof}(\cite[Theorem 3.3i)]{companion})
Let $n \geq 2$ and let $0<\varepsilon<b<1$ and $a>1$ such that $\E[\Theta^{1-a}]<\infty$ be such that  \ref{A1} holds for all $u\leq b$. Let $T_n = \tau^-(\varepsilon^n)\wedge \tau^+(b)$. As in \cite[Theorem 3.3i)]{companion}, combining Lemma \ref{prop:martingale-co} and \ref{A1},
we have that for all $0<\varepsilon<b$,
\begin{align}\label{eq:dichotomie}
\mathbb{P}_{\varepsilon}\left(\tau^-(0) =\infty\text{ or }\tau^+(b)<\tau^-(0)\right)=1.
\end{align}
We use a coupling to show that $\P_\eps(\tau^-(0)<\infty)=0$. 
Let for $N \in \N$,
$ r_{[0,N]}:= \sup_{0 \leq x \leq N} 2r(x), $
which is finite as $r$ is a continuous function.
Let
$\tY$ be the unique strong solution to
\begin{align*}\label{eq:coupling}\nonumber
\tY_t  = & \tY_0 + \int_0^t g(\tY_s) ds+ \int_0^t \sqrt{2 \sigma^2(\tY_s) }dB_s 
+ \int_0^t \int_0^{p(\tY_{s^-})}\int_{\mathbb{R}_+}z\widetilde{Q}(ds,dx,dz)
\\ &\int_0^t \int_0^{\tY_{s^-}} \int_{\R_+}  zR(ds,dx,dz)+ \int_0^t \int_0^{r_{[0,N]}} \int_0^1  (\theta-1)\tY_{s^-}N(ds,dx,d\theta),
\end{align*}
where the Brownian motion $B$ and the Poisson random measures $\widetilde{Q}$ and $N$ are the same as in \eqref{SDE_Y_diff_ct}. 
We will use four properties of this equation. 
\begin{itemize}
 \item[a)] It has a unique strong solution (see \cite[Proposition 1]{palau2018branching} and Assumptions \ref{ass_A}).
 \item[b)] If $\tY^{(1)}$ and $\tY^{(2)}$ are two solutions with $\tY^{(1)}_0 \leq \tY^{(2)}_0$, then
 $\tY^{(1)}_t \leq \tY^{(2)}_t$ for any $t\geq 0$.
 \item[c)] If $\tY$ is a solution with $\tY_0= Y_0$, then $\tY_t\leq Y_t$ for any $t$ smaller than $\tau^-(0) \wedge \tau^+(N) $.
 \item[d)] Equation \eqref{eq:dichotomie} holds for both $Y$ and $\tY$.
\end{itemize}
Following the proof of \cite[Theorem 3.3i)]{companion}, we obtain that
$
 \P_\eps\left( \ttau^-(0)<\infty \right)=0,
$
where the $\ttau$'s are defined as the $\tau$'s in \eqref{tau1} and \eqref{tau2} but for the process $\widetilde{Y}$.
Using the coupling described in point c), it implies that
$ \P_\eps\left( \tau^+(N) \leq \tau^-(0) \right)=1,
$
and letting $N$ tend to infinity, we get
$\P_\eps\left( \tau^-(0) =\infty \right)=1.$
\end{proof}
\subsection{Proof of \cite[Theorem 4.1]{companion} with stable jumps}\label{app:th41-with-stable-jumps} 
We consider again $Y$ being a solution \eqref{SDE_Y_diff_ct} and define 
$
\tau^+(\infty) = \lim_{n\rightarrow +\infty}\tau^+(n),
$
where $\tau^+(n)=\inf\{t\geq 0,\ Y_t\geq n\}$.
\begin{prop} \label{prop_B3}
Suppose that Assumption \ref{ass_A} holds and let $Y$ be the pathwise unique solution to \eqref{SDE_Y_diff_ct}. If \ref{eq:SNinfinity-1} holds,
then $\mathbb{P}_x\left(\tau^+(\infty)<\infty\right)=0$ for all $x>0$. 
\end{prop}
\begin{proof}(\cite[Theorem 4.1i))]{companion})
Similarly as in the proof of \cite[Theorem 4.1i))]{companion}, for $b^{-1}$ small enough and $\eps$ satisfying $0 < b < \eps^{-1}$, we have for any $\eps^{-1} \leq y \leq 2 \eps^{-1}$, 
\begin{equation}\label{two_cases}  
\P_{y}(\tau^+(\infty)=\infty \text{ or } \tau^-(b) < \tau^+(\infty)<\infty )=1. \end{equation}
Now, we have to take into account that $Y$ has two different types of positive jumps, from the PPMs $\widetilde{Q}$ or $R$. Let $$J_\pi(\eps) = \{\text{A jump associated to } \widetilde{Q} \text{ occurs at }\tau^+(\eps^{-1})\},$$
$$J_\rho(\eps) = \{\text{A jump associated to } R \text{ occurs at }\tau^+(\eps^{-1})\}.$$
Let us fix $\lambda>0$ and introduce the following real number:
$$ \mathcal{A}(\eps) := \sup_{\eps^{-1} \leq y \leq 2 \eps^{-1}}
 \E_y \left[ e^{-\lambda \tau^+(\infty)}; \tau^+(\infty)<\infty \right] . $$
Let $\Delta X_{\tau^+(\eps^{-1})} = X_{\tau^+(\eps^{-1})}-X_{\tau^+(\eps^{-1})-}$. For any $\eps < 1$, $y \leq \eps^{-1}$, we have
\begin{align*}\P_y(X_{\tau^+(\eps^{-1})}>2\eps^{-1}) \leq 
\P_y(\Delta X_{\tau^+(\eps^{-1})}>\eps^{-1}) = A_\pi + A_\rho,
\end{align*}
where
\begin{align*}
A_\pi :=\P_y(J_\pi(\eps),\ \Delta X_{\tau^+(\eps^{-1})}>\eps^{-1}) &= \P_y\left(J_\pi(\eps),\ \left(\Delta X_{\tau^+(\eps^{-1})}\right)^2\wedge \Delta X_{\tau^+(\eps^{-1})}>\eps^{-1} \right) \\
& \leq \eps \int_0^\infty (z\wedge z^2)\pi(dz),
\end{align*}
and for $\alpha>0$ such that $\alpha< 1 + \mathfrak{b}<1$,
\begin{align*}
A_\rho := \P_y(J_\rho(\eps),\ \Delta X_{\tau^+(\eps^{-1})}>\eps^{-1}) &= \P_y\left(J_\rho(\eps),\ \left(\Delta X_{\tau^+(\eps^{-1})}\right)^{\alpha}\wedge \Delta X_{\tau^+(\eps^{-1})}>\eps^{-\alpha} \right)\\
&  \leq \eps^{\alpha} \int_0^\infty (z^{\alpha}\wedge z)\rho(dz)
\end{align*}
where we used the Markov inequality.

Using Equation \eqref{two_cases} and the strong Markov property, we get, for any $\eps^{-1} \leq y \leq 2 \eps^{-1}$:
\begin{align*}
 \E_y \left[ e^{-\lambda \tau^+(\infty)}; \tau^+(\infty)<\infty \right]&=
 \E_y \left[ e^{-\lambda \tau^+(\infty)}; \tau^-(b) <\tau^+(\infty)<\infty \right]\\
 & \leq \E_y \left[ e^{-\lambda \tau^-(b)} \E_{X_{\tau^-(b)}}\left[e^{-\lambda \tau^+(\infty)}; \tau^+(\infty)<\infty\right]; \tau^-(b) <\infty\right].
\end{align*}
Using again the strong Markov property, we get for all 
$x\leq b\leq \eps^{-1}$
\begin{align*}
\E_{x}\left[e^{-\lambda \tau^+(\infty)}; \tau^+(\infty)<\infty\right] & = \E_y \left[ e^{-\lambda \tau^+(\infty)}; \tau^-(b) <\tau^+(\infty)<\infty \right]\\
&=\E_y \left[ e^{-\lambda \tau^+(\infty)}; \tau^-(b) <\tau^+(\eps^{-1})<\tau^+(\infty)<\infty \right]\\
&= \E_{x} \left[ e^{-\lambda \tau^+(\eps^{-1})} \E_{X_{\tau^+(\eps^{-1})}}\left[e^{-\lambda \tau^+(\infty)}; \tau^+(\infty)<\infty\right]; \tau^+(\eps^{-1}) <\infty\right]\\
& \leq \mathcal{A}(\eps)+\eps^{\alpha}\underbrace{\left(\eps^{1-\alpha} \int_0^\infty (z\wedge z^2)\pi(dz) +  \int_0^\infty (z^{\alpha}\wedge z)\rho(dz)\right)}_{C_{\pi,\rho}},
\end{align*}
where the last inequality is obtained by considering the event $\lbrace \eps^{-1}\leq X_{\tau^+(\eps^{-1})}\leq 2\eps^{-1}\rbrace$ and its complement.
Finally, combining the last two inequalities, we obtain
\begin{align*}
 \E_y \left[ e^{-\lambda \tau^+(\infty)}; \tau^+(\infty)<\infty \right] & \leq \E_y \left[ e^{-\lambda \tau^-(b)}; \tau^-(b) <\infty \right] \left(\mathcal{A}(\eps)+\eps^{\alpha} C_{\pi,\rho}\right).
\end{align*}
But there exists $C(b)<1$ such that for $2b \leq y$, 
$$ \E_y \left[ e^{-\lambda \tau^-(b)}; \tau^-(b) <\infty \right]<C(b). $$
Otherwise we would have
$$ \lim_{y \to \infty} \E_y \left[ e^{-\lambda \tau^-(b)}; \tau^-(b) <\infty \right] =1, $$
and thus $\tau^-(b)$ would converge to $0$ when the initial condition of the process goes to $\infty$ which would contradict our assumptions on the regularity
of the negative jumps. Hence, as for $\eps$ small enough, $2b \leq \eps^{-1}$, we obtain for such an $\eps$ that
$ \mathcal{A}(\eps) \leq \frac{C(b)C_{\pi,\rho}\eps^{\alpha}}{1-C(b)}. $
We thus deduce that
$$ \lim_{y \to \infty}\E_y \left[ e^{-\lambda \tau^+(\infty)}; \tau^+(\infty)<\infty \right]=0. $$
Now, let us take $x,\mu>0$. Then, there exists $N_0$ such that for any $N \geq N_0$,
$$ \E_N \left[ e^{-\lambda \tau^+(\infty)}; \tau^+(\infty)<\infty \right]\leq \mu. $$
Hence,
$$ \E_x \left[ e^{-\lambda \tau^+(\infty)}; \tau^+(\infty)<\infty \right]
\leq \E_x \left[ \E_{X_{\tau^+(N_0)}} \left[ e^{-\lambda \tau^+(\infty)}; \tau^+(\infty)<\infty\right] \right]\leq \mu $$
and thus for all $x>0$
$$ \E_x \left[ e^{-\lambda \tau^+(\infty)}; \tau^+(\infty)<\infty \right]=0,$$
which completes the proof.
\end{proof}
\subsection{Proof of \cite[Theorem 3.3]{companion} with time dependent positive jump rate}\label{app:th33withtimeinjumps}
Let $t>0$ and consider the process $Y$ solution to
\begin{align}\label{eq:EDS-pt}
Y_t= Y_0+ \int_0^t g(Y_s) ds+ \int_0^t\sqrt{2 \sigma^2(Y_s) }dB_s&+ \int_0^t\int_{\mathbb{R}_+}\int_0^{p_t(Y_{s^-},s,z)}z\widetilde{Q}(ds,dz,dx)\\
&+ \int_0^t  \int_0^{r(Y_{s^-})} \int_0^1 (\theta-1)Y_{s^-}N(ds,dx,d\theta), \nonumber
\end{align}
where $\widetilde{Q}$, $B$ are the same as in \eqref{X_sans_sauts}, $N$ is a PPM on $\R_+\times \R_+\times [0,1]$ with intensity $ds\otimes dx\otimes \kappa(d\theta)$, and for $x,z \geq 0$ and $s \leq t$, we assume that $ p_t(\cdot,s,z)$ is non-decreasing on $\mathbb{R}_+$ and that there exists a finite and positive constant $\mathfrak{c}_t$  depending on $t$ such that
\begin{equation} \label{bound_f3} p_t(x,s,z)=p(x)(1+\mathfrak{p}(x,s,z)) \leq p(x) (1+ \mathfrak{c}_t z). \end{equation}
 For the existence of a unique pathwise solution to \eqref{eq:EDS-pt}, we refer to Appendix \ref{app:prop_sol_SDE_auxi}.
\begin{thm}\cite[Theorem 3.3i) \al{ and iii)}]{companion}
Suppose that \al{$\int_{\R_+} z\pi(dz)<\infty$} and that \ref{ass_E} holds. Let $Y$ be a pathwise unique solution to \eqref{eq:EDS-pt}. Then, 
\begin{enumerate}[label=\roman*)]
\item If there exist $a>1$ and a non-negative function $f$ on $\mathbb{R}_+$ with $\E\left[\Theta^{1-a}\right]<\infty$ and
\begin{align}\label{eq:SN0-zpifini}
\frac{g(x)}{x}-a\frac{\sigma^2(x)}{x^2} = f(x) + o(\ln(x)),\ (x\rightarrow 0),
\end{align}
then $\mathbb{P}_x(\tau^-(0)<\infty)=0$ for all $x>0$. 
\item \al{If \ref{A2} holds and $r(x)>0$ for all $x\geq 0$, then for any $x>0$, $\P_x(\tau^-(0)<\infty)>0$. }
\end{enumerate} 
\end{thm}
\begin{proof}\cite[Theorem 3.3i) \al{and iii)}]{companion}
\begin{enumerate}[label=\roman*)]
\item Let us define for $a>1$
\begin{align}\nonumber\label{eq:Ga_s}
G_a^{(s)}(x) : = & (a-1)\frac{g(x)}{x}-a(a-1)\frac{\sigma^2(x)}{x^2}-r(x)\E[\Theta^{1-a}-1]\\
& -\int_{\mathbb{R}_+}p_t(x,s,z) \left((zx^{-1}+1)^{1-a}-1-(1-a)zx^{-1}\right)\pi(dz).
\end{align}
First, applying Itô's formula with jumps, we can check that Lemma \ref{prop:martingale-co} (\cite[Lemma 7.1 and Equation (7.1)]{companion}) holds for $Y$, replacing $G_a$ by $G_a^{(s)}$. Next, 
Using \eqref{bound_f3} we can show (see the proof of \cite[Remark 3.2]{companion}) that under \eqref{cond_moment2_pi}
\begin{align*}
\limsup_{x\rightarrow 0^+} \left( x^{-2}\int_0^\infty p_t(x,s,z) z^2\left(\int_0^1 (1+zx^{-1}v)^{-1-a}(1-v)dv\right)\pi(dz) \right)\\
\leq C_t \left(\limsup_{x\rightarrow 0^+}p(x)x^{-1}\right)\left(\int_0^\infty (z+z^2)\pi(dz)\right)< \infty
\end{align*}
where $C_t$ is a finite constant. Hence, the integral corresponding to the positive jumps is bounded in the neighborhood of $0$, and Condition \eqref{eq:SN0-zpifini} is enough to control the behaviour of the process around $0$. 
The proof of \cite[Theorem 3.3i)]{companion} is thus unchanged and the results hold also for processes whose rate of positive jumps satisfy \eqref{bound_f3}.
\item \al{First, we have to prove that \cite[Theorem 3.3ii)]{companion} still holds for processes whose rate of positive jumps satisfy \eqref{bound_f3}. Using that for $a<1$ and for all $x,s\geq 0$, $G_a^{(s)}(x)\geq G_a(x)$, the proof of \cite{companion} is unchanged. Next, because of the product form of the jump function, the positive jumps of $Y$ can be seen as occurring at rate $p(x)$, with a size given by the measure $(1+\mathfrak{p}(x,s,z))\pi(dz)$. Therefore, as the proof of \cite[Theorem 3.3iii)]{companion} relies only on computation on the jump rate, and not on the size of the jumps, the proof is unchanged. }
\end{enumerate}
\end{proof}
\subsection{Proof of \cite[Theorem 4.1.i)]{companion} with time dependent positive jump rate}\label{app:th41withtimeinjumps}
Let $t>0$ and consider again the process $Y$ solution to \eqref{eq:EDS-pt} under condition \eqref{bound_f3}.
\begin{thm}\cite[Theorem 4.1.i)]{companion}
Suppose that Assumption \ref{ass_E1} holds and that $\int  z^3\pi(dz)<\infty$ and let $Y$ be a pathwise unique solution to \eqref{eq:EDS-pt}. If \ref{eq:SNinfinity-1} holds,
then $\mathbb{P}_x(\tau^+(\infty)<\infty)=0$ for all $x>0$. 
\end{thm}
\begin{proof}\cite[Theorem 4.1.i)]{companion}
Let us recall the definition of $G_a^{(s)}$ in \eqref{eq:Ga_s}.
We have using Taylor's formula and \eqref{bound_f3} that
\begin{align*}
&\limsup_{x\rightarrow\infty}\int_{\mathbb{R}_+}p_t(x,s,z) \left((zx^{-1}+1)^{1-a}-1-(1-a)zx^{-1}\right)\pi(dz)\\
=&a(1-a)\limsup_{x\rightarrow\infty}x^{-2}\int_{\mathbb{R}_+}p_t(x,s,z) z^2\left(\int_0^1(1+zx^{-1}v)^{-1-a}(1-v)dv\right)\pi(dz)\\
&\leq C_t\limsup_{x\rightarrow\infty}\frac{p(x)}{x^2}\int_{\mathbb{R}_+}(z^2+z^3) \pi(dz)<\infty
\end{align*}
for some $C_t>0$, combining \ref{ass_E}, \eqref{eq:limsup_sigma} and the fact that $\int_{\mathbb{R}_+}(z^2\vee z^3)\pi(dz)<\infty$. Therefore, the integral corresponding to the positive jumps is bounded for large values of $x$, and Condition \ref{eq:SNinfinity-1} is enough to control the behaviour of the process near infinity. 
The proof of \cite[Theorem 4.1.i)]{companion} is thus unchanged and the results hold also for processes whose rate of positive jumps satisfies \eqref{bound_f3}.
\end{proof}

\subsection{Proof of \cite[Theorem 6.2]{companion} with time dependent positive jump rate.}\label{app:th62withtimeinjumps}

Let $t>0$ and consider again the process $Y$ solution to \eqref{eq:EDS-pt}
where for $x>0$, $z \geq 0$ and $s \leq t$,
$
 p_t(x,s,z) \leq p(x) (1+ \mathfrak{c}_t z)$ and $p_t(x,s,z) \leq p(x) (1+ zx^{-1})
$
with $\mathfrak{c}_t$ a finite and positive constant depending on $t$.
\begin{thm}\cite[Theorem 6.2.ii)]{companion}\label{th:20-6.2}
Suppose that Assumption \ref{ass_E1} and \ref{eq:SNinfinity-1} hold and that $\int z^3\pi(dz)<\infty$ and let $Y$ be a pathwise unique solution to \eqref{eq:EDS-pt}. Assume also that \ref{eq:LSG} holds (with $g,r$ instead of $\hat{g},\hat{r}$).
\begin{itemize}
\item[(i)] If \ref{A1} holds, then $Y$ converges in law as $t$ tends to infinity to $Y_\infty$ and for every bounded and measurable function $f$, almost surely
$$
\lim_{t\rightarrow+\infty}\frac{1}{t}\int_0^t f(X_s)ds=\E\left[f(X_\infty)\right].
$$
\item[(ii)] If \ref{A2} holds, then $\P_x\left(\exists t<\infty,Y_t=0\right)=1$.
\end{itemize}
\end{thm}
We first prove a lemma, as in \cite{companion}. Let $x_0,t_0>0$. For all $i\geq 0$, we consider the stopping times $T_i(x_0)$, given by $T_0=0$ and for all $i\geq 1$,
$
T_i(x_0) = \inf\lbrace t\geq T_{i-1}(x_0) + t_0, Y_t\leq x_0\rbrace.
$
\begin{lemma}\cite[Lemma 7.2.]{companion}\label{lemma:Ti}
Suppose that Assumption \ref{ass_E1} holds and that $\int  z^3\pi(dz)<\infty$. Then, if \ref{eq:SNinfinity-1}  and \ref{eq:LSG} hold,
then $\E\left[T_i(x_0)\right]<\infty$ for all $i\geq 0$.
\end{lemma}
\begin{proof}(\cite[Lemma 7.2]{companion})
Let us choose $0<x_0<x_1$ and let $\tau = \tau^-(x_0)\wedge\tau^+(x_1)$ (recall \eqref{eq:coupling} and \eqref{tau1}).
Following the same steps as in the proof of \cite[Lemma 7.2]{companion} we have,
\begin{align*}
\ln Y_{t\wedge \tau} \leq &  \ln Y_{0}+\int_{0}^{t\wedge \tau}\frac{g(Y_s)}{Y_s}ds-\int_{0}^{t\wedge \tau}
\frac{\sigma^2(Y_s)}{Y_s^2}ds
+\E[\ln \Theta]\int_{0}^{t\wedge \tau}r(Y_s)ds +M_{t\wedge \tau}
\end{align*}
where $(M_{s \wedge \tau}, s \geq 0)$ is a martingale. According to \ref{eq:LSG}, we have
$
\ln(Y_{t\wedge \tau})-\ln(Y_0)\leq -\eta(t\wedge\tau) + M_{t\wedge \tau},
$ and following the proof of \cite[Lemma 7.2]{companion}, we obtain
\begin{align*}
\E_x(\tau)\leq \frac{1}{\eta}\ln\left(\frac{x}{\Theta x_0}\right).
\end{align*}
For the end of the proof, we follow exactly \cite[Lemma 7.2]{companion}, using the generalization of \cite[Theorem 4.1.i)]{companion} proved in Appendix \ref{app:th41withtimeinjumps}.
\end{proof}
\begin{proof}[Proof of Theorem \ref{th:20-6.2}]\cite[Theorem 6.2.]{companion}
The first point follows directly from Lemma \ref{lemma:Ti} and \cite[Theorem 7.1.4]{BN} as in \cite{companion}. 
For the second point, let $a<1$ be such that Condition \ref{A2} is satisfied, and let $\delta<(3-2a)^{-1}$. Following the proof of \cite[Theorem 3.3.ii)]{companion}, replacing $G_a$ by $G_a^{(s)}$ defined in \eqref{eq:Ga_s}, and using that for $a<1$, under \ref{A2}, 
$$G_a^{(s)}(x)\geq (a-1)\frac{g(x)}{x}-a(a-1)\frac{\sigma^2(x)}{x^2} \geq (1-a)\ln(x^{-1})\left(\ln\left(\ln(x^{-1})\right)\right)^{1+\eta_2},$$
we get that there exist $\mathfrak{t}$ and $\mathfrak{p}$ such that for all $\varepsilon$ small enough, and $z \in (\varepsilon^{1+\delta},\varepsilon^{1-\delta})$, 
$$
\P_z\left(\tau^-(0)\leq \mathfrak{t}(\varepsilon)\right)\geq \mathfrak{p}(\varepsilon).
$$
From this result, we can prove \cite[Eq.~(7.29)]{companion}. Next, the proof of \cite[Eq.~(6.3)]{companion} requires \cite[Eq.(7.18)]{companion}. To prove \cite[Eq.(7.18)]{companion} in our case, we have to deal with the dependence on the jump size of the jump rate $p_t$ to obtain a lower bound on the probability to have no positive jump during a time interval of the form $[0,t\wedge\tau^+(y)\wedge \tau^-(x)]$ with $0< x < y$.
Hence the idea is to bound the expectation of the sum of positive jumps on $[0,t\wedge\tau^+(y)\wedge \tau^-(x)]$ and use Markov inequality. Let $T= \tau^+(y)\wedge \tau^-(x)$. Then, for any $y_0\in (x,y)$,
\begin{align*}
\P_{y_0}\left(Y_{t\wedge T}\geq y\right)\leq \P_{y_0}&\left(\int_0^{t\wedge T} g(Y_s)ds +I_Q(t\wedge T)+\int_0^{t\wedge T} \sqrt{2\sigma^2(Y_s)}dB_s\geq y-y_0\right),
\end{align*} 
where
\begin{align}\label{eq:twojumpint}\nonumber
I_Q(t\wedge T)&:=\int_0^{t\wedge T}\int_{\R_+}\int_0^{p_t(Y_s,s,z)}zQ(ds,dz,dx)\\\nonumber
&\leq \int_0^{t\wedge T}\int_{0}^{Y_s}\int_0^{2p(Y_s)}zQ(ds,dz,dx) + \int_0^{t\wedge T}\int_{Y_s}^\infty\int_0^{2zp(Y_s)/Y_s}zQ(ds,dz,dx)\\
&\leq \int_0^{t\wedge T}\int_{\R_+}\int_0^{2p(Y_s)}zQ(ds,dz,dx) + \int_0^{t\wedge T}\int_{\R_+}\int_0^{2zp(Y_s)/\hat{Y}_s}zQ(ds,dz,dx).
\end{align}
Then, as in \cite{companion} p.19, if we denote by $J(t,x,y)$ the event of having no
positive jumps due to the first integral in \eqref{eq:twojumpint}, we have for all $y_0\in (x,y)$,
$
\P_{y_0}(J(t,x,y))\geq e^{-2tp(y)},
$
because $p$ is non-decreasing according to Assumption \ref{ass_A}. Next,
\begin{align*}
& \P_{y_0}\left(Y_{t\wedge T}\geq y, J(t,x,y)\right)\\
& \leq \P_{y_0}\left(\int_0^{t\wedge T} g(Y_s)ds +\int_0^{t\wedge T}\int_0^{2p(Y_s)/Y_s}\int_{\R_+}z\overline{Q}(ds,dx,dz)+\int_0^{t\wedge T} \sqrt{2\sigma^2(Y_s)}dB_s\geq y-y_0\right),
\end{align*}
where $\overline{Q}$ is a PPM with intensity $ds\otimes dx\otimes z\pi(dz)$ (see Appendix \ref{app:generalization SDE}). Considering as before the event that there is no positive jumps associated to $\overline{Q}$, we get
\begin{align*}
\P_{y_0}\left(Y_{t\wedge T}\geq y\right)\leq 2-e^{-2tp(y)}-e^{-2tp(y)/y}+ \P_{y_0}\left(\int_0^{t\wedge T} g(Y_s)ds +\int_0^{t\wedge T} \sqrt{2\sigma^2(Y_s)}dB_s\geq y-y_0\right).
\end{align*}
We conclude as in \cite{companion}, that $\sup_{x \leq v \leq z}g(v)<\infty$ according to Assumption \ref{ass_A} and that $\inf_{x \leq v \leq y}r(v)>0$ holds under Assumption \ref{ass_E}.
\end{proof}

\section{Proof of Proposition \ref{prop_sol_SDE_auxi}}\label{app:prop_sol_SDE_auxi}

The proof is a direct application of \cite[Proposition 1]{palau2018branching}.
Notice that in the statement of \cite[Proposition 1]{palau2018branching}, the functions $b$, $g$ and $h$ do not depend on time, 
unlike the present case of our process. However this additional dependence does not bring any modification to the proofs (which are mostly  derived in the earlier 
paper \cite{li2012strong}).
First according to their conditions (i) to (iv) on page 60, our parameters are admissible.
Second, we need to check that conditions (a), (b) and (c) are fulfilled. It follows directly from Assumption \ref{ass_A}.

\section{Generalization to a division rate depending on the fragmentation parameter $\theta$}\label{app:generalization SDE}

In some proofs, we need to consider a slight generalization 
of the SDE \eqref{SDE_Y_diff_ct} where an individual with trait $x$ dies 
and transmits a proportion $\theta \in [0,1]$ of its trait to its left offspring
at a rate $r(x)\rtheta(\theta)$, that depends on $\theta$,
where $\rtheta:[0,1]\rightarrow \mathbb{R}_+$ is a nonnegative function. 
However, using the properties of Poisson random measures we can prove that 
a solution to such an SDE can be rewritten 
as the solution to \eqref{SDE_Y_diff_ct} by modifying the death 
rate $r$ and the fragmentation kernel $\kappa$.
\begin{lemma}\label{chgt-eds}
Assume that $\int_0^1 \rtheta(\theta)\kappa(d\theta)<\infty$. Let 
$$\hat{\kappa}(d\theta) = \rtheta(\theta)\left(\int_0^1\rtheta(\theta)\kappa(d\theta)\right)^{-1}\kappa(d\theta),\quad \hat{r}(x) = r(x)\int_0^1\rtheta(\theta)\kappa(d\theta),$$
and $\widetilde{Q}$, $B$ and $N$ be defined as in \eqref{SDE_Y_diff_ct}.
Then, there exists a Poisson random measure $N'$ with intensity $ds\otimes dz\otimes  \hat{\kappa}(d\theta)$ such that $X$ is the pathwise unique solution to 
\begin{align*}
X_t= X_0+ \int_0^t g(X_s)ds+ \int_0^t\sqrt{2 \sigma^2(X_s) }dB_s&+ \int_0^t\int_0^{p(X_{s^-})}\int_{\mathbb{R}_+}z\widetilde{Q}(ds,dx,dz)\\
&+ \int_0^t  \int_0^{\hat{r}(X_{s^-})} \int_0^1 (\theta-1)X_{s^-}N'(ds,dz,d\theta), \nonumber
\end{align*} 
if and only if $X$ is the pathwise unique nonnegative strong solution to
\begin{align*}
X_t= X_0+ \int_0^t g(X_s) ds+ \int_0^t\sqrt{2 \sigma^2(X_s) }dB_s&+ \int_0^t\int_0^{p(X_{s^-})}\int_{\mathbb{R}_+}z\widetilde{Q}(ds,dx,dz)\\
&+ \int_0^t \int_0^1\int_0^{r(X_{s^-})\rtheta(\theta)}(\theta-1)X_{s^-}N(ds,dz,d\theta). \nonumber
\end{align*}
\end{lemma}

\section*{Acknowledgments}
The authors are grateful to V. Bansaye for his advice and comments and to B. Cloez for fruitful discussions.
This work was partially funded by the Chair `Mod\'elisation Math\'ematique et Biodiversit\'e' of VEOLIA-Ecole Polytechnique-MNHN-F.X., by the French national research agency (ANR) via project ANR NOLO (ANR-20-CE40-0015) and in the framework of the `France 2030' program (ANR-15-IDEX-0002) and by the LabEx PERSYVAL-Lab (ANR-11-LABX-0025-01).

\bibliographystyle{abbrv}
\bibliography{biblio}

\end{document}